\documentclass[11pt]{article}

\usepackage{fullpage}
\usepackage{setspace}
\usepackage{graphicx}
\usepackage{subcaption}
\usepackage{amssymb,amsfonts,amsmath,amsthm,latexsym,epsfig,bm,bbm,mathtools}
\usepackage{xspace,times,color,todonotes}
\usepackage{enumitem}
\usepackage{macros}
\usepackage{thmtools}
\usepackage[dvipsnames]{xcolor}
\usepackage[unicode, colorlinks, linkcolor=Black]{hyperref}

\usepackage{authblk}
\usepackage{soul}

\title{Travel Bans vs. Other Disease Mitigation Measures: A Mathematical Analysis}
\author[1, 3]{Christian Borgs}
\author[2]{Karissa Huang}
\author[1]{Geng Zhao}
\date{}

\affil[1]{Department of Electrical Engineering and Computer Sciences, UC Berkeley}
\affil[2]{Department of Statistics, UC Berkeley}
\affil[3]{Bakar Institute of Digital Materials for the Planet}

\begin{document}

\maketitle

\begin{abstract}
As the world grows increasingly connected, infectious disease transmission and outbreaks have become a pressing global concern for public health officials and policymakers. While policy interventions to contain and prevent the spread of disease have been proposed and implemented, there has been little rigorous quantitative analysis of the effectiveness of such interventions. In this paper, we study the susceptible-infected-recovered (SIR) infection process on a dynamic network model that models two communities with travel between them with the infection starting in one of them. In particular, we consider two Erd\H{o}s--R\'enyi graphs where edges are dynamically changing based on node travel between the graphs. We characterize the time evolution of the outbreaks in both communities and pin down the time for when the infection first reaches the second community. Finally, we analyze two types of interventions--travel bans and intra-community interventions in the second community--and prove that travel bans are not effective, while the second type are effective even without travel bans, provided they sufficiently reduce the effective reproduction number. We complement our analytic results by numerical simulations on large networks with  realistic degree distributions and disease recovery times, showing that these results are robust, and hold for settings that model actual contact networks and disease spread more closely.
\end{abstract}


\maketitle

\section{Introduction}

The rapid spread of infectious diseases has become a global concern, drawing heightened public attention over the last few years in light of the COVID-19 pandemic, and the monkeypox and measles outbreaks. In response to such outbreaks, various interventions, including social distancing and travel bans, have been proposed and implemented to curb transmission. While there have been several empirical studies on the efficacy of social distancing and travel bans in response to various outbreaks, there has been little rigorous quantitative analysis to date. Our paper aims to put understanding of these interventions on a rigorous footing.

Empirically, the efficacy of travel bans as a standalone measure is dubious. Studies have shown that there is little evidence that travel restrictions can contain the spread of influenza, and at best they only delay the onset of the pandemic by a few days or weeks \cite{ferguson2006strategies, mateus2014effectiveness}. In \cite{yu2012transmission} the authors find that border entry screening was unlikely to have delayed the peak of the 2009 H1N1 pandemic in China by more than four days. More recently, in \cite{chinazzi2020effect} the authors show that travel restrictions during COVID-19 only delayed the epidemic (on the order of days) in mainland China, and that travel restrictions to and from mainland China only ``modestly" affected the epidemic trajectory unless combined with another intervention.

In contrast, intra-community interventions, including masks, vaccinations, case isolation, and social distancing, for reducing contact rates and transmission rates have long been used as tools to prevent disease transmission. During the 1918 influenza epidemic, cities that implemented such interventions earlier in the epidemic had lower transmission and mortality rates \cite{bootsma2007effect}. Since then, several studies have found that intra-community interventions are effective at reducing disease transmission and mortality rates, including during the COVID-19 pandemic \cite{ferguson2006strategies, kwon2021association, haug2020ranking}.

In this paper we study a simple model of an SIR epidemic on a dynamic network that models two communities with travel between them. Travel is modeled by two Poisson processes, with the rate of return much higher than the rate of travel, so that at any given time visitors only make up a small proportion in each community.  The epidemic spreads separately in each community (between all individuals currently there, whether they are residents or guests), and only spreads from one community to another when an infected individual travels before they recover (whether the individual is a guest infected while not at home, or an individual infected at home before traveling).  For simplicity, the contact network in each community is modeled by an Erd\H{o}s--R\'enyi random graph, even though our results should hold for other network models for the communities, like stochastic block models or configuration models, at least as long the models are such that $R_0<\infty$.

We start the infection in the first community and are interested in the trajectory of the epidemic in both communities as well as the efficacy of travel bans compared with intra-community interventions in our setting. In the real world, one may imagine that the two communities are separate geographical regions with flights going between them. We study the sparse regime, where the average degree of the vertices is bounded and where the number of travelers is polynomially small in the number of vertices in the graph. 

Our main results are two-fold. First, we characterize the trajectory of the epidemic, pinning down the times at which the infection reaches an $\epsilon$ fraction of vertices in the first and second communities, and the time at which the infection first passes from the first to the second community. In particular, we prove that the epidemic reaches the second community long before reaching an $\epsilon$ fraction in the first. Second, we show that intra-community interventions can reduce the final size of the infection, whereas travel bans have negligible impact on the final size of the infection giving a mathematical understanding of the observational results discussed above.

\paragraph{Related work.}
The SIR model is a canonical model of disease spread in closed, finite populations and was introduced by Kermack and McKendrick in their seminal 1927 paper \cite{kermack1927contribution}. The initial model, and its stochastic counterpart, do not take into account underlying network structure induced by contacts in a population. However, underlying heterogeneities in the network of a population are known to greatly influence the trajectory of an epidemic. As such, there has been a long line of work studying the SIR epidemic on random graphs with bounded average degree, mirroring real-world networks, including the Erd\H{o}s--R\'enyi random graph \cite{neal2003sir}, the configuration model \cite{Bohman2012, barbour2013approximating, Decreusefond2012, janson14:lln-sir}, graphs with local household structure and tunable clustering \cite{Ball2010, house2012modelling, Britton2008}, and the stochastic block model \cite{borgs2024lawlargenumberssir}. 

Another perspective on capturing heterogeneities in networks is the meta-population view; a heterogeneous population is divided into large, homogeneous sub-populations, between which the infection can spread at varying rates. Such models go back to at least 
\cite{hethcote1978immunization}, where the inter-group connections are assumed to be static.  A little later, Rvachev and Longini  \cite{rvachev1985mathematical} proposed a model that explicitly incorporated the effect of air travel and used it to analyze the 1968-69 ``Hong-Kong'' pandemic, using a deterministic, continuous variable model.  This work has been extremely influential, leading to a large body of  work modeling the influence of travel on various specific outbreaks, from influenza \cite{longini88,grais03,grais04} to AIDS \cite{flahault92} and SARS \cite{hufnagel04}.

There are several works studying the model introduced by Rvachev and Longini from a more theoretical point of view, see in particular the series of papers by Colizza et al. \cite{colizza2007modeling,colizza2007invasion,colizza2007reaction,colizza2008epidemic}. More closely related to our work is the paper of Barth\'elemy et al. \cite{barthelemy2010fluctuation}, in which the travel between communities is modeled by a stochastic, integer random variable described by a so-called Cox process. Based on a mean-field approximation for the dynamic within the communities and further approximations to study the Cox process describing travel, 
they determine the asymptotic behavior of the first time an infected individual arrives in a second community. Under the conditions considered in their paper, there is a non-zero probability that the epidemic never spreads to the second community, leading to a percolation type description for the probability of the epidemic jumping into other communities.

More similar to our work in spirit and in its aim for mathematical rigor is the  
the recent work of Ball et al. \cite{ball2024sir} who study the SIR epidemic in populations composed of large, weakly connected subcommunities.
Their setting resembles ours in that the description of the epidemic is fully stochastic, deriving laws of large numbers where needed instead of postulating them. However, it differs in that the mixing is fully homogeneous within each subpopulation and the connections between communities are much sparser, leading again to a percolation type description for the global spread, now rigorously derived from the underlying fully stochastic model. They characterize whether a global outbreak occurs, but do not examine temporal dynamics or control interventions. Furthermore, their model of inter-community connections is static, rather than based on an explicit model of travel.

By contrast, the rigorous recent works
\cite{lashari2018branching, britton2019preventive, durrett2019sir, ball2020epidemics, durrett2022evoSI, cocomello2023sirSEIR, huang2024sir} do consider dynamic networks, with \cite{lashari2018branching} modeling partnership dynamics, \cite{britton2019preventive, durrett2019sir, ball2020epidemics, durrett2022evoSI} studying SI and SIR dynamics on Erd\H{o}s--R\'enyi random graphs and configuration models where individuals drop or rewire connections to infectious neighbors,  \cite{cocomello2023sirSEIR} studying networks where the infection and recovery rates vary dynamically, and \cite{huang2024sir} studying the SIR epidemic on a dynamic inhomogneous Erd\H{o}s--R\'enyi random graph, where edges appear and disappear independently of each other.

From a technical point of view, our work differs from this line of work in that we consider both group structure and 
a dynamic network driven by the movement of individuals, thus making neither the approaches of \cite{britton2019preventive, durrett2019sir, ball2020epidemics, durrett2022evoSI} or \cite{lashari2018branching, cocomello2023sirSEIR, huang2024sir} directly applicable.  Furthermore, we are not only interested in the time evolution of the epidemic, but also the policy implications and efficacy of interventions.  

Our work also differs from previous empirical and theoretical work in that we incorporate ``memory'' into our travel model, allowing the concept of a home community and different rates for travelers leaving their community and for those returning, thus giving a more realistic model for travel.  Note that while we only present proofs for the simple setting of two communities each having an Erd\H{o}s--R\'enyi internal structure, it is clear that the model is well defined for more than two communities, as well as more realistic models for the internal community structure, like stochastic block models, degree corrected stochastic block models, or configuration models. Thus, we believe that in addition to our contribution to a rigorous understanding of the effectiveness of travel restrictions versus social distancing, our model should also advance more realistic modeling for empirical studies of the effect of travel on the spread of an epidemic.

\section{An SIR Epidemic Model on Two Communities with Travel}

Our aim is to understand how travel dynamics impact the  spread of an SIR epidemic between two communities connected through travel. We model these dynamics using a dynamic graph $G(t)$, consisting of two densely connected communities of individuals (nodes in the graph), which are sparsely connected through travel. Edges in the graph appear and disappear as individuals travel between the two communities. Individuals are assigned a label representing their home community in order to track whether they are currently traveling or home. Individuals travel infrequently to the other community, staying there for a duration comparable to the infectious period (e.g. days or weeks). Within each community, infection spreads between its residents and the small fraction of visitors, who may bring the infection back to their home community if they do not recover before they return. Thus, both travel to a guest community and home can facilitate spread between the two communities. For simplicity, we model the contact network within each community by an Erd\H{o}s--R\'enyi random graph although we believe our analyses would extend if other random graph models with more internal structure (e.g. stochastic block model or configuration model) were used within each community.

\paragraph{Contact Network} Let $G$ be a graph on $2n$ vertices with vertices partitioned into two intrinsic \textit{types}, $V_1$ and $V_2$, each containing $n$ vertices. A vertex's type represents its home location and does not change over time. In addition, to model a dynamic population with travel, there is a notion of \emph{community} representing the physical location of vertices. At any given time, each vertex resides in one of two physical communities, $Q_1$ or $Q_2$, and vertices can travel between communities. 
We denote the type of node $v$ by $\mathsf{type}(v)\in\{1,2\}$, and its physical community at time $t$ by $q_v(t)\{1,2\}$.
We say a node $v$ is at home at time $t$ if $\mathsf{type}(v) = q_v(t)$, and otherwise we say it is traveling. 
Let $V_1^\Home$ denote the set of type 1 vertices in community 1 (i.e., those currently at home), and let 
$V_1^\Travel$ denote the set of type 1 vertices in community 2 (i.e., those currently traveling). Similarly, 
define $V_2^\Home$ and $V_2^\Travel$ for type 2 vertices. 

Contacts are represented as edges between individuals. Within each community, edges are formed independently with equal probability $p := c/n$. When any individual travels for the first time, new edges are formed with individuals in the other community with probability $c'/n$ and these edges will remain. For simplicity of notation, we assume that $c = c'$.  Preflipping the random coins governing these edges, we equivalently may connect all vertices with probability $c/n$, whether they have the same label or not, but only activate the edges between vertices of different types when one of them travels.

\paragraph{Travel Dynamics} 
We model travel states $q_v(t)$ as independent continuous-time Markov processes governing movement between communities. 
For type 1 and type 2 nodes, the  transition matrix for travel is given by $\begin{pmatrix}
    -\rho_\Travel & \rho_\Travel \\
    \rho_\Home & -\rho_\Home
\end{pmatrix}$ and  $\begin{pmatrix}
    -\rho_\Home & \rho_\Home \\
    \rho_\Travel & -\rho_\Travel
\end{pmatrix}$, respectively. Standard results for Markov processes implies that in the steady state, an individual is at travel with probability $p_\Travel = \frac{\rho_\Travel}{\rho_\Travel + \rho_\Home}$ and at home with probability $p_\Home = 1 - p_\Travel$. Throughout the paper we will assume that $n^{-1}\rho_\Home \ll \rho_\Travel \ll \rho_\Home$ so that travel is rare, i.e., $p_\Travel \ll 1$, but frequent enough to give a substantial risk of transmission between the two communities, i.e., $n p_\Travel\gg 1$.
An edge between two individuals currently in the same (different) community is said to be active (inactive). We use $G(t)$ to denote the network of \emph{currently active} edges.

\paragraph{SIR Dynamics}
The SIR infection on $G(t)$ is a Markov Chain whose state space is given by the sets of susceptible, infected, and recovered vertices in each community at time $t$. The process starts with a single infected vertex $v_0\in V_1^\Home$ chosen uniformly at random, with all other vertices susceptible. At time $t$, an infected vertex $u$ can infect its neighbor $v$ only if the edge $(u, v)$ is active at time $t$. Infection events happen at rate $\beta$ and infected vertices recover at rate $\gamma$.

Let $S_k(t), I_k(t), R_k(t)$ denote the number of susceptible, infected, and recovered vertices in $V_k$ at time $t \ge 0$ for $k \in \{1, 2\}$. Initially there is one infected vertex, e.g. $I_1(0) = 1$. We assume that the initially infected vertex is not traveling. Since $S_k(t)$ is non-increasing in $t$ and thus must admit a limit, we define $S_k(\infty) = \lim_{t\to\infty} S_k(t)$; $R_k(\infty)$ can be defined similarly.

\paragraph{Key parameters of SIR Dynamics}
The basic reproduction number, denoted as $R_0$, is a fundamental parameter that quantifies the average number of secondary infections produced by a typical infectious individual in a fully susceptible population \cite{diekmann1990definition}.  In our model, $R_0$ is given by
\begin{equation}\label{eqn_R0_def}
    R_0 = \frac{c\beta }{ \beta + \gamma}
\end{equation}
in the limit as the rate of travel tends to zero.
As we will see later, consistent with the literature, the threshold $R_0 > 1$ coincides with the possibility of a major outbreak in the population. Furthermore, the rate of exponential growth (per unit time) of the initial outbreak is closely related to $R_0$ up to a constant factor determining the time scale. In particular, this exponential growth is governed by the rate parameter $\lambda$, defined as 
\begin{equation}\label{eqn_lambda_def}
    \lambda := c\beta - \beta - \gamma = (\beta + \gamma)(R_0 -1).
\end{equation}
Note that $\lambda > 0$ precisely when $R_0 > 1$, in which case the initial infection may grow rapidly and lead to a widespread outbreak; when $R_0 < 1 $, the infection will die out over time without causing a major outbreak.

Two other parameters of relevance are the asymptotic probability
of a large outbreak, $\pi$,  and its relative size, $r_\infty$ 
in the case without travel, i.e., in a single community of $n$ individuals with connection probability $c/n$. They are known to be given as the largest solutions in $[0,1]$ of the implicit equations
 \begin{equation}\label{eqn_theta_def}
        1-\pi = G(\pi) := \gamma \int_0^\infty \exp(-c(1-e^{-\beta \ell})\pi - \gamma \ell) \,d\ell
    \end{equation}
    and
    \begin{equation}\label{eqn_r_infty}
        1 - r_\infty = e^{-R_0 r_\infty},
    \end{equation}
respectively.

\section{Results}
\label{sec:results}

To assess the impact of interventions, we first characterize the spread of an epidemic across two communities in the absence of any control measures (Theorem~\ref{thm_sir_tcm_non_int}). We then use this as a baseline to evaluate how different interventions alter this baseline trajectory, focusing on their impact in reducing the final outbreak size, decreasing the probability of large outbreaks, and delaying the onset of the epidemic. We show that travel bans are largely ineffective in limiting outbreak size, yet intra-community interventions, even when implemented late, can substantially reduce the impact in the second (Theorems~\ref{thm_travel_ban} and \ref{thm_social_distance}).

In the discussion that follows, it is helpful to formalize what we mean by a substantial outbreak. In particular, we use the following definition to distinguish between cases where the infection dies out quickly and those where it spreads to a significant fraction of the population:

\begin{definition}
    An outbreak is called \textbf{noticeable} (in a community or in the entire population) if it eventually infects at least $\epsilon n$ individuals for some small $\epsilon \in (0,r_\infty)$.\footnote{The parameter $\epsilon$ should be thought of as a numerical constant (independent of $n$) that is a small fraction of $r_\infty$, even though mathematically all we need is the strict inequality $\epsilon<r_\infty$.}
\end{definition}

\paragraph{Heuristics.}
Intuitively, one may expect that a travel ban would be effective if implemented when the outbreak first becomes noticeable in community 1. However, we show that by that time, travel has already seeded a moderate number of infections in community 2---more than enough to sustain its own growth, even if all subsequent travel is halted. Further imported infections have only a secondary effect compared to the internal transmissions.

In contrast, intra-community interventions, modeled as reducing the transmission rate or equivalently a reduction in $R_0$, turns out to be an effective intervention. If implemented in community 2 when the outbreak is noticeable in community 1, a noticeable outbreak in community 2 will very likely be prevented. The intuition for this is as follows: at the time when the outbreak is noticeable in community 1, there are only a small number of infections in community 2. Intra-community interventions reduce the rate of transmission among individuals in community 2 such that each newly infected (or imported) case will, on average, infect a bounded number of individuals before they all recover. Meanwhile, the epidemic in community 1 will continue to grow until it reaches herd immunity and later terminates. During this period, the number of imported cases in community 2 is of low order. Thus, the total size of the infection will remain sublinear and the outbreak will never become noticeable in community 2.

\begin{figure}[ht]
    \centering
    \includegraphics[width=0.9\linewidth]{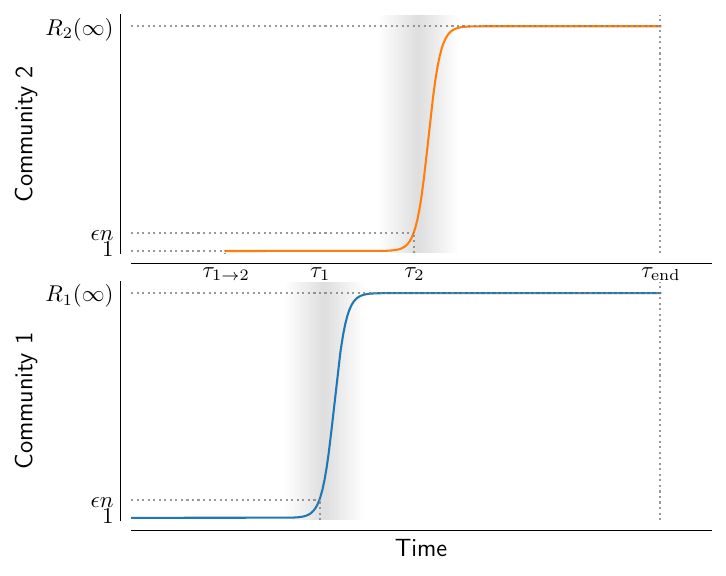}
    \caption{Schematic illustration of epidemic trajectory in the two communities (not to scale). $\tau_1$ and $\tau_2$ mark the major ``waves'' in the two communities: most infections occur within the narrow bands around them. $\tau_{1\to 2}$ marks the initial infection of any type-2 individual. $\tau_{\text{end}}$ indicates the end of the epidemic, i.e., when the number of infectious individual first drops to zero.}
    \label{fig:schematic}
\end{figure}

Figure~\ref{fig:schematic} provides a schematic illustration of the epidemic trajectory in the absence of intervention. The spread across the two communities can be partitioned into four phases:
\begin{itemize}
    \item \textbf{From time $0$ to $\tau_{1 \to 2}$:}  
    The infection starts to spread within community~1. It may die out early, but if it survives, it begins to grow exponentially.

    \item \textbf{From $\tau_{1 \to 2}$ to $\tau_1$:}  
    Infections begin seeding community~2 via travel. Both communities experience exponential growth, though community~2 still has only a small number of infections.

    \item \textbf{From $\tau_1$ to $\tau_2$:}  
    The outbreak in community~1 becomes noticeable and soon passes the herd immunity threshold. Community~2 accumulates a polynomial number of infections. At this stage, a travel ban is no longer meaningful, as enough infections have already been seeded to sustain exponential growth in community~2.

    \item \textbf{From $\tau_2$ to the end:}  
    Community~2 experiences its own noticeable outbreak and eventually reaches herd immunity, regardless of travel policy. However, if social distancing is implemented in community~2 at time $\tau_1$, the outbreak will stay below noticeable size throughout.
\end{itemize}

\paragraph{Formal statements.}
We now formalize our heuristics developed above. To formally state our results, we define the following stopping times for the times at which the epidemic reaches an $\epsilon$-faction of individuals in community 1 and 2:
\begin{align}
    \tau_1(\epsilon) &\defeq \inf \{t>0 : I_1(t) +R_1(t)\geq \epsilon n\}, \label{eqn:tau1def}\\
    \tau_2(\epsilon) &\defeq \inf \{t>0 : I_2(t) +R_2(t) \geq \epsilon n\}.\label{eqn:tau2def}
\end{align}
Also, define the stopping time for the first time that the infection arrives in community 2 by
\begin{equation}\label{eqn:tau1to2def}
    \tau_{1\to 2} \defeq \inf \{t > 0: I_2(t) > 0\}.
\end{equation}

Throughout this paper, we use the standard terminology of an event happening \emph{with high probability}, or w.h.p., to mean that the probability of the event happening tends to $1$ as the population size parameter $n$ goes to infinity. We use $O(\cdot)$ and $\Theta(\cdot)$ to hide constant factors; i.e., we say $f = O(g)$ if $g>0$ and there exists an absolute constant $C<\infty$ such that $\limsup_{n\to\infty} |f|/g < C$, and we say  $f=\Theta(g)$ for $f,g>0$ if both $f=O(g)$ and $g=O(f)$.
We also use the notation $\tilde O(\cdot)$ to suppress logarithmic factors; i.e., $f = \tilde O(g)$ if $f = O(g \ln^a n)$ for some $a\in\N$. We say $f = o(g)$ if $f/g\to 0$ as $n\to\infty$, and use $\pto$ to denote convergence in probability.

The following theorem characterizes the epidemic trajectory---typical timing and size of the outbreaks---in both communities in the absence of interventions\footnote{Note that we didn't define stopping  times in terms of the size of the outbreak in a particular community, but instead considered the size of the outbreak counting people of a given type. Note that asymptotically, this makes no difference, since the epidemic dies out in time $O(\ln n)$, implying that during this time only a tiny fraction $\tilde O(n^{-\alpha})$ of individuals  travel.  In fact, our proofs show that all of our theorems remain valid if one defines the stopping times $\tau_1(\epsilon)$ and $\tau_2(\epsilon)$ in terms of either location or community label.}.
Throughout, we will assume that $0<c<\infty$, $0<\gamma<\infty$, $0<\beta<\infty$ and $0<\rho_H<\infty$ are fixed constants independent of $n$, while we scale $\rho_T$ with $n$.

\begin{theorem}[SIR epidemic with no interventions]\label{thm_sir_tcm_non_int}
    Consider the SIR epidemic spreading on $G(t)$ with rate of travel given by $\rho_\Travel = \Theta(n^{-\alpha})$ for some $\alpha \in (0, 1)$.
        Then the following hold:
    \begin{enumerate}[label=(\alph*)]
        \item \label{thm:main-large-outbreak} 
        The probability of a large  outbreak converges to the largest solution $\pi\in[0,1]$ of the fixed point equation  \eqref{eqn_theta_def}, i.e.,
        \[
            \Pr(R_1(\infty) + R_2(\infty) \geq \ln n) \to \pi \text{ as } n\to\infty.
        \]
           \item \label{thm:main-survival-prob} 
           Conditioned on not having a large outbreak, the duration of the epidemic is bounded in probability, 
           with high probability never reaches community 2, and only ever infects a  number of individuals in community 1 that is bounded in probability.  Finally, with high probability, the stopping times $\tau_1(\epsilon)$ and $\tau_2(\epsilon)$ are infinite.
       \item \label{thm:main-final-size}Conditioned on a large outbreak, 
                  $n^{-1} R_1(\infty)$ and $n^{-1} R_2(\infty)$ converge to the largest solution $r_\infty\in[0,1]$ of the fixed point equation \eqref{eqn_r_infty},
                  \[\frac{\lambda}{\ln n}(\tau_{1\to 2}, \tau_1(\epsilon), \tau_2(\epsilon)) \pto (\alpha, 1, 1+\alpha),\]
   \[
   \frac 1{\ln n}\ln\left(I_2(\tau_1(\epsilon))+R_2(\tau_1(\epsilon))\right)\pto 1-\alpha,\]
   and the epidemic dies out in time $O(\log n)$.
   \end{enumerate}
\end{theorem}

\begin{remark}
Theorem~\ref{thm_sir_tcm_non_int} suggests that large outbreaks in the two communities are tightly coupled: a large outbreak in the first community will very likely lead to a large outbreak in the second, only with a delay that is logarithmic in the inverse travel rate. Within each community, once the infection begins to spread, the time it takes to reach macroscopic scale concentrates at $\lambda^{-1} \ln n$, and most infections occur in a narrow window around this time. This behavior closely mirrors what one would expect from SIR dynamics on a single (static) Erd\H{o}s–R\'enyi network. In other words, aside from seeding the outbreak across communities, the spread within each community behaves almost as if the other did not exist.
\end{remark}

We now turn to evaluating the effects of intervention. Specifically, we consider two types of intervention, travel bans and intra-community interventions, applied at the time when community 1 reaches a noticeable outbreak, and analyze their impact on the trajectory in community 2.
In our model of SIR dynamics, these policies are formulated as follows.

\begin{definition}[Travel Ban]
    A travel ban is implemented at time $t_{tb}$ if for all $t\geq t_{tb}$ and all $v \in V_1\cup V_2$, $q_v(t) = q_v(t_{tb})$.
\end{definition}

\begin{definition}[Intra-Community Intervention]
    An intra-community intervention is implemented in a community at time $t_{ic}$ if for all $t \geq t_{ic}$ we assume that in that community $\beta$ and $c$ are replaced by $\beta'$ and $c'$ such that $\lambda' = c'\beta' - \beta' - \gamma < 0$.  
\end{definition}

\begin{theorem}[Effect of Travel Ban]\label{thm_travel_ban}
Consider the setting of Theorem \ref{thm_sir_tcm_non_int} and suppose that at time $\tau_1(\epsilon)$ a travel ban is implemented. Then the same conclusions as in Theorem \ref{thm_sir_tcm_non_int} hold. In particular, the asymptotic probability of a noticeable outbreak remains unchanged and, conditional on such an outbreak, its final size in each community still converges to $r_\infty$ and $\tau_2(\epsilon)/\ln n \overset{p}{\to} 1+\alpha$ for any $\eps < r_\infty$. 
\end{theorem}

Not only does the theorem state that the probability and size of a large outbreak in the second community is asymptotically the same whether a travel ban is implemented or not, but it also says that a delay of onset (in Community 2) due to travel distancing is of lower order, consistent with empirical observations.

\begin{theorem}[Effect of Intra-Community Intervention]\label{thm_social_distance}
Let $R_0>1$, and suppose that conditioned on $\tau_1(\epsilon)<\infty$ intra-community interventions are implemented  at time $\tau_1(\epsilon)$ . Then all statements from Theorem \ref{thm_sir_tcm_non_int} still hold, 
except for those involving $\tau_2(\epsilon)$ and $R_2(\infty)$.  Instead, we have that conditioned on a large outbreak
 \[\frac 1{\ln n}\ln R_2(\infty)\pto 1-\alpha.\]
\end{theorem}
Note that the theorem implies that an intra-community intervention in community 2 does not change the asymptotics of the final size in community 1, while it does change the final size in community 2, which now is sublinear in $n$.  In fact, comparing the final size in community 2 to the size at time $\tau_1(\epsilon)$ from Theorem~\ref{thm_sir_tcm_non_int}, one sees that if an intra-community intervention is implemented at time $\tau_1(\epsilon)$,
then the final size of the epidemic in community 2 does not grow much afterwards -- it can grow at most by a factor $n^{o(1)}$.
\section{Discussion}

Our results demonstrate that common intervention policies have varied levels of efficacy in limiting epidemic impact across populations with weak ties. While travel bans are often perceived as a natural first response to emerging outbreaks, our analysis shows that they have minimal effect on both the final size of the outbreak and its timing. Even when implemented early, travel restrictions do not substantially delay the onset of infection in the second community: the excess delay is of lower order compared to the natural delay (without interventions). This suggests that travel bans offer limited utility as a standalone measure.

In contrast, intra-community interventions that reduce the transmission rate or contact rate prove highly effective. When enacted in the second community at the moment the outbreak has reached moderate levels in the first, such interventions dramatically reduce the final number of infections. Notably, the reduction is not just in constant factors: the number of infections becomes sublinear in the population size, demonstrating a strong containment effect even under persistent travel.

Although our model is based on a simple Erd\H{o}s--R\'enyi network, the insights likely extend to more general and realistic contact structures, such as configuration models, clustered networks, or household models, provided travel remains rare and mixing within communities is sufficient. With more than two communities, we expect qualitatively similar dynamics, with delays in outbreak timing depending on the inter-community travel densities, and infectious spread following the ``most densely connected'' paths. 

The intuition behind our proofs points to possible generalizations along the lines outlined in the previous paragraph.  We expect that the mechanism identified in our paper, the interaction of early exponential growth in the first community (before detection), hidden seeds in the second community, and further exponential growth in the second community, do hold much more generally. While there are many details that would need to be worked out, much of the technology needed to generalize our proof to the setting mentioned in the previous paragraph (and beyond) already exists; the powerful tools of multi-type, continuous time branching processes allow for the identification of the initial rate of exponential growth, even for more complicated network or disease models.

However, in the current context, such  generalizations would not add further insight; rather, they would hide the inherent simplicity of the underlying mechanism in more elaborate proofs. Instead, we run simulations to demonstrate that our results  hold on networks with more realistic contact structure and heterogeneity in individual infectiousness.

\begin{figure}[tbhp]
    \centering
    \begin{subfigure}[b]{0.49\textwidth}
        \centering
        \includegraphics[width=.9\linewidth]{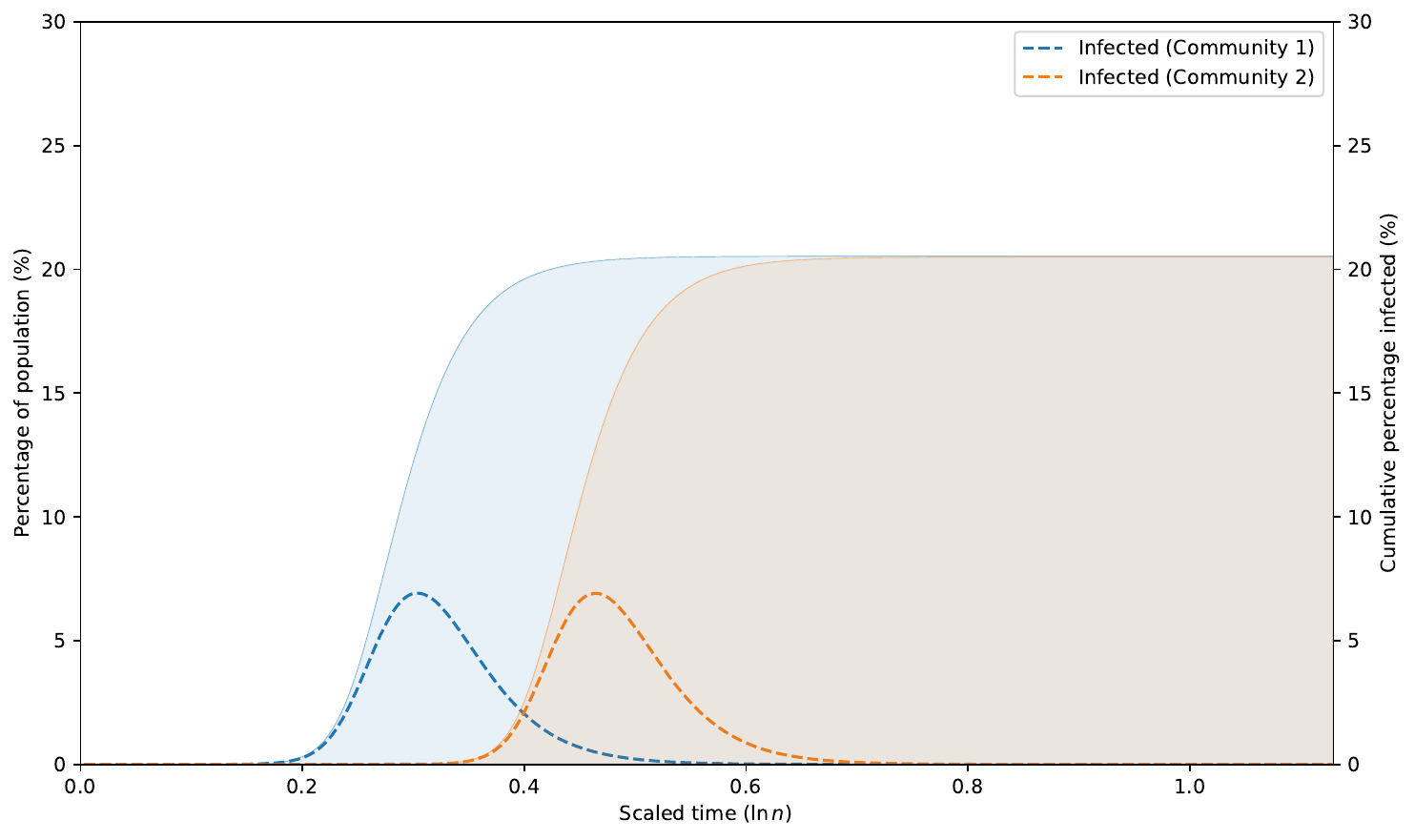}
        \label{fig:sim-no-intervene}
        \caption{No intervention.}
    \end{subfigure}
    \begin{subfigure}[b]{0.49\textwidth}
        \centering
        \includegraphics[width=.9\linewidth]{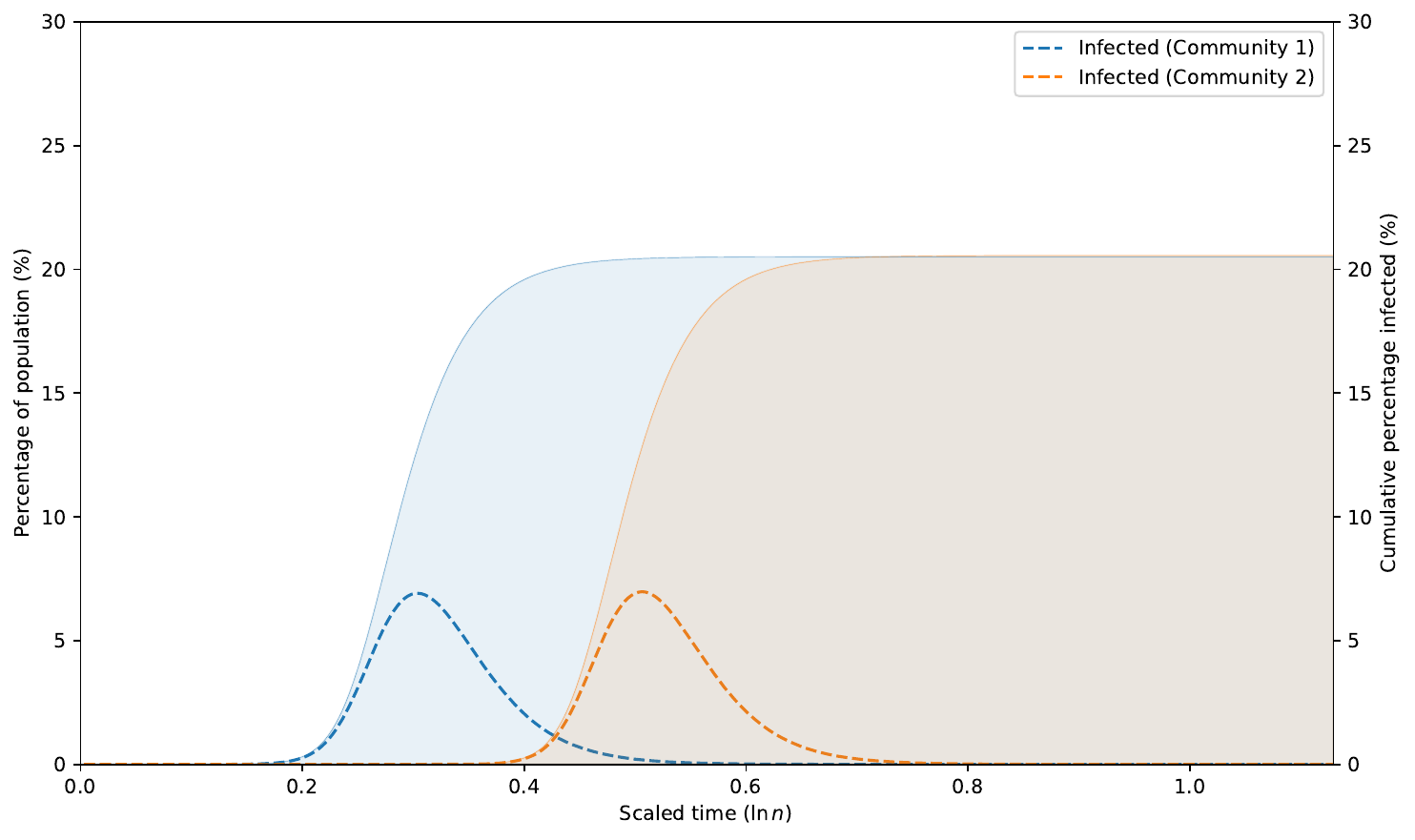}
        \label{fig:sim-tb}
        \caption{Travel ban implemented at $\tau_1(1\%).$}
    \end{subfigure}
    \begin{subfigure}[b]{0.49\textwidth}
        \centering
        \includegraphics[width=.9\linewidth]{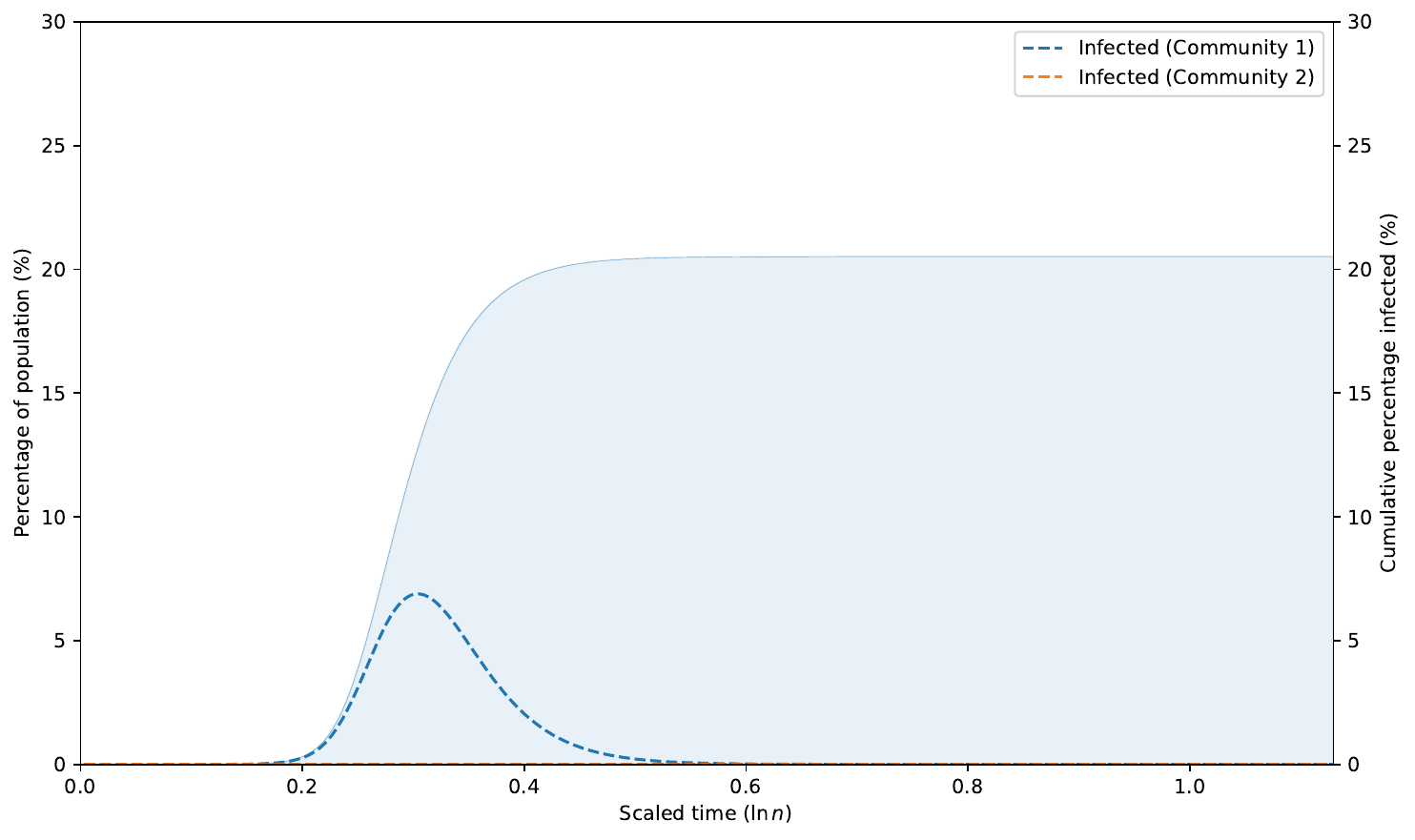}
        \label{fig:sim-sd}
        \caption{Intra-community intervention in community 2 implemented at $\tau_1(1\%).$}
    \end{subfigure}
    \caption{Simulation on two communities of $n=10^7$ individuals each, with an underlying network generated from a configuration model with lognormal distributed degrees ($\mu = 1.40, \sigma = 1.27$); parameters taken from \cite{hill2021network}. The dashed infection curves represent the proportion of infected individuals at any point in time (left axis), and the solid curves represent the cumulative percentage of infections (right axis). The infection rates are exponentially distributed with parameter $\beta = 0.1$. The recovery rates are gamma distributed with shape parameter $3.5$ and scale parameter $0.2041$. These parameters yield $R_0 \approx 3$. Other parameters are: $\rho_\Travel = n^{-1/2}$, and $\rho_\Home = 1$.}
    \label{fig:sim}
\end{figure}

\newpage

\newpage

Figure~\ref{fig:sim} shows a simulation on a contact network with two communities generated using configuration models with $n = 10^7$ individuals in each community. Individuals are randomly assigned into two types and an underlying network with degrees drawn from the lognormal distribution is sampled. The parameters of the lognormal distribution are taken from existing literature that constructs an individual-based network for studying COVID transmission dynamics \cite{hill2021network}.

We normalize $\rho_\Home = 1$ and set $\rho_\Travel = n^{-1/2}$. For more realistic infectiousness periods of individuals, the recovery rates are modeled using a gamma distribution rather than an exponential distribution \cite{lloyd2001realistic, wearing2005appropriate}, with $\beta$ and the parameters of the gamma distribution chosen to roughly match $R_0$ and the dispersion parameter of Omicron \cite{kremer2022serial, sender2022unmitigated}. We assume that any interventions are implemented when a $1\%$ outbreak is detected in the first community,\footnote{Although the $1\%$ intervention threshold may appear conservative, it is expressed in terms of the \emph{true} cumulative infections, which typically exceed the cases that are detected or reported and grow exponentially during the early phase. Even under ideal surveillance---where cases are perfectly identified and reported and a travel ban is imposed promptly---the ban's effectiveness still hinges on whether the earliest infectious individuals have already traveled. By contrast, social-distancing measures remain robust: once introduced, they limit transmission irrespective of when they start (except for the infections that have already occurred).} and that intra-community interventions have the effect of reducing the rate of infection by a multiplicative factor of $5$ (any factor that reduces the effective $R_0$ to below $1$ in community 2 suffices). We include additional robustness checks in the SI Appendix, where we run our SIR dynamics on contact networks taking into account degree assortativity and disassortativity, and separately on contact networks fit to phone network data from the Copenhagen Network Study. Those simulation results qualitatively match the ones above and in all cases, the message remains clear: intra-community interventions that reduce the effective reproduction rate in the second community will successfully prevent a noticeable outbreak, while a travel ban is ineffective in this respect and only delays the peak by a minor amount.

One could also consider the regime of very sparse travel (for example, when the travel rate scales as $\Theta(1/n)$), which is similar in spirit to the setting studied in \cite{ball2024epidemic}. This regime differs slightly from our current model in that even without intervention, infections may not always reach the second community, due to the infrequent travel. We conjecture that in the sparse setting, a travel ban could modestly reduce the likelihood of cross-community transmission, though not eliminate it completely. Nevertheless, we expect that at a high level the core insights persist: when infections are seeded in a second community, interventions that reduce within-community transmission remain substantially more effective than travel bans.

Finally we note that, since our results are asymptotic, if $n$ is small enough, then even in our setting travel may be infrequent enough such that a travel ban would reduce or eliminate cross-community transmission. Indeed, when the infection in community 1 has reached size $\epsilon n$, the expected number of seeds in community 2 is of order $\tilde O(\epsilon n \rho_T)=\tilde O(\epsilon n^{1-\alpha})$. Thus when $n$ is small enough and infections are detected at a sufficiently low threshold $\epsilon$, it is possible that no infections have been seeded in community 2 at time $\tau_1(\epsilon)$, making travel bans effective if implemented  strictly. However, for large communities this will not be the case, and our simulations show that for a population of size $10^7$ with $\epsilon = 0.01$ and $\alpha=0.5$, $n$ is already large enough such that a travel ban is no longer effective.

\section{Proof Overview}

We will prove Theorems~\ref{thm_sir_tcm_non_int}, \ref{thm_travel_ban} and \ref{thm_social_distance} in the supplement.  Here we give a brief overview, concentrating mainly on the intuition for why one should expect these theorems to hold.  We start with Theorem~\ref{thm_sir_tcm_non_int}.

To understand the early phases of an epidemic seeded in community 1, we note that as long as the epidemic is small, say of order at most $O(\ln n)$, the chance that one of the infected individuals travels after being infected is bounded by
$\tilde O(n^{-\alpha})$.  Thus in the early stages, the epidemic evolves essentially as if there was just one community, so in particular, the probability that it  grows to $\ln n$  is given by the solution $\pi$ to \eqref{eqn_theta_def}, the probability of a large outbreak in  a single community. Furthermore, in these early stages, the infection grows at the single community rate $\lambda$, showing that as long as $\lambda t\ll \ln n^\alpha$, the size in community $1$ is much smaller than $n^{\alpha}$, making it still unlikely that any of the infected individuals will travel.

Once the infection passes size $n^{\alpha}$, it becomes likely that one of the infected individuals travels after being infected (or that the infection hits a traveler who then returns to their home community).  In either case, these infected individuals serve as seeds in 
community 2, each having constant probability of contributing to $I_2(t)$.  When the number of these attempts diverges, say like $(\ln n)^2$, we will w.h.p. get a contribution to $I_2$, explaining why $\tau_{1\to 2}$ behaves like $\frac 1\lambda \ln n^\alpha$.  

After this point, the infection continues to grow exponentially in community 1, but it also starts to grow exponentially in community 2, in both cases at rate $\lambda$.  While at some point, one might expect travel in the opposite direction to create new seeds in community 1, when that happens, the epidemic in community 1 is already large, and the new seeds contribute negligible to the overall growth in community 1. This shows that it will take roughly time $\lambda^{-1}\ln (\eps n)=\lambda^{-1} \ln n (1+o(1))$ to grow to size $\epsilon n$, showing that we should expect $\tau_1(\epsilon)$
to be roughly equal to $\lambda^{-1} \ln n$.

Between time $\tau_{1\to 2}$ and time $\tau_1(\epsilon)$, the epidemic grows exponentially in both communities at rate $\lambda$, with the size in community $2$ behind by a factor of $n^{-\alpha}$.  Thus, seeds by travel (which happens at rate $\Theta(n^{-\alpha}))$ and the exponential growth of the original seeds contribute the same order of magnitude to the growth of the infection. Since this phase takes time of order $\frac {1-\alpha}\lambda\ln n$, one might expect slightly faster growth in community $2$, possibly multiplying the size of the epidemic by a factor of order $\ln n$.  However, once the epidemic in community $1$ has reached size $\epsilon n$, it will reach herd immunity in constant time and die out, at which point all that happens is exponential growth in community 2 at rate $\lambda$.  Thus, it should take between time roughly 
$\lambda^{-1} \ln \big(\frac {\eps n}{\ln n}\big)$ and 
$\lambda^{-1} \ln (\eps n)$ to grow the epidemic from size of order $1$ to size $\epsilon n$ in community 2. Since the difference between the two is of lower order, we expect $\tau_2(\epsilon)$ to be roughly  $\tau_{1\to 2}$ plus $\lambda^{-1}\ln n$, explaining the scaling of $\tau_2(\epsilon)$.

Note that the above arguments also explain why travel restrictions don't help much once the epidemic is noticeable in community 1.  At this point, the outbreak in community is already of order $n^{1-\alpha}$ (possibly slightly larger, by a factor of $\ln n$), large enough to make the epidemic grow in community 2 until it has reached its final size.  Travel restrictions might reduce the additional seeds after that time a little, but at this point, the growth in community 2 is still exponential, while community 1 can't grow by more than a constant factor. This shows that the effect on the time (which is logarithmic in the seed size) is of constant order, i.e., of much lower order than the leading order $\ln n$.  To understand our final statement, the fact that the overall outbreak size is only influenced minimally, we note that as long as the seed of an infection is sublinear, the final size of an epidemic does not depend on the seed size, as long as it is large enough to lead to large outbreak with high probability.

To make the above considerations rigorous, we proceed in three steps to prove Theorem~\ref{thm_sir_tcm_non_int}, which we now explain.  After that, we also give a short overview of the proofs of Theorem~\ref{thm_travel_ban} and \ref{thm_social_distance}.

\paragraph{One community lower bounds}
As a first step, we establish a lower bound on the size of the outbreak in community 1 by only considering individuals which never travel.  More precisely, since we expect the total time of the epidemic to be of order $\ln n$, we exclude everyone who travels up to time $T = \ln^2 n$.
Since the number of travelers up to this time is bounded in probability by $ O(n^{1-\alpha}\ln^2n)$ this gives a lower bound on the size of the epidemic in community 1 by that on $G(n',p)$ where $n'=n-\tilde O(n^{1-\alpha})$.  Asymptotically, these bounds are indistinguishable from those for $G(n,p)$, giving us an upper bound on both $\tau_1(\epsilon)$ and the time it takes for the epidemic to grow to, say, size $n^{\alpha+\delta}$ in community 1, where $\delta$ is a small constant. Once the epidemic has grown to size $n^{\alpha+\delta}$ in the group of non-travelers in community 1, the 
non-travelers infected by that time had enough contacts with travelers for them to seed, say, $n^{\delta/2}$
infections among the non-travelers in community 2,
giving in particular an upper bound on $\tau_{1\to 2}$; furthermore, having established enough seeds in the community of non-travelers of type 2,
we can now  bound the subsequent growth in community $2$ from below by that of an isolated community on $G(n',p)$, leading to the desired upper bound on $\tau_2(\epsilon)$, and a lower bound on the final sizes in both communities.

\paragraph{2-type branching process upper bounds} Next, we use the well-known fact that a branching process constructed by running the same SIR dynamics but ignoring the depletion of susceptible individuals serves as an upper bound for the number of infected and recovered vertices for all times. It is less clear whether this is still true for dynamic graphs, since speeding up earlier parts of the time evolution may change where infections are possible: e.g., individuals that could have been infected in one community may now be traveling, leading to an unexpected slowdown instead of a speedup. It turns out that a careful coupling of (a) the exploration of the random graph over the two communities (b) the travel dynamic, and (c) the dynamics of the epidemic does give a valid branching process upper bound for all times.  The analysis of this branching process will lead us to consider a 16-dimensional rate matrix, which we then analyze by perturbation theory around the dynamics without travel. While the branching process eventually gives a pessimistic upper bound, for determining the stopping times $\tau_{1\to 2}$, $\tau_1(\epsilon)$ and $\tau_2(\epsilon)$, the relevant bounds are off by only additive terms which grow slower than the leading $\ln n$ behavior and suffice for our purposes.

\paragraph{Herd immunity and the final size of the epidemic.}
While our branching process upper bounds are tight enough to establish the asymptotics of the stopping times $\tau_{1\to 2}$, $\tau_1(\epsilon)$ and $\tau_2(\epsilon)$, they become non-informative  as we move past these times and do not provide any upper bounds on the final size of the epidemic. To deal with this fact, we loosely decompose the total number of infections into two parts: (1) travelers who get infected, and (2) the spread contained within each of the two non-traveling populations (possibly starting with a transmission from a traveler). The contribution of the first part is limited since the total number of travelers is of order $\tilde O(n^{1-\alpha})$ up to time $T = \ln ^2 n$. With travelers isolated, the spread within the non-traveling parts can be coupled to an SIR process within a single community, except that we may have multiple initial seeds due to contact with travelers. But it is well known that the relative final size of an SIR epidemic on $G(n,p)$ converges to $r_\infty$ as long as the number of seeds goes to infinity in a sublinear fashion. Since we expect at most $\tilde O(n^{1-\alpha})=o(n)$ seeds up to time $T$ (the number of travelers is of this order and each has, in expectation, a constant number of neighbors), this shows that the asymptotic final size is of order $r_\infty n$ plus lower order corrections.

To conclude this argument, we only need to show that the infection dies out by time $T = \ln^2 n$. This relies on techniques developed in \cite{borgs2024lawlargenumberssir}, which are based on the the intuition of herd immunity towards the end of the epidemic. By time $C \ln n < T$ for some sufficiently large constant $C < \infty$, the level of infection in both communities can be controlled to be low as desired (as our lower and upper bounds are converging at later times), and the remaining fraction of susceptible individuals is too small to sustain further growth. Using a subcritical branching process to bound the dynamics in the remaining time, we show that the number of actively infected individuals decreases exponentially and thus drops to zero in $O(\ln n)$ time, ending the epidemic before $T$ as claimed.

\paragraph{Travel bans.}
To prove Theorem~\ref{thm_travel_ban}, we first note that our lower bound never used the additional seeds after $\tau_{1}(\epsilon)$, in fact, we entirely ignored travels beyond the time $\lambda^{-1}(\alpha+o(1))\ln n $ (which is before $\tau_1(\eps)$ w.h.p.) in the proof anyway.
    
    For the upper bound on the sizes of outbreak in both communities, the analysis is analogous to the proof in the presence of travel.  A branching process with growth rate $\lambda$ serves as an upper bound until the outbreak reaches a linear size; the same argument as before gives $r_\infty$ as an upper bound on the final fractional sizes of outbreak in both communities; and finally the arguments from \cite{borgs2024lawlargenumberssir} again provide an upper bound on the elapsed time of the outbreak using the fact that we have reached herd immunity.
    
\paragraph{Intra-community Interventions.}
Suppose that we implement an intra-community intervention in community 2 at time $\tau_1(\epsilon)$. Since the infection in community 1 dies out in time $O(\ln n)$, only $\tilde O(n^{1-\alpha})$ vertices from community 1 travel to community 2 before the infection dies out in community 1. However, since the effective $R_0$ in community 2 is smaller than $1$ due to the intervention, any newly seeded infection tree dies out exponentially quickly and only infects a constant number of individuals in community 2. Thus, only $\tilde O(n^{1-\alpha})$ individuals in community 2 will ever be infected and the final size of the infection in community 2 is sublinear in the number of individuals.

\section*{Acknowledgements}
Karissa Huang acknowledges support from the National Science Foundation under grant DGE 2146752.

\bibliographystyle{alpha}
\bibliography{ref}

\newpage

\appendix
\section*{Appendix}
In this appendix, we rigorously prove the results from the main text, first establishing our results on the dynamics of the epidemics without interventions (Theorem \ref{thm_sir_tcm_non_int} from the main text). As explained in Section \ref{sec:results} of the main text, we will do this in three steps:
\begin{itemize}
    \item First, we will use known results for the dynamics of SIR on $G(n,p)$ to establish upper bounds on the final size, and
    lower bounds on the growth $I_i+R_i$ for both types by considering the subset of individuals which never travel.  This will give upper bounds on the three quantities $\tau_{1\to 2}(\epsilon)$, $\tau_2(\epsilon)$, and $\tau_2(\epsilon)$ matching the expected behavior of these quantities.
    \item Next, we derive upper bounds on this growth by constructing a suitable multi-type branching process which allows us to establish the matching lower bounds on 
    $\tau_{1\to 2}(\epsilon)$, $\tau_2(\epsilon)$, and $\tau_2(\epsilon)$.
    \item Finally, we will prove that the epidemics lasts at most time $O(\log n)$ by exploiting that eventually, the epidemic hits the herd immunity threshold.
\end{itemize}

To prove our lower bounds, we will rely on various properties of the SIR dynamics on a single $G(n,p)$ random graph.

\section{Preliminaries}

\subsection{Single Community Results}
To state the needed single community results, we recall our notation $R_0=\frac {c\beta}{\beta+\gamma}$ for the basic reproduction number and  $\lambda =(c-1)\beta -\gamma$  for the exponential growth rate.   We will use $T_{\Pois(c)}$ to denote a random tree with Poisson $c$ off-spring, and $\BPsub{c}(t)$ to denote the infection tree at time $t$ of an SIR infection on $T_{\Pois(c)}$ starting at the root,  we define  $r_\infty$ and $\pi$ to be the survival probabilities of the branching processes $T_{\Pois(R_0)}$ and  $\BPsub{c}(t)$.

Next, we note that if we start an infection with at least $\epsilon n$ infected vertices, then the SIR dynamic obeys a law of large numbers which can be expressed in terms of the solution of the following set of differential equations
\begin{equation}\label{SIR-diff}
\frac {ds}{dt}=-\beta xs
,\quad
\frac {dx}{dt}= -(\beta+\gamma)x +\beta c x s 
,\quad
\frac{di}{dt}=\beta c x -\gamma i
, \quad
\frac {dr}{dt}=\gamma i
\end{equation}
where $s$, $i$ and $r$ represent the limit of $\frac 1n S(t)$, $\frac 1n I(t)$ and $\frac 1nR(t)$ and $x$ is the limit of $\frac 1n X(t)$, where $\beta X(t)$ is the force of the infection, i.e., the rate at which a random vertex in $S$ gets infected, see, e.g., 
\cite{ball2020epidemics,borgs2024lawlargenumberssir} for the derivation of this equations and the proof of the law of large numbers.

Finally, 
we define two stopping times: the time $\tau(\epsilon)$ when the infection has reached relative $\epsilon$ where $0<\epsilon<r_\infty$, 
\begin{equation}\label{tau-single}
      \tau(\epsilon) := \inf\{t > 0 : S(t) < (1-\epsilon)n\},
\end{equation}
and the time $\tau_\infty$ when the 
infection dies out,
\begin{equation}\label{tau-infty-single}
    \tau_\infty=\sup\{t>0: I(t)>0\},
\end{equation}
both for an infection starting from a single infected individual at time $0$.
We call $R(\infty)$  the \textit{size of the outbreak}, and say it is a \textit{large outbreak} if $R(\infty)\geq \ln n$.

\begin{theorem}\label{thm:single-comm}
Consider the SIR epidemic spreading on an Erd\H{o}s--R\'enyi random graph $G(n, c/n)$,  Assume there is one initially infected vertex unless stated explicitly otherwise.
Then the following hold:
\begin{enumerate}[label=(\alph*)]
    \item\textbf{(Continuity of $\mathbf r_\infty$ and $\mathbf \pi$)} Both $r_\infty=r_\infty(R_0)$  and $\pi=\pi(c,\beta,\gamma)$ are continuous functions taking values in $[0,1)$,  both are strictly positive if and only if $R_0> 1$. For $R_0>1$, the  herd immunity threshold $s_0=1/R_0$ is strictly larger than $s_\infty=1-r_\infty$.
    
    \item \label{bullet:single_comm_prob_outbreak}\textbf{(Probability of outbreak)} The probability of a large outbreak of at least $\ln n$ vertices in the graph converges to $\pi$, i.e., 
    \begin{equation}
        \Pr(R(\infty) \ge \ln n) \to \pi \text{ as }n\to\infty.
    \end{equation}
    If $R_0>1$ and $\epsilon\in (0,r_\infty)$ then
    \begin{equation}
        \lim_{n\to\infty} \Pr(R(\infty) \ge \epsilon n \mid R(\infty) \ge \ln n) = 1.
    \end{equation}

    \item \textbf{(Final size)} \label{bullet:single_comm_final_size}
   Conditioned on not having a large outbreak, $R(\infty)$ 
     is bounded in probability.  If $R_0>1$ then conditioned
  on a large outbreak, we have
    \begin{equation*}
        \frac{R(\infty)}{n} \overset{p}{\to} r_\infty.
    \end{equation*}

    \item\textbf{(Initial growth)}\label{bullet:single_comm_init_growth}
    The SIR epidemic on $G(n,c/n)$ and the branching  processes $\BPsub{c}$  and $\BPsub{c(1-\epsilon)}(t)$ can be coupled in such a way that the following holds whenever $N_n=o(\sqrt n)$.  
        \[
            \Pr\left(\Big||\BPsub{c}(t)| = (I(t)+R(t))\Big| \text{ for all } t \text{ s.t. } |\BPsub{c}(t)| \leq N_n)\right) \to 1\quad\text{as}\quad n\to \infty.
        \]
        Furthermore, 
    for $R_0>1$ and any $a\in(0,1)$, conditioned on a large outbreak, we have 
    \begin{equation*}
        \frac{\tau(n^{a-1})}{\lambda^{-1}\ln n} \overset{p}{\to} a.
    \end{equation*}

    \item\textbf{(Duration of the epidemic)}\label{bullet:final-time}
    Conditioned on not having a large outbreak  $\tau_\infty$ is bounded in probability.  If $R_0>1$ and we condition on a large outbreak then $(\ln n)^{-1}\tau_\infty$ is bounded in probability.
    
    \item\textbf{(Growth of a linear-sized infection)}\label{bullet:single_comm_diff_eq} 
    Assume $R_0>1$.  Then there exist constants $\epsilon_0>0$ and $0<C_1<C_2<\infty$ such that for $0<\epsilon\leq \epsilon_0$, conditioned on a large outbreak, w.h.p.  $C_1\epsilon n \leq X(\tau(\epsilon))\leq C_2\epsilon n$.
    Furthermore, conditioned on the state of the infection at time $\tau(\epsilon)$, the evolution of the SIR dynamic is well described by set of equations \eqref{SIR-diff}, in the sense that 
    $$\sup_{t\geq \tau(\epsilon)}\left|\frac 1n(S(t),X(t), I(t)))-(s(t), x(t), i(t))\right| \overset{p}{\to}0.$$

    \item\textbf{(Concentration for $\tau$)}\label{bullet:single_comm_tau_concentration} Suppose that $R_0>1$, $I(0) = 1$, and let $0<\epsilon'<\epsilon_0$ and $\epsilon'<\epsilon<r_\infty$. Then there exists $K=K(\epsilon', \epsilon)$ such that conditioned on a large outbreak,     
    w.h.p. $\tau(\epsilon) - \tau(\epsilon')< K$. In particular, combined with \ref{bullet:single_comm_init_growth}, this implies that conditioned on a large outbreak
    \begin{equation*}\label{eqn_T_limit}
        \frac{\tau(\epsilon)}{\lambda^{-1}\ln n} \overset{p}{\to} 1.
    \end{equation*}

    \item\textbf{(Many initially infected vertices)}\label{bullet:LargeI0} 
    Assume $I(0)\to\infty$ and $I(0)=o(n)$,  with the initially infected vertices chosen independently from the edge set of $G(n,c/n)$.  
    Then 
   \begin{equation*}
        \frac{R(\infty)}{n} \overset{p}{\to} r_\infty
    \end{equation*}

\end{enumerate}

\end{theorem}

\begin{proof}
Except for the last statements in \ref{bullet:single_comm_init_growth} and \ref{bullet:final-time} and both statements in  \ref{bullet:single_comm_tau_concentration},
    the proofs of the rest of the statements
        are standard and can be found in the literature, e.g. in \cite{neal2003sir, ball2020epidemics, borgs2024lawlargenumberssir}. 
        
        The proof of the last statement of \ref{bullet:single_comm_init_growth}, that conditioned on a large outbreak we have concentration for $\tau(n^{\alpha-1})$, follows from standard branching process theory once we make the following two observations:
        \begin{enumerate}[label=(\alph*)]
            \item we can upper bound $I(t)+R(t)$ by the size of an infection tree on the random tree $T_{\Bin(n,p)}$ with $\Bin(n,p)$ off-spring distribution
            \item we can  lower bound $I(t)+R(t)$ by the size of the infection tree on $T_{\Bin(n',p)}$ up to the stopping time when $I(t)+R(t)$ first exceeds $n-n'$.
        \end{enumerate}    
        The proof of  \ref{bullet:single_comm_tau_concentration} follows by combining these arguments with the law of large numbers behavior of the process. Specifically,  combining the upper and lower bounds on $I(t)+R(t)$ with standard branching process theory, we get that for $0<\delta<1$, whp, 
    \[
     (1-\delta)\lambda^{-1}\leq    \frac{\tau(\epsilon')}{\ln n}\leq \frac{1+\delta}{\lambda -c\beta\epsilon'}.
    \]
Finally, using the  bound $C_1\epsilon' n<X(\tau(\epsilon'))< C_2\epsilon' n$ from
 \ref{bullet:single_comm_diff_eq} and the law of large numbers for $S$ and $X$ it is straightforward 
 to prove that whp the epidemic will grow from size $I+R=n-S=\epsilon' n$ to $n-S\geq  \epsilon n$ in finite time, i.e., in time at most $K(\epsilon',\epsilon)<\infty.$  Sending first $n$ and then $\delta$ and $\epsilon'$ to zero one gets that 
    \[
        \frac{\tau(\epsilon)}{\lambda^{-1}\ln n} \pto 1.
    \]
The last statement of \ref{bullet:final-time} is easily proven by first noting that \ref{bullet:single_comm_tau_concentration} implies that it  takes time of order $\ln n$ to pass the herd immunity threshold and then using the arguments of, e.g., \cite{borgs2024lawlargenumberssir} that after passing the herd immunity threshold $X(t)$ and $I(t)$ decay exponentially in $t$.
  \end{proof}

Next, we recall several couplings for the SIR dynamics which will be useful in this paper.

\begin{enumerate}
\item \textbf{(Infection digraph)} \label{D-SIR-couplings}
This coupling is defined for an SIR infection on an arbitrary finite graph $G$, and involves the construction of an  infection digraph, to be denoted $\mathcal D_G^{SIR}$.  It is defined by drawing  recovery times $T_v\sim \Expo(\gamma)$ i.i.d. for all vertices $v$ as well infection times $T_{uv}\sim \Expo(\beta)$ i.i.d. for all  oriented edges $uv$ such that $\{u,v\}$ is an edge in $G$.  The infection digraph $\mathcal D_G^{SIR}$ is then defined by drawing an edge from $i$ to $j$ whenever $T_{ij}\leq T_i$, giving edges $ij$ a weight or length $T_{ij}$.   If we infect a single vertex $v$ at time $0$, the time the infection has reached vertex $w$ is then equal to the length, $t_{v\to w}$, of the shortest oriented path from $v$ to $w$ in $\mathcal D_G^{SIR}$.  More generally, in a setting where we infect several vertices, $w$  will get infected at time $t=\min_i(t_i+ t_{v_i\to w})$ where the minium goes over all seeds, and $t_i$ is the time vertex $v_i$ gets infected (from an external source).  

\item When $G$ is random itself, it is often useful to couple the randomness of $G$ and that of the SIR dynamics in a way which more tightly represents the course of the infection.  In the setting of this section, where
$G\sim G(n,c/n)$, this can be done in several ways.  
\begin{enumerate}
    \item {\bf(Poisson Coupling)} Here one first defines a multi-graph analog of $G(n,c/n)$ where edges have a multiplicity given by a Poisson random variable $\Pois(c/n)$.   Next, when a vertex gets first infected, one  attaches 
    $d\sim \Pois(c)$ half-edges to the infected vertex, each running an infection clock at rate $\beta$, until the vertex recovers.  When an infection clock on a half-edge clicks, we stop the Poisson clock, chose an endpoint for the half-edge uniformly among all vertices, and if this endpoint happens to be susceptible, infect the endpoint.  After establishing all desired results for the Poisson random graph, one finally conditions on the Poisson graph being simple to get the desired results for  $G(n,c/n)$.  This approach is used, e.g., when deriving the differential equations \eqref{SIR-diff}, with $X(t)$ representing the number of infection clocks which are running at time $t$; see \cite{ball2020epidemics,borgs2024lawlargenumberssir} for details.

\item  \textbf{(Binomial Coupling to Susceptible Neighbors)} In our context, the Poisson coupling is less useful, since it does not generalize well in the presence of travel clocks.  Here the most natural coupling proceeds by drawing the neighbors of an infected vertex when it first gets infected.  More precisely, we will draw the neighbors  from all currently susceptible vertices, including each of them independently with probability $c/n$, then draw the recovery time $T_v$ for the infected vertex, and run infection clocks at rate $\beta$ along the edges to susceptible vertices until the vertex recovers.  Removing edges that used to point to susceptible vertices as they get infected, active edges, i.e., edges with a running $\beta$ clock, will always point to susceptible vertices, and when one of them clicks, we infect the new vertex and iterate.  
Note that in this setting, when a vertex gets infected at time $t$, it starts with $d\sim \Bin(S(t),c/n)$ active edges, and when the first edge of an infected vertex clicks, its endpoint is uniform among the susceptible individuals at that time.

\item \textbf{(Binomial Coupling to all Vertices)} A variant of the last coupling draws a larger set of edges when a vertex first gets infected, namely i.i.d. with probability $c/n$ from all $n$ vertices including thus the possibility of a self-loop to the infected vertex itself. Infections along active edges should then be considered infection attempts, and will be ignored if the endpoint of such an infection is already infected.  An advantage of this coupling is that we could first draw the number $d$ of active half-edges according to $d\sim \Bin(n,c/n)$, and only commit which edges we choose once the first infection clock clicks, at which point we choose $d$ endpoints uniformly without replacement.  In this way, the first infection attempt will be uniformly among all $n$ vertices, which turns out to be useful for some of our bounds.
\end{enumerate}
\end{enumerate}

\noindent The next lemma gives upper and lower bounds on the spread of the SIR epidemic
in the case of multiple initial ``seed infections'' appearing at times $t_v$ that are not necessarily all equal (as they were in the setting of Statement~\ref{bullet:LargeI0} of the last theorem). It will be one of the key ingredients we use to tie the development of the epidemic in the two communities together.

\begin{lemma}[Outbreak in a single community with multiple seeds]\label{lem:singlecomm-multipleseeds}
    Consider an SIR epidemic over $G\sim G(n,c/n)$ with infection rate $\beta$ and recovery rate $\gamma$, assume $R_0>1$, and  let $\delta>0$ $\eps\in(0,r_\infty)$, and $a\in (0,1)$ be fixed, i.e., independent of $n$.  Assume that  the SIR dynamic evolves as before, except that at times $0\leq t_1<t_2<\dots$ we seed additional infections at some vertices $v_1,v_2,\dots$ in $V(G)$, with the distribution of both the seeds and the seed times possibly coupled to the dynamics of SIR on $G(n,c/n)$.
    \begin{enumerate}
        \item\label{multi-seed-upperbd} 
        If all external seeds lie in a possibly random set
         $\mathcal S$ which is chosen independently of the draw $G\sim G(n,c/n)$ and the recovery and infection clocks for SIR, and if $\frac 1n|\mathcal S|\pto 0$, then w.h.p.
        $$\frac{R(\infty)}{n} \leq r_\infty +\delta
        $$
        \item\label{multi-seed-lowerbd} If a subset of the external seeds are chosen i.i.d. uniformly at random in $[n]$ and by time $T$ there are at least $\omega_n\to\infty$ of these uniformly chosen, external seeds, then w.h.p.
        \begin{equation}\label{eq:seed_lemma_2}
            \tau(\eps) \le T + (1+\delta)\frac{\ln n}{\lambda},\quad
            \tau(n^{a-1}) \le T + (1+\delta)\frac{a\ln n}{\lambda},
            \quad\text{and}\quad \frac{R(\infty)}{n} \geq r_\infty -\delta.
        \end{equation}
    \end{enumerate}
\end{lemma}
\begin{proof}
The first statement is a direct consequence of the representation of the epidemic in terms of the infection digraph $\mathcal D_{G}^{SIR}$
and the observation that in this representation, the final set of infected vertices is just the set of vertices that can be reached from the set of seeds.  Since this set does not depend on the infection times $t_v$, the statement follows from statement \ref{bullet:LargeI0} of the last theorem.

Next we note that if $\omega_n\to\infty$ of the vertices are infected by time $T$, that for any finite $M$, at least $M$ of them are infected by time $T$ (assuming $n$ is sufficiently large).  Thus all we need to prove is that for any given $\delta'>0$ and $M$ large enough, the probability that $\tau(\epsilon)\leq T+ (1+\delta)\ln n /\lambda$ is at least $1-\delta'$.
Let $\mathcal{A}_i$ be the event that the SIR epidemics with a single seed $v_i$ at time $t_i$ would satisfy statement \ref{eq:seed_lemma_2} in the lemma, and  let $\mathcal{B}_i$ be the event that the SIR epidemics with seed  $v_i$   would terminate before infecting $L$ individuals for some $L<\infty$ to be specified. Note that in the above coupling where we pre-sample the infection and recovery clocks,  we can ensure that $\bigcup_{i=1}^M \mathcal{A}_i$ implies the desired events: if any seed infection is capable of producing a large outbreak, then having additional seeds of infection only expedites the spread. Thus, it suffices to bound $\Pr(\bigcap_{i=1}^N \mathcal{A}_i^c)$ to be at most $\delta'$ for all $\delta'>0$.

Using inclusion-exclusion, we first bound
\begin{equation}\label{multi-seed-bound}
    \begin{aligned}
        \Pr\Big(\bigcap_{i=1}^M \mathcal{A}_i^c\Big) &= 
        \Pr\Big((\mathcal B_1\cap \mathcal A_1^c)\cap\bigcap_{i=2}^M \mathcal{A}_i^c\Big) +  \Pr\Big(({\mathcal B}_1^c\cap\mathcal A_1^c)\cap\bigcap_{i=2}^M \mathcal{A}_i^c\Big) \\
        &\leq   \Pr\Big(\mathcal B_1\cap\bigcap_{i=2}^M \mathcal{A}_i^c\Big) +  \Pr\Big(\mathcal{A}_1^c\setminus \mathcal B_i\Big)\\
        &\leq\cdots\leq  \Pr\Big(\bigcap_{i=1}^M \calB_i\Big) + \sum_{i=1}^M\Pr\Big(\calA_i^c\backslash\calB_i\Big),
    \end{aligned}
\end{equation}
and then note that by statements \ref{bullet:single_comm_final_size}, \ref{bullet:single_comm_init_growth}, \ref{bullet:final-time} and \ref{bullet:single_comm_tau_concentration} of the previous theorem, we can make $\Pr(\calA_i^c\backslash\calB_i)$ as small as desired by choosing $L$ and $n$ sufficiently large. Thus given $M$, for $L$ and $n$ sufficiently large, we can bound the second term on the right  by at most $\delta'/2$.

To bound the first term, we
first observe that for the first seed, the bad event $\mathcal B_1$ is just the event that $R(\infty)<L$ on the graph $G(n,c/n)$ with a single random seed.  By the fact that conditioned on not having a large outbreak, 
$R(\infty)$ is bounded in probability, and the fact that the probability of a large outbreak converges to $\pi$ in probability, we have that for $L$ and $n$ sufficiently large, $\Pr(R(\infty)<L)$ can be bounded by, say, $1-\pi/2$.  It turns out we need slightly more, namely that
\begin{equation}\label{Rinfty-on-Gn'}
\Pr(R(\infty)<L)\leq 1-\frac \pi 2\quad\text{on}\quad G(n',c/n)\quad\text{whenever} \quad n'\geq n-\ln n
\end{equation}
and $L$ and $n$ are sufficiently large.  To see that this holds, we first note that the probability $\Pr(R(\infty)<L)$ is decreasing in $n'$, showing that $n'=n-\ln n$ is the worst case.  The proof now follows once we use that $G(n',c/n)=G(n',c'/n')$ with $c'=cn'/n$, together with the fact that $\pi(c',\beta,\gamma)$ is continuous in $c'$.

To prove an upper bound on the first term in \eqref{multi-seed-bound}, we now use the following coupling in which we sequentially expose the set of vertices
$C_+(v_i)$ that can be reached from the seed $v_i$ 
in $\mathcal D_G^{SIR}$ by the usual breadth first search starting from $v_i$. In this process, whenever a vertex $v$ gets infected, we first draw its recovery clock, then draw the edges to the remaining susceptible vertices i.i.d. with probability $c/n$, and then draw the infection clocks for these edges, keeping only those edges for which the infection clock clicks before the recovery clock, and declaring the end point infected. We declare success if at any point in this process, we have discovered at least $L$ infected vertices. We declare failure at step $i$ if the process dies before infecting $L$ vertices; in this case, we conclude that $|C_+(v_i)|<L$, declare the vertices in $C_+(v_i)$ used, and move on to the next seed $v_{i+1}$. Furthermore, since we are interested in an upper bound on the probability of failure event 
$\mathcal B=\bigcap_{i=1}^M \calB_i$, we will augment the definition of failure by
\begin{enumerate}
    \item adding an additional $k=L-|C_+(v_i)|$ random vertices to the set of used vertices and declaring failure if the new seed $v_{i+1}$ falls into the set of used vertices
    \item declaring failure if the BFS process starting from $v_{i+1}$ discovers one of the used vertices.
\end{enumerate}
In this way, the failure probability in step $i$ (assuming failure in all previous steps) can be bounded from above by $\frac{(i-1)L}n$ plus the probability that a BFS exploration on $\mathcal D^{SIR}_{G(n-(i-1)L,c/n)}$ fails before reaching $L$ vertices, which in turn is nothing but the probability that $R(\infty)<L$ on $G(n-(i-1)L,c/n)$.
Combined with the bound \eqref{Rinfty-on-Gn'}, we see that
$$\Pr\Big(\bigcap_{i=1}^M \calB_i\Big)\leq \prod_{i=1}^M\Big(1-\frac\pi 2+\frac{(i-1)L}n\Big)\leq \exp\left(-\sum_{i=1}^M\Big(\frac\pi 2- (i-1)\frac Ln\Big)\right)\leq e^{-M\pi/3}
$$
provided $\frac{LM}n\leq \pi /6$.
Choosing first $M$ large enough so that the right hand side is smaller than $\delta'/2$, and then $L$ and $n$ large enough to guarantee that the bound \eqref{Rinfty-on-Gn'} holds and
the second term in \eqref{multi-seed-bound} is smaller than $\delta'/2$ and finally increasing $n$ if needed to make sure that $\frac{LM}n\leq \pi /6$ and $ML\leq \ln n$, we obtain the desired bound $\Pr(\bigcap_{i=1}^N \mathcal{A}_i^c) \le \delta'$.
\end{proof}

We will also need the following lemma, which states that a large fraction of the early infected nodes will be infected for a substantial time. This fact will be used to show that
once the infection in the first community has grown to size roughly $n^\alpha$, with high probability it will seed enough infections in the second community to apply Lemma~\ref{lem:singlecomm-multipleseeds}.

\begin{lemma}\label{lem:emprical-distribution-Gnp}
Consider an SIR epidemic over $G\sim G(n,c/n)$ with infection rate $\beta$ and recovery rate $\gamma$, assume $R_0>1$,  let $a\in (0,1)$, let $0<\delta<a$, and let $x$ be an arbitrary positive number. Conditioned on a large outbreak, we have that w.h.p. $\tau(n^{a-1}) < \infty$ and among these first $\lceil n^a \rceil$ infected nodes, at least $n^{a - \delta}$ will have infectious periods of at least $x$ (i.e., having a recovery time at least $x$). 
\end{lemma}

\begin{proof}
    We will want to prove a lower bound on the conditional probability of the event $\calE$ that $\tau(n^{a-1}) < \infty$ and among these first $\lceil n^a \rceil$ infected nodes, at least
    $n^{a - \delta}$ will have infectious periods of at least $x$, where we condition on the event $\calE_0$ of a large outbreak.

    It is helpful to consider an augmented process where a fresh new node will be infected as a new seed whenever the SIR process dies out (and this gets repeated). More precisely, whenever the number $I(t)$ of actively infectious vertices drops to zero, if $S(t) > 0$, we pick a vertex $v$ among the remaining susceptible ones uniformly at random, and infect it at time $t+1$. This way the process runs until the entire population have been infected once, and we can be sure to reach $n^a$ number of infections eventually. The true SIR process is simply the stopped process on the augmented process upon first extinction (i.e., number of active infections dropping to zero).

    In the context of the augmented process, denote by $\tilde \tau(\eps)$ the stopping time for an $\eps$-fraction infection, analogous to our $\tau(\eps)$ in the original process. Note that $\tilde\tau(n^{a - 1}) < \infty$ a.s. Denote by $\tilde \tau_0$ the first time infections drop to zero. Per our observation above, our original SIR process is simply the augmented process stopped at $\tilde\tau_0$. Further,  the coupling between the two process is exact until $\tilde\tau_0$; in particular, whenever $\tilde\tau(n^{a - 1}) < \tilde\tau_0$ $\tilde\tau(n^{a - 1}) = \tau(n^{a - 1})$ and those nodes infected by then share identical lives in these two version.
    
    It now suffices to show the following event $\tilde\calE$ happens whp in the augmented process: among the first $\lceil n^a \rceil$ infected nodes, at least $n^{a-\delta}$ will have infectious periods at least $x$. To derive our target statement in the lemma, the event $\tilde\calE$ along with $\tilde\tau(n^{a - 1}) < \tilde\tau_0$ implies that $\tau(n^{a-1}) = \tilde\tau(n^{a-1}) < \infty$ and, as these $\lceil n^a\rceil$ nodes have the same infection periods in the original SIR as in the augmented process, at least $n^{a-\delta}$ of them have infectious periods at least $x$ in the original SIR process. Thus, the probability of our original event is lower bounded by $\Pr(\tilde\tau(n^{a - 1}) < \tilde\tau_0, \tilde\calE) \ge \Pr(\tilde\tau(n^{a - 1}) < \tilde\tau_0) - \Pr(\tilde\calE^c)$, which has a limit $\pi$ once we show $\Pr(\tilde\calE^c) \to 0$. Since the probability of the event $\calE_0$ tends to $\pi$ as well, this would imply that $\Pr(\calE\mid\calE_0)=\Pr(\calE \cap \calE_0)/\Pr(\calE_0)=\Pr(\calE)/\Pr(\calE_0)\to 1$ as desired.

    It remains to understand the event $\tilde\calE$. By our definition of the infection process, each node has an $\Expo(\gamma)$ recovery clock independent from all other clocks; in other words, when a node is infected, no matter what past event we condition on (including other infection clocks we have seen, past extinction events, etc.), its recovery time is $\Expo(\gamma)$ distributed. Thus, the recovery times for the first $\lceil n^a \rceil$ infected nodes form a sequence of $\lceil n^a \rceil$ i.i.d. $\Expo(\gamma)$ samples\footnote{Another way to see this is to presample an infinite sequence of i.i.d. recovery time.  Each time a new vertex gets infected, it just chooses the next recovery time on the list, which means the sequence we see is an i.i.d. sequence.}. By Hoeffding's inequality, 
    \[
    \Pr(\tilde\calE^c) = \Pr_{N\sim\Bin(\lceil n^a \rceil, e^{-\gamma x})}(N < n^{a-\delta}) \le \exp\Big(-\frac{2 (\lceil n^a \rceil e^{-\gamma x} - n^{a-\delta})^2}{\lceil n^a \rceil}\Big),
    \]
    which tends to zero as $n\to \infty$ for any constant $x > 0$. This finishes our proof.
\end{proof}

\subsection{Branching Process Upper Bounds}\label{sec:BP-results-summary}
In addition to the properties of SIR on a single community random graph $G(n,p)$, we will rely on the following branching process upper bounds, to be established rigorously in Section~\ref{sec:BP-upper-bounds}.
These reflect the following intuition.  When the epidemic originally spreads,  all individual in $V_1\cup V_2$ are susceptible, and can be infected.  As the epidemic evolves, less and less individuals are susceptible, and hence can be infected.  One might thus hope that by replacing the number of susceptible individuals in each community by $n$, one gets upper bounds on the speed at which the epidemics spreads, suggesting that the following process gives upper bounds.

Starting with one infected individual in community $1$ (which w.h.p, will not be traveling, and hence can be assume to be in $V_1$), we explore its neighbors in our random, two community network as individuals get infected, giving $\Pois(c_A)$ neighbors of travel type $A\in\{H,T\}$ where $c_A$ is equal to $c$ tines the stationary probability $p_A$ of being in travel state $A$, independently for both travel types, and for the possible types $i\in\{1,2\}$.  Since these may or may not be infected later on, we think of them as possible, or not yet {\color{blue} born} children in the branching process.  With this in mind, we have the following continuous time branching process, with 
individuals have labels $(i,A)\in \{1,2\}\times \{H,T\}$:
    \begin{itemize}
        \item At birth, each individual has $\Pois(2c_A)$ possible children that start in the travel state $A$, independently for $A=H$ and $A=T$, all of them having type $1$ and $2$ i.i.d. with equal probability
        \item Both the mother and the possible children change their travel state according to the rates $\rho_T$ and $\rho_H$
        \item {Each potential child is linked to their mother by an edge.  This edge is considered active when mother and  child are in the same physical community, i.e., they either have the same community label and are in the same travel state, or they have different community labels and different travel states}
        \item {Potential children connected to their mother via active edges are born at rate $\beta$, starting their new life in the travel state they had just before been born}
        \item At rate $\gamma$, the mother dies, and with her all her not yet born children.
    \end{itemize}
{The active edges between mothers and their not yet born children in this continuous time branching process then correspond to what were defined as the active edges 
between infected and susceptible individuals in the main text. }

One might hope that the number of individuals of type $i$ born at or before time $t$ in this branching process is an upper bound on $I_i(t)+R_i(t)$.  As it turns out, this is true, provided we make the following adjustments:  
\begin{enumerate}[label=(\alph*)]
     \item we restrict ourselves to times $t$ which are not too large (say $t\leq T=\ln^2n$, a time we will later prove to be way past the time when the epidemic has died out anyhow).  
     \item  rather than choosing $c_A=c p_A$, we replace these by the upper bounds $c_H=c$ and $c_T=n^{-\alpha}\ln^3 n$ (adding an extra power of $\ln n$ for ``safety'').
    \item instead of requiring upper bounds which hold w.h.p., we ask for upper bounds which hold in expectation, i.e., we ask that $\E[I_i(t)+R_i(t)]$ is bounded by the {\it expected} number of individuals of type $i$ born at or before time $t$ in this branching process; 
 \end{enumerate}
Using this approach, we will then be able to prove the following theorem, which will be the main input for our proofs, in addition to the one community results from the last section.

\begin{proposition}\label{prop:BP-upper-bound}
If $t=t(n)\leq\ln^2n$, the following statements hold w.h.p.
    $$
    I_1(t)+R_1(t)= \tilde O\left(e^{\lambda t}\right)
  \quad\text{and}\quad
    I_2(t)+R_2(t)=\tilde O\left(n^{-\alpha}e^{\lambda t}\right).
    $$
\end{proposition}

\begin{proof}
    We defer the proof to Section~\ref{sec:BP-upper-bounds}.
\end{proof}
Finally, we will use the following proposition, which we will prove using the technology developed in \cite{borgs2024lawlargenumberssir}, using the intuition that past herd immunity, the epidemic should die out fast.  The proof will also involve a branching process upper bound, to be established in Subsection~\ref{sec:Herd-Immunity} of Section~\ref{sec:BP-upper-bounds}.

\begin{proposition}\label{prop:herd-immunity} Assuming $R_0>1$, define a herd immunity threshold $s_0$ by $s_0R_0=1$, and let $0<\delta<s_0$.
Choose $t_0=t_0(n)<\frac 12\ln^2 n$  such that the probability of the event
$$
\calE_0=\big\{\max\{S_1(t_0),S_2(t_0)\}\leq (s_0-\delta)n\big\}
$$
is bounded from below uniformly in $n$.
Then there exists a constant $C$ (depending on $\delta$ and the parameters of the model) such that conditioned on  $\calE_0$, with high probability
the infection will die out by time $t_0+C\ln n$, i.e., 
$$
\Pr\big(I(t_0+C\ln n)=0\;\big|\;\calE_0\big)\to 1 \quad\text{as}\quad n\to\infty.
$$
\end{proposition}

\begin{proof}
    We defer the proof to Section~\ref{sec:BP-upper-bounds}.
\end{proof}

\section{Travelers vs Non-Travelers}

On an intuitive level, all our results should follow from the fact that we can divide the set of individuals into a large majority of non-travelers, with the epidemic growth mainly driven by the growth in the two resulting sub-communities, and a smaller population which travels and transmits the epidemic between communities.  It turns out that most of our proofs follow this intuition, except for some of our upper bounds, which require a more detailed understanding of the inter-community interactions and will be carried out with the help of suitable multi-type branching process approximations, see the remarks in the last section.

To translate this intuition into a rigorous proof, we consider the set of non-travelers up to time $T$ at which point we expect the epidemic to have died with high probability.  Specifically, we set 
 $T := \ln^2 n$ and then define the random sets
\[
   \mathcal{T}_i = \{v \in V_i \mid v \text{ travels at least once during } [0,T]\},
\]
and
\[V_i^{\static}=V_i\setminus {\calT}_i,\quad
V^{\static}=V_1^{\static}\cup V_2^{\static}
\quad\text{and}\quad \calT=\calT_1\cup\calT_2,\]
as well as the random variables
$I_i^{\static}(t)$ and $R_i^{\static}(t)$,
denoting the number of infected and recovered individuals in $V_i^{\static}$, respectively, using
\begin{equation}\label{Yi-definition}
 Y_i^\static(t)=I_i^{\static}(t) + R_i^{\static}(t)   
\end{equation}
for their sum. We refer to individuals in  $\calT$ as travelers, and those in $V^\static$ as static.

The following lemma says that whp, the set of travelers contains only a small proportion of all individuals. In addition it  gives lower bounds on the probability of two particular events, $\calE_{HT}(v,t)$ and  $\calE_{TH}(v,t)$, where  $\calE_{HT}(v,t)$ is the event that an individual $v$ is home at time $t$, stays home for at least one and at most two time units, and then travels for at least one time unit, with  $\calE_{TH}(v,t)$ defined by interchanging the roles of traveling and staying home. To motivate the definition of these events, we note that for a traveler to transport the infection from one community to another, it must either be  home long enough after one of its contacts in its resident community to gets infected, then travel soon enough to still be infectious, and finally travel long enough to infect someone in the host community, or it must travel long enough to get infected, and return home soon enough (and for enough time) to infect someone in their home community.

\begin{lemma}\label{lem:few-travelers}
With high probability, 
\begin{equation}
    \label{eqn_travelers_bound}|
    {\calT_i}|
    =np_T(1+\rho_H\ln^2 n)]\left(1+\tilde O(n^{-\tilde\alpha})\right)
    \leq \frac 12n^{1-\alpha}\ln^3 n
    \end{equation}
    where $\tilde\alpha=\min\big\{\alpha,\frac 12(1-\alpha)\big\}$.  
Furthermore, for all $t\in [0,T-3]$, the  conditional probabilities of the two events  $\calE_{HT}(v,t)$ and  $\calE_{TH}(v,t)$ are
\begin{align}
    \Pr(\calE_{HT}(v,t)\mid v\in\calT_i)
    &=\frac {n p_He^{-\rho_T}(1-e^{-\rho_T})e^{-\rho_H}}{\E[|\calT_i|]}=\Theta\left(\frac 1{\ln^2n}\right)
\\
    \Pr(\calE_{TH}(v,t)\mid v\in\calT_i)
    &=\frac {np_Te^{-\rho_H}(1-e^{-\rho_H})e^{-\rho_T}}{\E[|\calT_i|]}=\Theta\left(\frac 1{\ln^2n}\right)
\end{align}
 \end{lemma}
\begin{proof}
Each individual in $V_1\cup V_2$ has probability
    \begin{equation}
p_{\geq 1}=p_T+(1-p_T)(1-e^{-\rho_T \ln^2 n})=p_T(1+ \rho_H\ln^2 n)(1+\tilde O(n^{-\alpha}))
\end{equation}   
of traveling at least once in the time interval $[0,T]$, leading to 
$$
\E[|\calT_i|]= np_T(1+\rho_H\ln^2n)(1+\tilde O(n^{-\alpha}))
$$
Using that $\E[|\calT_i|]= \Theta(n^{1-\alpha}\ln^2n)$
\eqref{eqn_travelers_bound} follows using a  multiplicative Chernoff  bound.

To prove the second, we note that the unconditional probability of being home at time $t$ is $p_T$, the probability of then traveling for the first time between $t+1$ and $t+2$ is equal to $\int_1^2e^{-\rho_Ts}\rho_T \, ds=e^{-\rho_T}(1-e^{-\rho_T})$, and the probability of not returning home for at least one time unit is $e^{-\rho_H}$, showing that
\[  
  \Pr(\calE_{HT}(v,t))= p_He^{-\rho_T}(1-e^{-\rho_T})e^{-\rho_H}.
\]
Since the event $\calE_{HT}(v,t))$ implies that $v$ is traveling at least once between $0$ and $T$, the conditional probability is obtained by dividing this expression by 
$p_{\geq 1}=\E[|\calT_i|]/n$.  The  expression for the conditional probability of the event  $\calE_{TH}(v,t))$ is derived in the same way.  Finally, the claimed order of magnitude for both conditional expectations follows by the expression $p_{\geq 1}$ above and the observation that
$p_H=\Theta( \rho_H)=\Theta(1)$ and $p_T=\Theta(\rho_T)=\Theta (1-e^{-\rho_T})$
\end{proof}

\subsection{Exact coupling for the early stages of the epidemic}

Our first result concerns the early stages of the epidemic, which, as it turns out, is with high probability driven by the spread of the epidemic in 
the community of non-travelers of type 1.  This is the content of the next proposition.

\begin{proposition}\label{pro:small-time-coupling}
Let $N_n$ be such that $N_n(n^{-1/2}+n^{-\alpha}\ln^3 n)\to 0$ as $n\to\infty$.
 Then the two community model with travel can be coupled to $\BPsub{c}(t)$ such that
 w.h.p.
 \begin{equation}\label{eq:size_of_tree_coupling}
     I(t)+R(t)=I_1^{\static}(t)+R_1^{\static}(t)=|\BPsub{c}(t)| \quad\text{for all $t$ such that} \quad|\BPsub{c}(t)| \leq N_n.
 \end{equation}

   As a consequence, the probability of a large outbreak in community $1$ converges to $\pi$, and conditioned on not having a large outbreak, the following statements are true
    \begin{enumerate}
        \item the time of an outbreak, as well as its total size $R_1(\infty)+R_2(\infty)$ are bounded in probability
        \item whp., $R_2(\infty)=0$ and all vertices contributing to $R_1(\infty)$ recover before traveling
    \end{enumerate}
\end{proposition}

\begin{proof}[Proof of Proposition~\ref{pro:small-time-coupling}]
To prove the proposition, we couple the exploration of the two community network to the SIR dynamics as follows.  First, we draw all the travel clocks, determining in particular the set ${\calT}_1\cup {\calT}_2$ of travelers up to time $T=\ln^2 n$.  Next, 
starting from a single infected individual $v_1$ in $V_1^H\cup V_2^T$ (which whp will be in $V_1^{\static}$), whenever a new vertex $v_j$ gets infected (where we order the vertices according to their infection time), we draw its recovery time and explore  the edges from $v_j$ into the set of not yet infected vertices, noting that
whenever we infect a new vertex $v_j$, the ``forward'' degrees of $v_j$  into the not yet infected vertices in $V_1$ and $V_2$ have distributions $\Bin(n_1,c/n)$ and $\Bin(n_2,c/n)$ where $n_1$ and $n_2$ are the number of not yet infected vertices in $V_1$ and $V_2$, respectively.  As a consequence, these forward degrees can be coupled to two independent $\Pois(c)$ random variables in such a way that the Binomial and Poisson versions are equal with probability
$1-O(j/n)$. Finally, we start independent infection clocks with rate $\beta$ on the newly discovered edges, pausing them whenever the two endpoints are not in the same travel state (and stopping them when $v_j$ recovers).  

Defining the stopping time $\tau$ as the first time one of the following events happens:
\begin{enumerate}
    \item the number of ``children'' of $v_j$ in the exploration process is different from $\Pois(c)$ in either community
    \item $v_j\in{\calT}_1\cup\calT_2$.
\end{enumerate}
We see that up to the stopping time $\tau$, the infection tree is equal to $\BPsub{c}(t)$.
Furthermore, for $k=o(\sqrt n)$, w.h.p. none of the first type of events happens before the process reaches size $k$, and if $k=o(n^{\alpha}\ln^{-3} n)$, w.h.p., none of the second type happens before the process reaches size $k$. This proves statement \ref{eq:size_of_tree_coupling}.  The next statement (on the probability of a large outbreak) follows by setting $N_n=\ln n$ and observing that $\Pr(|\BPsub{c}(t)|\geq \ln n)\to \pi$. Conditioned on not having a large outbreak, statement 1 follows from observing that conditioned on $\BPsub{c}(t)$ not surviving, it lives for a time that is bounded in probability, and reaches a final size that is bounded in probability.
Since this implies that with high probability the branching process dies before time $T$,
we see that w.h.p.,, all vertices which ever get infected (which w.h.p. are in $V_1^{\static}$) recover before ever traveling, proving statement 2.
\end{proof}

\subsection{Early and late stage upper bounds in the two communities}\label{sec:upper-bd-Rinfty}

Next we derive an upper bound on the size of the epidemic up to time $T$, noting that the branching process upper bounds from Proposition~\ref{prop:BP-upper-bound} are too crude to be useful at $T=\ln^2n$.  We  will also strengthen the upper bounds for small $t$, showing that for $t\leq \frac{\alpha-\delta}\lambda \ln n$, w.h.p. the infection never leaves the set of non-travelers $V_1^\static$, while Proposition~\ref{prop:BP-upper-bound} only shows that it never leaves the set $V_1$.

To do so, we now couple the epidemic on the full graph to the epidemics on the communities of non-travelers as follows.  
First, we draw all travel clocks, determining in particular the sets $V_1^\static$ and $V_2^\static$.
Next we draw the two community network, noting that the induced graph on, say, 
$V_1^\static$ is nothing but a copy of $G_1'\sim G(n',c'/n')$, where $n'=|V_1^\static|$ and $c'$ is defined by $c'/n'=c/n$, and similarly for $V_2^\static$ (with  $n''=|V_2^\static|$).
Finally, for each individual $v$, we draw a recovery clock $T_v$ and for 
each directed edge $(u,v)$ in the 2-community network we draw
an infection clock $T_{u,v}\sim \Expo(\gamma)$.

Note that the travel clocks introduce some subtleties here: when we pre-draw a time $T_{u, v}$ with either $u$ or $v$ in the set of travelers ${\calT}_1\cup\calT_2$, we have to take into account that the original clock was only active when both endpoints were physically in the same community. In order to determine if an infection is passed from $u$ to $v$ and when, we therefore have to look at the whole travel history of the two vertices and the time $t_u$ when $u$ was initially infected.  Specifically, we take $T_{u, v}$ as the time-budget that we use whenever the edge $uv$ is active (starting from the time $t_u$). Let $t_{u, v}$ be the time at which the time budget is exhausted (this, by definition, will be a time when $uv$ is active). If $u$ has not recovered by time $t_{u, v}$, i.e., if $t_{u, v}\leq t_u+T_u$, the edge transmits at time $t_{u, v}$, otherwise it doesn't.  The above process has to be done iteratively, starting by setting $t_u=0$ for the first infected individual and $t_v=\infty$ for all others, updating the infection times as we carry out the above steps for the infected vertices, one by one as they get infected.

Considering the epidemic as experienced by individuals in $V_1^\static$, once the infection has reached $V_1^\static$, the epidemic will spread inside $G_1'$ via edges into the set of currently susceptible individuals in $V_1^\static$, while injecting seeds into the external directed network along the edges from $V_1^\static$ to ${\calT}_1\cup\calT_2$ (note that there will be no active edges between the non-travelers in $V_1^\static$ and $V_2^\static$ until time $T$). The epidemic inside  $G_2'$ proceeds by a similar process.

\begin{proposition}\label{prop:Rinfty-upper-bd}
    Fix $\delta>0$,  let  $t_-= \frac{\alpha-\delta}\lambda\ln n$, and as before let $T=\ln^2n$.  Then w.h.p., the following statements hold: 
    \begin{enumerate}
        \item at time $t_-$, $I(t_-)+R(t_-)$ consists only of vertices in $V_1^\static$.
        \item $R_1(T)+I_1(T)\leq (r_\infty+\delta)n$, and
     $R_2(T)+I_2(T)\leq (r_\infty+\delta)n$.
    \end{enumerate} 
  \end{proposition}
\begin{proof}
To prove the first statement, define a stopping time $\tau$ to be the first time an individual outside of $V_1^\static$ gets infected.  Up to the stopping time, the epidemic in 
 $V_1^\static$ grows at most as fast as the SIR epidemic on $G(n,c/n)$, showing that if we consider the time $\tilde t_-=\min\{\tau,t_-\}$ and an arbitrary
 $\delta'>0$, then 
 w.h.p. $I_1^\static(\tilde t_-)+R_1^\static(\tilde t_-)\leq e^{\lambda (t_-+\delta')}$.
 But each individual contribution to $I_1^\static+R_1^\static$
has only 
$\Bin(|\calT|,c/n)$ neighbors outside of  $V_1^\static$ it can infect, showing that up to the time
$\tilde t_-$, the probability of an infection leaving  $V_1^\static$ is bounded by
$\tilde O(n^{-\alpha} e^{\lambda (t_-+\delta')})$ which goes to zero if $\delta'$ is chosen small enough.
This completes the proof of the first statement.

To prove the second statement, we note that 
with probability tending to $1$,  $I_i(T) + R_i(T)\leq Y_i^\static(T) + n^{1-\alpha}\ln^3 n$ where $Y_i^\static$ is defined in \eqref{Yi-definition}. As a consequence, it is enough to prove the upper bound for 
$Y_i^\static$, which can be obtained by considering an epidemic on
$G_i' \sim G\bigl(n', c'/n'\bigr)$ with external seeds emanating from infected vertices in ${\calT}_1\cup\calT_2$.
Since $I_i^{\static}(T)+R_i^{\static}(T)$ does not depend on the evolution of the SIR dynamics past time $T$, we can further assume that $G_i'$ does not receive further seeds after $t=T$, which allows us to bound $I_i^{\static}(T)+R_i^{\static}(T)$ by $R(\infty)$ on $G_i'$.
This puts us into the setting of Lemma~\ref{lem:singlecomm-multipleseeds}, with $c$ and $R_0$ replaced by
$c'=c\frac n{n'}$ and $R_0'=R_0\frac n{n'}\leq R_0$, implying in particular that $r_\infty'\leq r_\infty$.  To apply the lemma, we note that each vertex in $V_i^\static$ has $\Bin(|\calT_i|,c/n)$ edges joining it to 
$\calT_i$, showing it has probability $q=1-(1-c/n)^{|\calT_i|}=\tilde O(n^{-\alpha})$ of being part of the seed set, which in turn shows that the seed set is $o(n)$ with high probability.
\end{proof}
\begin{remark}\label{rem:tau12-lower}
In the main part of the paper, we defined the stopping time $\tau_{1\to 2}$ as the first time when  $I_2(t)+R_2(t)$ is non-zero, i.e.,  the first time an individual {\it with label 2} gets infected.  As pointed out in the main text, a natural, alternative definition would say that  $\tau_{1\to 2}$ is the first time a vertex located in community 2 (whether an individual with label 2 that is not traveling, or an individual with label 1 that is currently visiting community 2) gets infected.  The above proposition implies that $t_-$ is a lower bound for $\tau_{1\to 2}$ for both definitions.
\end{remark}

\subsection{Lower bounds in community 1} 

Let $T$, $\mathcal{T}$,  $V_i^{\static}$,  $I_i^{\static}(t)$,
$R_i^{\static}(t)$ and $Y_i^\static(t)$ be as in the last section, and let
\begin{equation}
\label{t12-in-community1}
\tau_i^{\static}(\epsilon)=\inf\{t\geq 0:Y_i^0(t)\geq \epsilon n\}.
\end{equation}
In this section we prove the following proposition.

\begin{proposition} \label{prop_lower_bounds_comm1}
    Consider the setting of Theorem 2.2 in the main text, and assume that $R_0>1$. Fix $\epsilon \in (0,r_\infty)$, $\delta, \delta'\in (0,1)$ and let
    \[
    \tau^+=({1+\delta})\frac {\ln n}\lambda\quad\text{and}\quad \tau^+_{1\to 2}=(1+\delta)\frac {\alpha \ln n}\lambda.
    \]
With probability at least $\pi - \delta'$, the following holds provided $n$ is sufficiently large
\begin{enumerate}
\item   $\tau_1(\eps)\leq \tau_1^{\static}(\epsilon)\leq \tau^+$ and
        $R_1(\infty)\ge Y_1^\static(\tau^+)\ge (r_\infty - \delta) n$
\item
       $       Y_1^{\static}(\tau^+_{1\to 2} )\geq n^{\alpha(1+\delta/4)}$,
       with at least $n^{\alpha(1+\delta/8)}$ of the individuals in $V_1^\static$ contributing to $Y_1^{\static}(\tau^+_{1\to 2} )$
       having a recovery time at least $1$.
\end{enumerate}
\end{proposition}

\begin{proof}
As in the proof of Proposition~\ref{prop:Rinfty-upper-bd}, we couple to the static subgraph process on $G_1'\sim G(n', c/n')$, but this time, we will 
derive \emph{lower bounds} by a coupling to what we call a \emph{restricted epidemic process} on $V_1^{\static}$ by disallowing any edges or infections involving $\calT=\calT_1\cup \calT_2$ during $[0,T]$. To motivate this coupling, we note that
throughout $[0,T]$, the epidemic on $V_1^{\static}$ in the full travel-based setting is stochastically \emph{at least as fast} as it is in the single-community setting on $G_1'$, because:
\begin{itemize}
  \item In the actual process, the nodes in $V_1^{\static}$ can still gain exposure from infected travelers (type-1 or type-2) who move into community~1. Thus, the infection can only spread earlier or faster in the real process.
  \item Removing the travelers (i.e., ignoring any infections they cause or receive) can only reduce the outbreak size and slow the onset of infections, giving a lower bound for the real process.
\end{itemize}
Because edges among $V_1^{\static}$ are exactly those in $G_1'$, and these edges remain active for times $[0,T]$, the restricted process is precisely the single-community SIR on $G_1'$. By monotonicity of SIR processes under edge removals, the actual epidemic on $V_1$ (and hence community 1) is stochastically at least as large---i.e., at any fixed time $t \le T$, the number of infections in $V_1^{\static}$ under the full model is at least the number of infections under the restricted single-community epidemic.

From here on the proof is straightforward: as before, let $n' = |V_1^{\static}|$, and let 
\[
   E_{\mathrm{travel}} =\big\{|\mathcal{T}| \le 
  n^{1-\alpha} \ln^3 n \big\}
\]
be the high-probability event from Lemma~\ref{lem:few-travelers}. 
Conditioned on $ E_{\mathrm{travel}}$, we have $n' \ge n - n^{1-\alpha} \ln^3 n$. Since we assume $\alpha > 0$, the ratio $n'/n \overset{p}{\to} 1$ as $n\to\infty$. Furthermore, conditional on $V_1^{\static}$, the induced contact network on these non-traveling nodes is an Erd\H{o}s--R\'{e}nyi graph
\[
   G_1' \sim G\bigl(n', c'/n'\bigr)
\]
with $c' = cn'/n$
since edges among these vertices remain active throughout $[0,T]$, and there are no new edges formed via travel into $V_1^{\static}$.
The standard SIR epidemic over $G(n', c/n')$ then has parameters $R_0' = \frac{c'\beta}{\beta + \gamma}1$ and
$\lambda' = c'\beta - \beta - \gamma$.
Noting that $R_0'\overset{p}{\to} R_0$, 
$\lambda'\overset{p}{\to} \lambda$, we have that the infection probability and final size converge as well,
$r_\infty'\pto r_\infty$ and $ \pi'\pto \pi$ due to the continuity of $r_\infty$ and $\pi$ in $R_\infty$ and $c$.
By the properties of the standard SIR epidemics over $G(n', c/n')$
summarized in Theorem~\ref{thm:single-comm} and Lemma~\ref{lem:emprical-distribution-Gnp}, the following statements then hold with probability at least $\pi' - \delta'/2$ if $n$ is sufficiently large:
\begin{enumerate}
      \item  Let 
$\epsilon'=\epsilon n/n'$ and let      
    $\tau_1'(\eps')$ be the first time the infected set in $G_1'$ reaches $n'\epsilon'=n\epsilon$ many vertices.  Then
    \[
      \tau_1'(\eps') \;\le\; (1+\delta/4)\,\frac{\ln n'}{\lambda'}\leq \tau^+.
    \]
\item Similarly, setting $\alpha'=\alpha (1+\delta/3)$, we have
  \[
      \tau_1'((n')^{\alpha'-1}) \;\le\;  (1+\delta/4)\,\frac{\alpha'\ln n'}{\lambda'}\leq  (1+3\delta/4)\,\frac{\alpha\ln n'}{\lambda'}
     \leq \tau^+_{1\to 2} .
    \]
and  hence
\[ I_1^{\static}(\tau^+_{1\to 2} )+R_1^{\static}(\tau^+_{1\to 2} )\geq (n')^{\alpha'}\geq
n^{\alpha(1+\delta/4)}\]
\item At least $(n')^{\alpha'(1-\delta/8)}\geq n^{\alpha(1-\delta/8)}$
of the first $(n')^{\alpha'}$ infected individuals have a recovery time at least $1$.
\end{enumerate}
Denote by $E_{\mathrm{1comm}}$ the event that the single-community SIR on $G_1'$ behaves ``as expected'', i.e., 1-3 hold.
We then combine the events $E_{\mathrm{travel}}$ and $E_{\mathrm{1comm}}$ via a union bound:
\[
  \Pr\bigl(E_{\mathrm{travel}} \cap E_{\mathrm{1comm}}\bigr) 
    \;\ge\; 1 - \Bigl[\,\Pr(E_{\mathrm{travel}}^c) + \Pr(E_{\mathrm{1comm}}^c)\Bigr]
    \;\ge\; 1 - \delta'-\pi.\]
On $E_{\mathrm{travel}}$, we have $n' \ge n 
-\tilde O(n^{1-\alpha})$ and thus $n'/n \to 1$. On $E_{\mathrm{1comm}}$, the single-community SIR in $G_1'$ infects at least $( r_\infty-\delta)n$ nodes by time $T$ and reaches size $\eps n$ by time at most $ \tau^+=(1+\delta)\frac{\ln n}{\lambda}$.

By the coupling argument, the actual epidemic in $V_1$ must infect at least as many nodes in $V_1^{\static}$ (and do so at least as fast). Hence, on the event $E_{\mathrm{travel}} \cap E_{\mathrm{1comm}}$
\[
   \tau_1(\eps)
   \le\tau_1^{\static}(\eps)
   \le
   (1+\delta)\,\frac{\ln n}{\lambda},
\]
\[ Y_1^{\static}(\tau^+_{1\to 2})\geq n^{\alpha(1+\delta/4)},\]
with at least $n^{\alpha(1+\delta/8)}$ of the individuals in $V_1^\static$ contributing to $Y_1^{\static}(\tau^+_{1\to 2} )$ having a recovery time at least $1$,
all for $n$ sufficiently large. This proves the statements of the proposition, with the statement on $R_1(\infty)$ following from the first statement by setting $\epsilon=r_\infty -\delta$.
\end{proof}

\subsection{Lower Bounds for Community 2}
For the lower bounds in community 2, we will want to use Lemma~\ref{lem:singlecomm-multipleseeds}, which means we will need to understand the transmission of infections from community 1 to 2 via travelers.  For this, it will be important to define a suitable coupling on the whole two-community network.

To define the coupling, we first predetermine the sets $\calT_1$ and $\calT_2$ of vertices that will travel before time $T=\ln^2 n$ keeping track of the fact that when we draw the travel times for vertices in these sets the first travel has to happen before time $T$.  Assuming the first infected vertex is in $V_1^\static$, an event which happens with high probability, we then define a set of discovered vertices, $V_i^{disc}(t)\subset \calT_i$
which is initially empty, and gets updated as part of our coupling as follows. Starting with one initially infected individual chosen uniformly at random in $V_1^\static$, whenever an individual $v$ gets infected we do the following:
 \begin{enumerate}
        \item Start a recovery clock which clicks at rate $\gamma$
        \item\label{forward-degree} Draw four forward degrees, $d_i^\static(v)\sim \Bin (|V_i^\static|,c/n)$ and $d_i^{\calT}(v)\sim \Bin(|\calT_i|,c/n)$ with $i=1,2$
        \item Draw half-edges without an endpoint for the  degrees $d_i^\static$ into static vertices 
        \item\label{discovery-step} Draw vertices $v_{i,1},\dots,v_{i,d_i^\calT}\in \calT_i$
        uniformly without replacement, and link them to $v$ by an oriented edge pointing away from $v$; if any of these vertices has not yet been discovered, add them to
        $V_i^{disc}$,  draw their initial travel state from the conditional distribution, and update their state according to the conditional distribution going forward
\item Start infection clocks for the edges and half-edges out of $v$ (and note that we can determine which edges/half-edges are active, since the travel state of the vertices in $V_i^\static$ is determined)
     \item Once the first infection clock into $V_i^\static$ clicks, we choose $d_i^\static$ endpoints $v_{i,1}^\static,\dots v_{i,d_i^\static}^\static$ uniformly in $V_i^\static$ without replacement, starting with the endpoint for the half-edge which just clicked.  With a slight abuse of notation, we will say that at this point, the infection clock on the edge $vv_{i,1}^\static$ has clicked
    \item When an infection clock on an edge $vw$ clicks, and $w$ is still infected, we say that $v$ tries to infect $w$; if at this point in time $w$ is still susceptible, we call the infection attempt successful and declare $w$ infected.
\end{enumerate}

We would like to couple this process to the following process, which is easier to analyze:
\begin{itemize}
    \item In Step~\ref{forward-degree}, we choose
    $d_i^\calT(v)\sim \Bern(|\mathcal T_i|c/n)$ when $v\in V_1^\static\cup V_2^\static$, and set $d_i^\calT(v)=0$ when $v\notin V_1^\static\cup V_2^\static$.
    \item In Step~\ref{discovery-step}, when $v\in V_1^\static\cup V_2^\static$ and $d_i^\calT=1$, instead of choosing 
     $v_{i,1}^{\calT}\in \calT_i$ uniformly and copying the travel clock of a previously discovered vertex if the uniform choice falls onto such a vertex, we chose  $v_{i,1}^{\calT}\in \calT_i$ uniformly among the not yet discovered vertices, and in particular add them to $V_i^{disc}$ with a fresh draw of travel states from the conditional distribution.
\end{itemize}

\begin{lemma}\label{lem:no-conflict-coupling}
    Let $0<\delta\leq \frac{1-\alpha}2$
    and let 
    $\tau_{1\to 2}^+$ be as in Proposition~\ref{prop_lower_bounds_comm1}.
    Then the above two processes can be coupled such that w.h.p. they are equal for all $t\leq \tau_{1\to 2}^+$.
\end{lemma}

\begin{proof} 

Let $V_{I+R}^\static(t)$ be the set of infected vertices in $V^\static=V_1^\static\cup V_2^\static$,  let $N(t)=|V_{I+R}^\static(t)|=I^\static(t)+R^\static(t)$ and note that by Proposition~\ref{prop:BP-upper-bound}, w.h.p.
\begin{equation}\label{I+R smaller nalpha}
N(t)=\tilde O(n^{(1+\delta)\alpha}) \quad\text{for all}\quad t\leq \tau_{1\to 2}^+.
\end{equation}
Next, recalling that by Lemma \ref{lem:few-travelers} $|\calT| = \tilde O(n^{1-\alpha})$, we use that for $v\in V^\static$,
$$\Pr\left(d_i^\calT(v)\geq 2\right)=O(|\calT|^2n^{-2})=\tilde O(n^{-2\alpha}),
$$
showing that we can couple $d_i^{\calT}(v)\sim \Bin(|\calT_i|,c/n)$
to  a $\Bern(|\mathcal T_i|c/n)$ random variable such that they are equal with probability $1-\tilde O(n^{-2\alpha})$.
Since $N(t)=|V_{I+R}^\static|=\tilde O(n^{(1+\delta)\alpha})$ this shows that we can do this coupling for all
$v\in V_{I+R}^\static$ such that the new random variables are equal to the original ones with probability
$1-\tilde O(n^{-(1-\delta)\alpha})\to 1$.

Next we define $\tau$ to be the stopping time when the first infected vertex $v\in\calT$ draws a degree $d_i^\mathcal T>0$. For $t<\tilde\tau=\min\{\tau_{1\to 2}^+,\tau\}$, the set of discovered vertices $V^{disc} = V_1^{disc} \cup V_2^{disc}$ is just the set of vertices discovered by an infected vertex in $V_{I+R}^\static$, showing that
$$
|V^{disc}|\leq\sum_{v\in V_{I+R}^\static}(d_1^\calT(v)+d_2^\calT(v)).$$  This random variable has 
distribution $\Bin(N(t)|\calT|,c/n)$, and, by \eqref{I+R smaller nalpha}, it is bounded w.h.p. by $\tilde O(|\calT|n^{(1+\delta)\alpha-1})=\tilde O(n^{\alpha\delta})$. This implies that by time $\tilde \tau$, the total number of infections emanating from $\calT$ is bounded by $\tilde O(n^{\alpha\delta})$, which shows that the probability of drawing a random variable $d_i^\calT(v)>0$ is bounded by
\[
\tilde O(n^{\delta \alpha})|\calT|\frac cn=
\tilde O(n^{-(1-\delta)\alpha})=o(1).
\]
Thus, with high probability, the stopping time $\tau$
happens after $\tau_{1\to 2}^+$.  
This allows us to couple Step~\ref{forward-degree}.

Concerning Step~\ref{discovery-step}, we note that under the assumption that  $d_i^\mathcal T=0$ for all $v\in\calT$, the discovery steps now involve just the edges from $V_{I+R}^\static$ into $\calT$.  Since the total number of these edges is a random variable with distribution  $\Bin(N(t)|\calT|,c/n)$, it is w.h.p. bounded by $\tilde O(|\calT|n^{(1+\delta)\alpha-1})$, where we again used
\eqref{I+R smaller nalpha}.  
  Since the endpoints of these edges are chosen i.i.d. uniformly at random in $\mathcal T$, the probably of a two of them hitting the same vertex is bounded by
  $$\frac 1{|\calT|}\tilde O(|\calT|^2n^{(2+2\delta)\alpha-2})=
  \tilde O(|\calT|n^{(2+2\delta)\alpha-2})
  =\tilde O(n^{(1+2\delta)\alpha-1})=\tilde O(n^{-(1-\alpha)^2})=o(1),
   $$
where we used that $\delta\leq\frac {1-\alpha}2$   in the second to last step.
This shows that w.h.p. a single vertex will not be chosen twice and all newly selected vertices will be undiscovered.
\end{proof}

\begin{proposition}\label{prop_lower_bounds_comm2}
Consider the setting of Proposition~\ref{prop_lower_bounds_comm1}, assuming in addition that $\delta\leq \frac {1-\alpha}2$, and set 
$$\tau_2^+=\tau_{1\to 2}^++\tau^+
    =(1+\delta)\frac{(1+\alpha)\ln n}\lambda.$$
Then with probability at least $\pi - \delta'$, the following hold provided $n$ is sufficiently large:
\begin{enumerate}
\item   
    $I_2(\tau^+_{1\to 2}+3)+R_2(\tau^+_{1\to 2}+3)\geq Y_2^{\static}(\tau^+_{1\to 2}+3)\geq n^{{\alpha\delta}/{16}}$
    
\item   
    $\tau_2(\eps)\leq \tau_2^{\static}(\epsilon)\leq \tau_2^++3$ 
\quad\text{and}\quad 
    $\tau_2(n^{-(\alpha+\delta)})\leq \tau_2^\static(n^{-(\alpha+\delta)})
    \leq(1 - \delta^2)\frac{\ln n}{\lambda}+3$
\item 
    $R_2(\infty)\ge I_2^{\static}(\tau^+_2) +R_2^\static(\tau^+_2)\ge (r_\infty - \delta) n$.
\end{enumerate}
\end{proposition}

\begin{proof}  

To prove the proposition, we will first use the coupling described before Lemma~\ref{lem:no-conflict-coupling} to show that by the time $\tau_{1\to2}^++3$, there have been at least $n^{\delta/16}$ external infection events into $V_2^\static$, with the  seeds $w\in V_2^\static$ chosen i.i.d. uniform at random, and then apply Lemma~\ref{lem:singlecomm-multipleseeds}.
To this end, we we note that the SIR epidemic on $V^\static_1$ will only slow down if we neglect all re-infections into $V_1^\static$, and that we can lower bound the number of seeds into $V_2^\static$ by only considering the seeds  which appear because a vertex in $V_1^\static$ infects a vertex in $\calT_2$ which is traveling and returns before its infection clock runs out. 

Define an individuals $v$ contributing to $Y_1^\static(\tau_{1\to 2}^+)$ dangerous if its recovery clock lasts for at least one time unit.   By Proposition~\ref{prop_lower_bounds_comm1} and Lemma~\ref{lem:emprical-distribution-Gnp},  with probability at least $\pi-\delta'/2$, the 
set of
dangerous vertices has size at least $ n^{\alpha(1+\delta/8)}$.
 Furthermore, by Lemma~\ref{lem:no-conflict-coupling},
independently for a dangerous vertex $v\in V_1^\static$ infected at some time $t\leq \tau_{1\to 2}^+$,
\begin{enumerate}
\item  With probability $|\calT_2|c/n$ the vertex $v$ is connected to a single vertex $w\in \calT_2$ 
\item With probability  $\Pr(\calE_{TH}(v,t)\mid w\in \calT_2)=\Theta\left(\frac 1{\ln^2 n}\right)$, $w$ is traveling for one time unit after $v$ got infected, returns home latest one time unit later, and stays home for at least one time unit, implying in particular that the edge $vw$ is active during at least one time unit starting from the infection of $v$.
\item With probability $1-e^{-\beta}$, the infection clock on the edge $vw$ clicks within time $1$ after $v$ was infected, implying $w$ gets infected.
\item With probability $e^{-3\gamma}$ the individual $w$  stays infected for at least three time units, so in particular is home and infected during a time interval at length 1 ending latest at $t+3$
\item With probability of order $\Theta(1)$, the degree $d_2^\static(w)$
into $V_2^\static$ is positive
\item With probability at least $1-e^{-\beta}$ during the time interval of length one following its return home, $w$ creates a new, uniformly random seed in $u\in V_2^\static$.
\end{enumerate}
Thus the probability of transmitting the infection from a dangerous vertex in $v\in V_1^\static$ to a uniformly random seed $u\in V_2^\static$ by time $t_{1\to 2}^++3$ is bounded below by
$$q=\Theta(|\calT_2| n^{-1}\ln^{-2}n)=\Theta(n^{-\alpha}),
$$
i.i.d. for all dangerous vertices. By a simple multiplicative Chernov bound, 
we see that with probability at least $\pi-\delta'$, we have at least $\frac 12 q  n^{\alpha(1+\delta/8)} \geq n^{\alpha\delta/16}$ i.i.d. uniform seeds into $V_2^\static$
by time $\tau_{1\to2}^++3$.
The first statement of the proposition then follows by observing that by the birthday paradox, with high probability all these seeds are distinct.
The second and third statements follow with the help of Lemma~\ref{lem:singlecomm-multipleseeds}, where the proof of the statement about $\tau_2(n^{-(\alpha+\delta)})$ uses the fact that with high probability
$$
\tau_{1\to 2}^+ +3+(1+\delta)\frac{(1-\alpha-\delta)\ln n}\lambda=(1-\delta)(1+\delta)\frac{\ln n}\lambda +3 =(1-\delta^2)\frac{\ln n}\lambda+3.
$$
\end{proof}

\begin{remark}\label{rmk:tau_1_2_upper_bd}
    Note that a direct result of Proposition \ref{prop_lower_bounds_comm2} is that if $\delta'>0$ and $0<\delta \leq \frac{1-\alpha}{2}$, then for all large enough $n$,
    $$
   \Pr\left( \tau_{1\to 2} \leq (1+\delta)\frac{\alpha\ln n}\lambda+3\right)\geq \pi-\delta'
    $$
Similarly to the lower bound following from the first statement of Proposition~\ref{prop:Rinfty-upper-bd} (see Remark~\ref{rem:tau12-lower}), this bound continues to hold if instead of defining $\tau_{1\to 2}$ in terms of individuals with label 2, it is defined in terms of individuals physically present in community 2.
    \end{remark}

\section{Proof of Theorem \ref{thm_sir_tcm_non_int}}

For the  convenience of the reader, we restate Theorem \ref{thm_sir_tcm_non_int} from the main text as Theorem~\ref{SI-thm_sir_tcm_non_int} below.

\begin{theorem}[SIR epidemic with no interventions]\label{SI-thm_sir_tcm_non_int}
    Consider the SIR epidemic spreading on $G(t)$ with rate of travel given by $\rho_\Travel = \Theta(n^{-\alpha})$ for some $\alpha \in (0, 1)$.
        Then the following hold:
    \begin{enumerate}[label=(\alph*)]
        \item\label{SI-thm:main-large-outbreak} The probability of a large  outbreak converges to the survival probability $\pi\in[0,1]$ of the branching process  $\BPsub{c}(t)$.
        i.e.,
        \[
            \Pr(R_1(\infty) + R_2(\infty) \geq \ln n) \to \pi \text{ as } n\to\infty.
        \]
           \item \label{SI-thm:main-survival-prob} Conditioned on not having a large outbreak, the duration of the epidemic is bounded in probability uniformly in $n$, with high probability never reaches community 2, and only ever infects a  number of individuals in community 1 that is bounded in probability.  Finally, with high probability, the stopping times $\tau_1(\epsilon)$ and $\tau_2(\epsilon)$ are infinite.
        \item \label{SI-thm:main-final-size}Conditioned on a large outbreak, $n^{-1} R_1(\infty)$ and $n^{-1} R_2(\infty)$ converge to $r_\infty$, the survival probability of the branching process $T_{\Pois(R_0)}$,           \[
            \frac{\lambda}{\ln n}(\tau_{1\to 2}, \tau_1(\epsilon), \tau_2(\epsilon)) \pto (\alpha, 1, 1+\alpha),
        \]
        \[
            \frac 1{\ln n}\ln\left(I_2(\tau_1(\epsilon))+R_2(\tau_1(\epsilon))\right)\pto 1-\alpha,
        \]
   and the epidemic dies out in time $O(\ln n)$.
       
    \end{enumerate}
\end{theorem}

\noindent Given the results from the last section, the proof of this theorem is an easy exercise. 
\begin{enumerate}[label=(\alph*)]
    \item Let $\calE$ be the event $\{R_1(\infty) + R_2(\infty) < \ln n\}$. 
   Since $\ln n(n^{-1/2}+n^{-\alpha}\ln^3 n)\to 0$ as $n\to\infty$, Proposition \ref{pro:small-time-coupling} implies that
    on the event $\calE$,  $R_1(t) + R_2(t) = \BPsub{c}(t)$ for all $t$ Thus,
    \begin{align*}
        \Pr(\calE) = \Pr(R_1(\infty) + R_2(\infty) < \ln n) &= \Pr(\BPsub{c}(\infty) < \ln n) \xrightarrow{n\to\infty} 1-\pi.
    \end{align*}
    Taking the complement event implies the desired result.
    \item This is an immediate consequence of Proposition \ref{pro:small-time-coupling}. 
    \item 
    Let $\calO$ be the event that there is a large outbreak. 
    We first show that, conditioned on the event $\calO$, we have concentration for $\tau_{1\to 2}, \tau_1(\epsilon),$  $\tau_2(\epsilon)$ and $Y_2(\tau_1(\epsilon))=I_2(\tau_1(\epsilon)) + R_2(\tau_1(\epsilon))$, and that the epidemic dies out in $O(\log n)$ time. 
    
    First, consider $\tau_{1\to 2}$. By Proposition \ref{pro:small-time-coupling}, w.h.p. $\tau_{1\to 2} \geq \frac{\alpha-\delta}{\lambda}\ln n$. By Remark \ref{rmk:tau_1_2_upper_bd}, if 
    $\delta \leq \frac{1-\alpha}2$ then
    $\tau_{1\to 2} \leq (1+\delta)\frac{\alpha\ln n}\lambda$ with probability 
    at least $\pi - \delta'$. Letting 
    \[
        \calA_{1\to 2} =\left\{\frac{\alpha - \delta}{\lambda}\ln n \leq \tau_{1\to 2} \leq \frac{\alpha(1+\delta)}{\lambda}\ln n +3\right\},
    \]
    we have $\Pr(\calA_{1\to 2}) > \pi - \delta' - o(1)$. Now, conditioning on a large outbreak we have
    \[
        \Pr(\calA_{1\to 2} \mid \calO) = \frac{\Pr(\calA_{1\to 2}, \calO)}{\Pr(\calO)}.
    \]
    Notice that $\Pr(\calA_{1\to 2}, \calO) \leq \pi + o(1)$ since $\Pr(\calO) \to \pi$. By Proposition \ref{pro:small-time-coupling}, $\Pr(\calA_{1\to 2}, \calO^c) = o(1)$, since conditioned on no large outbreak, with high probability, the epidemic will not reach community 2. Thus, sending $\delta,\delta'\to 0$ and $n\to\infty$ gives concentration of $\tau_{1\to 2}$ at $\frac{\alpha}{\lambda}\ln n$ conditioned on a large outbreak.

    Next, consider $\tau_1(\epsilon)$. By Proposition \ref{prop_lower_bounds_comm1}, we have that if $R_0 > 1$, $\epsilon \in (0, r_\infty), \delta,\delta' \in (0, 1)$, then $\tau_1(\epsilon) \leq \frac{1+\delta}{\lambda}\ln n$ with probability at least $\pi - \delta'$. By Proposition \ref{prop:BP-upper-bound}, we have that 
    \[
        \Pr\left(\tau_1(\epsilon) \leq \frac{1-\delta}{\lambda}\ln n\right) = \Pr\left(Y_1\left(\frac{1-\delta}{\lambda}\ln n\right) \geq \epsilon n\right) = \tilde O(n^{-\delta}) \xrightarrow{n\to\infty} 0.
    \]
    Thus, with probability at least $\pi-\delta'-o(1)$, $\frac{1-\delta}{\lambda}\ln n \leq \tau_1(\epsilon) \leq \frac{1+\delta}{\lambda}\ln n$. By the same conditioning argument as in the proof for concentration of $\tau_{1\to 2}$, we get the conditional statement of concentration for $\tau_1(\epsilon)$. 

    The argument for $\tau_2(\epsilon)$ is analogous to the argument for $\tau_1(\epsilon)$ and is left to the reader. 

    Using  techniques similar to those proving concentration
    for $\tau_{1\to 2}, \tau_1(\epsilon),$ and $\tau_2(\epsilon)$,
   we can show the statement for concentration of $Y_2(\tau_1(\epsilon))$. Proposition \ref{prop:BP-upper-bound} shows that w.h.p., $Y_2((1+\delta)\frac{\ln n}\lambda)=\tilde O(n^{1-\alpha+\delta})$,
    while Proposition~\ref{prop_lower_bounds_comm2} 
 implies that $Y_2((1-\delta^2)\frac {\ln n}\lambda +3)\geq n^{-\alpha-\delta}$ with probability at least $\pi-\delta'$.

 Let 
\[
    \calE_{\delta} = \left\{
        Y_2\Big((1+\delta)\frac{\ln n}\lambda\Big)\leq n^{1-\alpha+2\delta})
    \quad\text{and}\quad 
        Y_2\Big((1-\delta^2)\frac {\ln n}\lambda +3\Big)\geq n^{-\alpha-\delta}
   \right\}
    \]
    Now, conditioning on a large outbreak we have 
    \begin{align*}
        \Pr(\calE_{\delta} \mid \calO) &= \frac{\Pr(\calE_{\delta}, \calO)}{\Pr(\calO)}
    \end{align*}
    and by the same arguments as for the concentration of $\tau_{1\to 2}$, in particular that $\Pr(\calO) \to \pi$ and $\Pr( \calO\setminus \calE_{\delta}) = o(1)$, we may conclude that for all $\delta>0$,
    $$
    \Pr(\calE_{\delta} \mid \calO) \to 1\quad\text{as}\quad n\to\infty.
    $$
    Combined with the fact that conditioned on $\calO$,  $\frac{\tau_1(\epsilon)}{\ln n}\pto 1/\lambda$, this proves the desired concentration of $Y_2(\tau_1(\epsilon))$.

To prove that the infection dies out in time $O(\ln n)$, we will use the upper bounds for $\tau_1(\epsilon)$ and $\tau_2(\epsilon)$ from Proposition~\ref{prop_lower_bounds_comm1} and \ref{prop_lower_bounds_comm2} along with Proposition~\ref{prop:herd-immunity}. Explicitly, setting $\epsilon=s_0-\delta/2$ and using the upper bounds from Proposition~\ref{prop_lower_bounds_comm1} and \ref{prop_lower_bounds_comm2}, we see that if we set $t_0=\tau_{1\to 2}^++3+\tau^+$ , then with high probability, $\max\{S_1(t_0),S_2(t_0)\}\leq \epsilon n$, allowing us to use   Proposition~\ref{prop:herd-immunity} to establish the claim.
    
    Finally, we show convergence of $R_1(\infty)$ and $R_2(\infty)$. Let $\calE_1^-$ be the event that $\{(r_\infty - \delta) n \leq R_1(\infty)\}$ and $\calE_1^+$ be the event that $\{(r_\infty + \delta) n \geq R_1(\infty)\}$. By Proposition \ref{pro:small-time-coupling}, $\Pr(\calE_1^-, \calE_1^+) = \Pr(\calE_1^-, \calE_1^+, \calO) + \Pr(\calE_1, \calO^c) = \Pr(\calE_1^-, \calE_1^+,\calO) + o(1)$ since no large outbreak implies that $R_1(\infty) + R_2(\infty)$ is bounded with high probability. By Proposition \ref{prop_lower_bounds_comm1}, $\Pr(\calE_1^-) \geq \pi-\delta$, implying that $\Pr(\calE_1^- \mid \calO) \to 1$. Additionally, Proposition \ref{prop:herd-immunity} combined with Proposition \ref{prop:Rinfty-upper-bd} imply that $\Pr(\calE_1^+ \mid O) \to 1$. By taking the complement events, we see that $\Pr((\calE_1^-, \calE_1^+)^c \mid \calO) \to 0$, which implies that $\Pr(\calE_1^-, \calE_1^+ \mid \calO) \to 1$ and $n^{-1}R_1(\infty) \pto r_\infty$ as desired. 
    
    An analogous argument replacing $\calE_1^-$ and $\calE_1^+$ with the analogous events coupled with the upper and lower bounds from Propositions \ref{prop:Rinfty-upper-bd} and \ref{prop_lower_bounds_comm2} imply the result for $n^{-1}R_2(\infty)$.
\end{enumerate}

\section{Interventions}
In this section we present the main steps of the proofs of Theorems \ref{thm_travel_ban} (Travel Ban) and \ref{thm_social_distance} (Social Distancing) from the main text, with some of the needed branching process upper bounds deferred to Sections~\ref{sec;BP-travelban-upper-bd} and \ref{sec;BP-socialdist-upper-bd}.

\begin{proof}[Proof of Theorem \ref{thm_travel_ban}]
Recall that Theorem \ref{thm_travel_ban} states that all statements of Theorem \ref{thm_sir_tcm_non_int} (of the main text) remain valid if we implement a travel ban at time $\tau_1(\epsilon)$. The core idea for proving this is that the travel ban has little effect on the spread within the static part of both populations, since they don't travel anyway, and that the cross-community transmission  happens much sooner than the travel ban can be implemented.

Turning to the formal proof, we first note that $\tau_{1\to 2}$ and $\tau_1(\epsilon)$ have the same concentration as before, since the travel ban is implemented after those times. In a similar way, the fact that parts (a) and (b) from Theorem \ref{thm_sir_tcm_non_int} in the main text hold are immediate since they are statements about events which are determined by the behavior of the epidemic before the travel ban is implements. It remains to show that, conditioned on a large outbreak $R_1(\infty), R_2(\infty)$, $Y_2(\tau_1(\epsilon))$ and $\tau_2(\epsilon)$ converge as expected.   
  
Let $\calT_i$ and $V_i^\static$ be defined as before, i.e., without implementation of a travel ban.  Note that the individuals in  $V_i^\static$ never travel up to time $T=\ln^2n$, whether or not we implement a travel ban at time $\tau_1(\epsilon)$. Recall that in the proof of Proposition \ref{prop_lower_bounds_comm1} we prove a lower bound on the final size of the infection in community 1 by lower bounding the number of infected vertices in $V_1^\static$, neglecting all infections that happen 
through vertices that travel. Therefore the size of an epidemic restricted to $V_1^0$ remains a valid lower bound for the epidemic in community 1 even with the travel ban implemented, and thus we still have the lower bound on the final size of the infection for community 1. Now consider community 2. In the proof of Proposition \ref{prop_lower_bounds_comm2}, we show that by the earlier (for $n$ large enough) time $\tau_{1\to 2}^+ + 3$, there have been at least $n^{\delta/16}$ i.i.d. uniform seeds into $V_2^0$. Even if a travel ban is implemented at time $\tau_1(\epsilon)$, the epidemic in community 2 is lower bounded by an outbreak in a single community with $n^{\delta/16}$ seeds. This places us in the setting of Lemma \ref{lem:singlecomm-multipleseeds}, and conditioned on a large outbreak, we have the lower bound on $Y_2(\tau_1(\epsilon))$ and the final size of the epidemic in community 2 as well as the upper bound on $\tau_2(\epsilon)$.

Next we observe that the proof of the upper bound 
from Proposition \ref{prop:Rinfty-upper-bd} remains unchanged if we implement a travel ban at time  $\tau_1(\epsilon)$,
giving us the desired upper bound on the size of the epidemic in the two communities at time
$T=\ln^2n$.
This leaves us with the proof of the  lower bound on $\tau_2(\epsilon)$, the upper bound on $Y_2(\tau_1(\epsilon))$ and the proof that the infection dies out in time $O(\ln n)$ when we implement the travel ban.
These turn out to be easy generalizations of the corresponding proofs without travel ban, and are deferred to Section~\ref{sec;BP-travelban-upper-bd}.
\end{proof}

Next we prove  Theorem \ref{thm_social_distance} of the main text, which for the convenience of the reader, we restate as Theorem~\ref{SI:thm_social_distance} below.

\begin{theorem}[Effect of Social Distancing]\label{SI:thm_social_distance}
Let $R_0>1$, and suppose that conditioned on $\tau_1(\epsilon)<\infty$ social distancing is implemented  at time $\tau_1(\epsilon)$ . Then statements \ref{thm:main-large-outbreak} and \ref{thm:main-survival-prob} from Theorem \ref{thm_sir_tcm_non_int} still hold, as do the statements that the epidemic  dies out in time $O(\ln n)$ and that conditioned on a large outbreak, 
$n^{-1}R_1(\infty) \pto r_\infty$  showing that   social distancing in community 2 does not change 
the asymptotic final size in community 1.  But  the final size in community 2 conditioned on a large outbreak is now sublinear in the number of vertices --- in fact, it grows by at most a factor $n^{o(1)}$ from where it was when we implemented social distancing at time $\tau_1(\epsilon)$, i.e.,
     \[
        \frac 1{\ln n}\ln\left(R_2(\infty)\right) \pto 1-\alpha.
    \]
\end{theorem}

 \begin{proof}
Recall that mathematically, we model social distancing in community 2 by decreasing the rate of infection for individuals in community 2 from $\beta$ to $\beta'$ where $c\beta' - \beta' - \gamma < 0$. Intuitively, the epidemic trajectory in community 1 will remain largely the same since a positive fraction of the population has already been infected. However, in community 2, although there are a super-constant number of seeded infections, social distancing causes the epidemic to become subcritical leading it to die out quickly and preventing it from reaching a positive fraction of the population.

We start the formal proof by noting that 
since social distancing is only  implemented at time $\tau_1(\epsilon)$,
it is again immediate that parts (a) and (b) from Theorem \ref{thm_sir_tcm_non_int} in the main text hold, and that $\tau_{1\to 2}$ and $\tau_1(\epsilon)$ have the same concentration as before.
 In a similar way, the arguments showing that for all $\delta>0$, with high probability 
$R_1(\infty) \geq (r_\infty - \delta)n$ and  $R_1(T)\leq (r_\infty + \delta)n$, can be immediately be copied over from the above proof for the case of a travel ban at $\tau_1(\epsilon)$. We also have a lower bound on $\ln R_2(\infty)=\ln Y_2(\infty)$ in terms of $\ln Y_2(\tau_1(\epsilon))$ which concentrates around $(1-\alpha)\ln n$ by statement \ref{thm:main-final-size} of Theorem~\ref{thm_sir_tcm_non_int}.

Thus we are left with proving that again, the infection dies out in time $O(\ln n)$ and that for all $\delta$, w.h.p.,  $\ln R_2(\infty)$ is upper bounded by $(1-\alpha+\delta)\ln n$.
We prove these bounds in Section~\ref{sec;BP-socialdist-upper-bd}, noting that conceptually, the main point is to prove that social distancing leads to the improved upper bound on the final size of the infection in community 2 in spite of the additional infections transmitted by travel from community 1 to community 2.
\end{proof}

\section{Branching Process Upper Bounds}\label{sec:BP-upper-bounds}

In this section, we will derive upper bounds on the expected number of infected vertices in the full process by first bounding the expected number in the original process by the expected number in a suitable branching process, and then analyzing the latter.

\subsection{Continuous Time Markov Chain Representation}
We start by giving a full description of the 2-community process as a continuous time Markov chain, adopting the approach of
\cite{Ball2010, borgs2024lawlargenumberssir} to include travel.  As in \cite{Ball2010, borgs2024lawlargenumberssir} we do not pre-draw the underlying random network, rather we explore it as the infection progresses. In contrast to \cite{Ball2010, borgs2024lawlargenumberssir} we will not couple to a Poisson random graph, but stay with the original model drawing edges with probability $c/n$.  Our Markov chain is characterized by a set of possible states, an initial state, and transition rates as defined below.  We will use $\delta_{ij}=\delta^{ij}$ for the standard Kronecker delta, and
$$\delta_{ij}^{AB}=\delta_{ij}\delta^{AB}+(1-\delta_{ij})(1-\delta^{AB}).
$$

\medskip

\noindent\textbf{States:}
\begin{enumerate}
    \item For each vertex $v\in V=V_1\cup V_2$ we have an infection state (susceptible, infected, recovered) and a travel state $\sigma_v(t)\in \{H,T\}$.  We denote the set of susceptible, infected, and recovered vertices at time $t$ by $\calS(t)$, $\calI(t)$ and $\calR(t)$, adding a subscript $1$ or $2$ if we take the intersection with $V_1$ or $V_2$, respectively
    \item In addition, we specify a set of oriented edges $E_{IS}(t)$ such that for all $t$, the edges in $E_{IS}$
    always point from a vertex in $\calI(t)$ to a vertex in $\calS(t)$.  We call these the set of $IS$ edges at time $t$, and refer to an $IS$ edge $uv$ with $u\in V_i$ and $v\in V_j$ as active if $\delta_{ij}^{\sigma_u(t)\sigma_v(t)}=1$, and inactive otherwise.
\end{enumerate}

\noindent\textbf{Initial State:}
\begin{enumerate}
    \item We start the system at time $0$ with 
one infected vertex $v_0\in V_1$ with $\sigma_{v_0}(0)=H$, all other travel states drawn from the empirical distribution, 
$\calR=\emptyset$,  $\calS=V\setminus \{v_0\}$, and 
\item  $E_{IS}$ sampled by including each of the $(2n-1)$ possible edges from $v_0$ into $\calS$ i.i.d. with probability $c/n$, giving $v_0$ $\Bin(n-1,c)$ edges into $V_1$ and $\Bin(n,c)$ edges into $V_2$.
\end{enumerate}

\noindent\textbf{Transitions}
\begin{enumerate}
    \item Each vertex transitions from $H$ to $T$ with rate $\rho_T$, and from $T$ to $H$ with rate $\rho_H$, with $IS$ edges being updated accordingly
    \item Infected vertices recover at rate $\gamma$, and when a vertex $v$ recovers all edges in $E_{IS}$ that start at $v$ get removed
    \item Active edges ``click'' at rate $\beta$.  When an edge $uv$ clicks, the following happens:
    \begin{enumerate}
        \item the  edge $uv$ is removed from $E_{IS}$ as are all other edges pointing into $v$ (we call the removal of these additional edges, if there are any, the clean-up step)
        \item $v$ moves from $\calS$ to $\calI$
        \item $v$ creates additional $IS$ edges, including each of the edges into $\calS\setminus\{v\}$ with probability $c/n$
    \end{enumerate}
\end{enumerate}

\medskip

\noindent\textbf{Macroscopic Random Variables}
Next we define projections onto the random variables we are interested in, defined as functions from the state $\zeta$ of the system onto (a vector of) random variables.  In particular, we will be interested in
\begin{enumerate}
    \item $S_i^\sigma(t)$, $I_i^\sigma(t)$ and $R_i^\sigma(t)$, the number of susceptible, infected, and recovered vertices in $V_i$ that have travel label $\sigma$
    \item $
    Y_i^A(t)=I_i^A(t)+R_i^A(t)$ and $Y_i(t)=Y_i^H(t)+Y_i^T(t)$,
    \item $X_{ij}^{AB}(t)$, the number of edges $uv\in E_{IS}$ such that $u\in V_i$, $v\in V_j$, $\sigma_u(t)=A$ and $\sigma_v(t)=B$
\end{enumerate}
As usually, if we denote the states of the continuous time Markov Chain by $\zeta$, the drift of a bounded function $f:\zeta\mapsto f(\zeta)$ is defined as
$$
\drift f (\zeta) =\sum_{\zeta'} (f(\zeta')-f(\zeta))q(\zeta,\zeta')$$
where $q(\zeta,\zeta')$ is the rate at which $\zeta$ transitions to $\zeta'$.  Then the quantity $M_t[f]$ defined by 
$$
f(\zeta(t))=f(\zeta(0))+\int_0^t f(\zeta(s))ds +M_t[f]
$$
is a martingale, and the expectation of $f$ obeys Dynkin's formula
$$
\E [f(\zeta(t))|\zeta(0)] = f(\zeta(0)) + \int_0^t \E[\drift f(\zeta(s)]ds,
$$
see, e.g., \cite{EthierKurtz1986} Ch. IV, Prop 1.7.  As a consequence, if $|g|$ is bounded then by dominated convergence
\begin{equation}
    \drift f\leq g(\zeta)\quad\text{implies that}\quad
    \E[f(\zeta(t))]\leq f(\zeta(0))+\int_0^t \E[g(\zeta(s)]ds
\end{equation}

\subsection{Derivation of Branching Process Upper Bounds}
\label{sec:bp_upper_bounds}

We will be interested in upper bounds on the expectations of $Y_i^A$.
To this end, we first note that the rate at which infections increase $Y_i^A$ by $1$ is equal to
$$
\beta_i^A(X)=\beta \sum_{j,B}\delta_{ij}^{AB}X_{ji}^{BA}
$$
where we made sure to only count active edges into susceptible vertices with label $(i,A)$ . Thus we get an exact expression for the drift of $Y_i^A$ (denoted by $\drift Y_i^A(t)$) and $Y_i$
$$
\drift Y_i^A=\beta_i^A(X) -\rho_{\bar A} Y_i^A +\rho_{ A}Y_i^{\bar A}
\quad\text{and}\quad \drift Y_i=\beta_i^H(X)+\beta_i^{ T}(X)
$$
where $\bar T=H$ and $\bar H=T$. Recalling that we initially start with $1$ infected vertex in the state $(i,A)=(1,H)$, we note that
at time $t=0$, $Y_i^A(0)=\delta_{i1}\delta^{A1}$ and $Y_i(0)=\delta_{i1}$.

It is more complicated to write down an exact expression for the drift of $X_{ij}^{AB}$,
but we can upper bound it by leaving out the clean-up step, leading to
$$
\drift X_{ij}^{AB}\leq 
-\left(\gamma+\rho_{\bar A}+\rho_{\bar B}+\beta \delta_{ij}^{AB}\right)  X_{ij}^{AB}
+\rho_{ A}X_{ij}^{\bar A B}+\rho_{B}X_{ij}^{A\bar B}
+\beta_i^A (X)\frac cn (|S_j^B|-\delta_{ij}\delta^{AB}).
$$
If we are only interested in the behavior up to time $T=\ln^2n$,
we can bound $|S_j^B|$ by 
$$
\max_{t\in [0,T]}S_j^H\leq n\quad\text{and}\quad
\max_{t\in [0,T]}S_j^T\leq |\calT_j|,
$$
where $\calT_j$ is the set of individuals in $V_j$ which travel at least once between $0$ and $T$.  
Note that the travel histories of the different individuals in $V_j$ are i.i.d., and that the probability of a given individual to travel at least once between $0$ and $T$ is
$$
p_{\geq 1}=p_T+(1-p_T)(1-e^{-\rho_T \ln^2 n})\leq p_T+\rho_T\ln^2 n=O(n^{-\alpha}\ln^2n).
$$
With the help of a Chernov bound, we conclude
that
$$
\Pr(|\calT_j|\geq n^{-\alpha} \ln^3 n)\leq e^{-\theta (n^{1-\alpha}\ln^3 n)}\leq n^{-3}
$$
where the last inequality holds for all sufficiently large $n$. 
 Motivated by this, we define a stopping time
$$ 
\tilde\tau=\min\left\{t\colon \frac 1n S_j^T> n^{-\alpha}\ln^3n\right\},
$$
set
\begin{equation}  \label{tilde-X}
\tilde X_{ij}^{AB}(t)=X_{ij}^{AB}(t) 1_{t<\tilde\tau}  
\end{equation}
and then bound the drift of $\tilde X_{ij}^{AB}$
by
\begin{equation}\label{tilde-X-drift-upper-bound}
     \drift \tilde X_{ij}^{AB}\leq
 -\left(\gamma+\rho_{\bar A}+\rho_{\bar B}+\beta \delta_{ij}^{AB}\right) \tilde X_{ij}^{AB}
+\rho_{ A}\tilde X_{ij}^{\bar A B}+\rho_{B}\tilde X_{ij}^{A\bar B}
+c_B\beta_i^A (\tilde X) 
\end{equation}
where 
\begin{equation}\label{tilde-rho-A}
c_H=c\quad\text{and}\quad
  c_T=c n^{-\alpha}\ln^3n=\tilde O(n^{-\alpha}).
\end{equation}
Defining a $16\times 16$ matrix $M$ by
\begin{equation}\label{M-def}
(M\mathbf v)_{ij}^{AB}=
-\left(\gamma+\rho_{\bar A}+\rho_{\bar B}+\beta \delta_{ij}^{AB}\right) v_{ij}^{AB}
+\rho_{A}v_{ij}^{\bar A B}+\rho_{B}v_{ij}^{A\bar B}
+c_B\beta_i^A (\mathbf v)      
\end{equation}
we can integrate the bound on the drift to obtain the bound
\[
  \E[\tilde X_{ij}^{AB}(t)]\leq \left(e^{tM}\E[X(0)]\right)_{ij}^{AB}.  
\]
Combined with the  {\it a priory} bound $X_{ij}^{AB}\leq n^2$ and the fact that $\tilde \tau\geq \ln^2 n$ with probability at least
$1-n^{-3}$, we conclude that
\begin{equation}\label{EtildeX-bd}
\E[X_{ij}^{AB}(t)]\leq \left(e^{tM}\E[X(0)]\right)_{ij}^{AB} +\frac 1n
\end{equation}
as long as $t\leq \ln^2 n$.  This  bound, together with the observation that
$\E[X_{ij}^{AB}(0)]\leq c_B\delta_{i1}\delta^{A1}$
immediately implies the  following proposition.
\begin{proposition}\label{prop:e^tM-upper-bound}
Let $M$   be defined in \eqref{M-def}, and let
$\mathbf e(0)$ be the 16-dimensional vector $e_{ij}^{AB}=\delta_{i1}\delta^{AB}c_B$.  Then the following holds for all $n$ sufficiently large and all $0\leq t\leq \ln^2n$:
$$
\E[X_{ij}^{AB}(t)]\leq\left(e^{tM}\mathbf e\right)_{ij}^{AB} +\frac 1n
$$
and
$$
\E[Y_i(t)]\leq \delta_{i1}+
\sum_{A,B,j}\delta_{ij}^{AB}\int_0^tds \left(e^{sM}\mathbf e\right)_{ji}^{AB}
+\frac tn
$$
\end{proposition}
\begin{proof}
The first statement  follows from the bound \eqref{EtildeX-bd} by observing that $\E[\tilde X_{ij}^{AB}(0)]\leq \mathbf e_{ij}^{AB}$, and the second follows from the first by the fact that 
\begin{align*}
    \E[Y_i(t)] &= \delta_{i1}+\int_0^t ds\E[\drift Y_i(s)] \\
    &=\E[Y_i(0)]+     \int_0^t \E[\beta_i^H(X(s))+\beta_i^T(X(s))]\\
    &=\delta_{i1} + \sum_{A,B,j}\delta_{ij}^{AB}\int_0^tds\E[X_{ij}^{AB}(s)].
\end{align*}
\end{proof}
\begin{remark}
The bounds of the proposition can expressed in terms of expectations for the numbers of active edges and the the number of descendants in the
 4-type continuous time CMJ-branching process described at the beginning of Section~\ref{sec:BP-results-summary}. Note however, that we did \textit{not} establish a coupling which allows to prove stochastic domination up to the stopping time $\tilde\tau$ --- all we derived are bounds in expectation.
\end{remark}

\subsection{Perturbation Theory Bounds}
\label{sec:perturbation_bounds}

To evaluate the bounds from Proposition~\ref{prop:e^tM-upper-bound}, we need to understand the exponential $e^{tM}$, whose leading behavior will be given by the dominant eigenvalue of $M$, i.e., the eigenvalue with the largest real part.  Unfortunately, the 16x16 dimensional matrix $M$ is hard to diagonalize directly.  To deal with this, we will write $M$ as $M_0+W$, where $M_0$ is 
obtained by setting all terms of order $\tilde O(n^{-\alpha})$ to zero, and then use the following lemma to control the difference of $e^{tM}$ and $e^{tM_0}$.

\begin{lemma}\label{lem:pert}
Let \( M_0 \) and \( W\) be real matrices, and let \( \| \cdot \| \) be a submultiplicative matrix norm. Assume that \( \lambda \geq 0 \) is an eigenvalue of \( M_0 \), and that all other eigenvalues of \( M_0 \) have real parts less than or equal to \( \lambda \). Define
\[
C := \sup_{t \geq 0} e^{-\lambda t} \| e^{tM_0} \|.
\]
Then for all \( t \geq 0 \), we have
\[
\left\| e^{t(M_0 + W)} - e^{tM_0} \right\| \leq C  \left( e^{C \|W\|t} - 1 \right) e^{\lambda t}=O\left(\|W\|{t}e^{C \|W\|t}\right) e^{\lambda t}.
\]
\end{lemma}

\begin{proof}
For normal matrices, a bound of the above form (with $C=1$) follows from  Theorem 4.2 in \cite{van1977sensitivity}.  Unfortunately, without such the normality assumption, the results from \cite{van1977sensitivity} only imply a  bound of the form $O\big(\|W\|e^{C \|W\|t}\big)e^{\kappa t}$, where $\kappa = \sup_{t > 0} \frac{1}{t} \log \|e^{tM_0}\| > \lambda$, which is too weak for our purpose.  We therefore include the (straightforward) proof here.  

We use the Dyson expansion for the matrix exponential:
\[
e^{t(M_0 + W)} = e^{tM_0} + \sum_{k=1}^\infty \int_{\Delta_k(t)} e^{(t - s_1)M_0} W e^{(s_1 - s_2)M_0} \cdots W e^{s_k M_0} \, ds_1 \cdots ds_k,
\]
where \( \Delta_k(t) = \{ 0 \leq s_k \leq \cdots \leq s_1 \leq t \} \) is the standard \( k \)-simplex of time-ordered variables.  
By submultiplicativity of the norm and the assumption on \( C \), we estimate the norm of the \( k \)-th term as
\[
\left\| e^{(t - s_1)M_0} W \cdots W e^{s_k M_0} \right\| 
\leq 
 C^{k+1} \|W\|^k e^{\lambda t}
\]
where we used that the sum of the increments $(t - s_1) + (s_1 - s_2) + \cdots + s_k) $ is $t$.
Since the volume of \( \Delta_k(t) \) is \( \frac{t^k}{k!} \),  each term in the series is bounded by
\( \frac{t^k}{k!} C^{k+1} \|W\|^k e^{\lambda t},\)
giving the desired bound
\[
\left\| e^{t(M_0 + W)} - e^{tM_0} \right\| 
\leq C e^{\lambda t} \sum_{k=1}^\infty \frac{(C \|W\| t)^k}{k!} 
= C \left( e^{C \|W\| t} - 1 \right) e^{\lambda t}.
\]
\end{proof}

For our application,
we use the matrix norm
$$
\|M\|=\max_{\mathbf v\in \R^{16}\setminus\{0\}}
\frac{\|M \mathbf v\|_1}{\|\mathbf v\|_1}
$$
and write $M=M_0+W$, with 
 $M_0$ defined by
\begin{equation}\label{M0-def}
(M_0 \mathbf v)_{ij}^{AB}=
 -\left(\gamma+\rho_H\delta_{AT}+\rho_H\delta_{BT}+\beta \delta_{ij}^{AB}\right)v_{ij}^{AB}
+\rho_H\left(\delta_{AH}v_{ij}^{T B}+\delta_{BH}v_{ij}^{AT}\right)
+\beta c \,\delta_{BH}\sum_{k,C}\delta_{ik}^{AC}v_{ki}^{CA}
\end{equation}
with the last term being $\beta_i^A(v)$ explicitly written out.

\begin{lemma}
\label{lem:matrix_eigenvalues}
The matrix $M_0$ has 16 real-valued eigenvalues, with   $\lambda=(c-1)\beta-\gamma$ being an eigenvalue of algebraic and geometric multiplicity $2$, and all other eigenvalues being $-\gamma$ or smaller.  The projection onto the eigenspace of $\lambda$ is the two-dimensional projection
  $$
  (P_\lambda\mathbf v)=
  \left(\mathbf e_{11}^{HH}+\frac c{c-1}\mathbf e_{12}^{HH}\right)v_{11}^{HH}+
 \left(\mathbf e_{22}^{HH}+\frac c{c-1}\mathbf e_{21}^{HH}\right)v_{22}^{HH},
  $$
  where  we use $\mathbf e_{ij}^{AB}$ to denote the coordinate vector for the coordinate 
\(
\renewcommand{\arraystretch}{0.7}
\left[
\begin{array}{@{}c@{}c@{}c@{}c@{}}
A & B\\
i & j
\end{array}
\right]
\).
\end{lemma}

\begin{proof}
As we will see, after a suitable ordering of the indices, the matrix $M_0$ is lower triangular, allowing us to read off the eigenvalues and their multiplicity from the diagonal.  Assuming this for the moment, we start by computing the diagonal of the matrix $M_0$,
noting that the last term in \eqref{M0-def} gives a contribution to the diagonal if and only if 
$B=H$,
$C=A=B$, $i=j=k$ and
 $\delta_{ik}^{AC}=1$, showing that the only positive contribution to the diagonal are the two entries
$$
\delta_{ij}\delta^{AH}\delta^{BH} (c\beta-\gamma-\beta)=\lambda \delta_{ij}\delta^{AH}\delta^{BH} 
$$
with $i=j=1$ or $i=j=2$, while all others are smaller or equal than $-\gamma$.  This gives the statement about the eigenvalues of $M_0$.

Next we note that 
the restriction of  the matrix $M_0$ 
to the $4$-dimensional subspace of $\mathbb R^{16}$ spanned by 
  $\mathbf e_{11}^{HH}$,   $\mathbf e_{12}^{HH}$,   $\mathbf e_{22}^{HH}$ and  $\mathbf e_{21}^{HH}$ is equal to
  $$
\begin{bmatrix}
    \lambda&0&0&0\\
    c\beta&-\gamma&0&0\\
    0&0&\lambda&0\\
    0&0&c\beta&-\gamma
\end{bmatrix}.
  $$
Diagonalizing this matrix one immediately gets the statement about the projection to the eigenspace of $\lambda$.

We are left with is a proof of the lower triangular structure. 
To this end, we write down the edges $\alpha\to\alpha'$ for which $M_{\alpha',\alpha}\neq 0$ and $\alpha\neq\alpha'$, giving
\begin{align}
\renewcommand{\arraystretch}{0.7}
\left[
\begin{array}{@{}c@{}c@{}c@{}c@{}}   T&T \\ i&i \end{array}\right]
  &\to
\left[
\begin{array}{@{}c@{}c@{}c@{}c@{}}  T&H \\ i&i \end{array}\right],\quad
\left[
\begin{array}{@{}c@{}c@{}c@{}c@{}} T&H \\ i&\bar i \end{array}\right]
     \\
\left[
\begin{array}{@{}c@{}c@{}c@{}c@{}}  H&T \\ \bar i& i \end{array}\right]
  &\to
\left[
\begin{array}{@{}c@{}c@{}c@{}c@{}} T&H \\ i&i \end{array}\right],\quad
\left[
\begin{array}{@{}c@{}c@{}c@{}c@{}}  T&H \\ i&\bar i \end{array}\right]  
  \\
\left[
\begin{array}{@{}c@{}c@{}c@{}c@{}} T&H \\ \bar i& i \end{array}\right]
  &\to
\left[
\begin{array}{@{}c@{}c@{}c@{}c@{}} H&H \\ i&i \end{array}\right],\quad
\left[
\begin{array}{@{}c@{}c@{}c@{}c@{}} H&H \\ i&\bar i \end{array}\right]  
  \\
\left[
\begin{array}{@{}c@{}c@{}c@{}c@{}}  H&H \\ i& i \end{array}\right]
  &\to
\left[
\begin{array}{@{}c@{}c@{}c@{}c@{}}  H&H \\ i&\bar i \end{array}\right]  
\end{align}
where the last line only has one edge since the other is a contribution to the diagonal. If we introduce a partial order by 
\[
\renewcommand{\arraystretch}{0.7}
(TT) > (HT) > (TH) > (HH)
\quad\text{and}\quad
\left[
\begin{array}{@{}c@{}c@{}c@{}c@{}}
H & H\\
i & i
\end{array}\right]
>
\left[
\begin{array}{@{}c@{}c@{}c@{}c@{}}
H & H\\
i & \bar i
\end{array}
\right]
\]
we see that 
all transition corresponding to the last term in \eqref{M0-def} correspond to matrix elements $M_{\alpha',\alpha}$ where $\alpha'<\alpha$.  This clearly also holds for the second term, showing that all off-diagonal terms of $M_0$ respect this order.  Extending the above order to a total order (e.g., by defining a DAG and using topological sorting) this proves that $M_0$ is lower triangular.
\end{proof}

\begin{corollary}\label{cor:etM0-decay}
Let $\|\cdot\|$ be a sub-multiplicative matrix norm, and let $M_0$ be the matrix defined in \eqref{M0-def}. Then there exists a constant $C$ (depending on $M_0$) such that
\begin{equation}
\label{etM0-decay}
\|e^{tM_0}\|\leq C e^{t\lambda},
\end{equation}
showing in particular that the constant $C$ from Lemma~\ref{lem:pert} is finite.  More generally, these statements hold if $\lambda$ is a real eigenvalue of $M_0$ which has the same algebraic and geometric multiplicity, and the real part of all other eigenvalues is strictly smaller.
\end{corollary}

\begin{proof}By the previous lemma, the eigenvalue $\lambda$ is semi-simple, implying that the Jordan block corresponding to $\lambda$ is diagonal. Together with the fact that all other eigenvalues of $M_0$ are real and strictly smaller, the corollary follows.  More explicitly, let $P$ be an invertible matrix such that $J=PM_0P^{-1}$ is in Jordan normal form.  Since $\|\cdot
\|$ is sub-multiplicative, we have
$$  
\|e^{tM_0}\|=\|P^{-1}Pe^{tM_0}P^{-1}P\|=\|P e^{tJ} P^{-1}\|\leq \|P\| \|P^{-1}\|\|e^{tJ}\|.
$$
Now we can bound the norm $\| e^{tJ}\|$. Let $J_{k_i}$ denote the $i$th block of the Jordan matrix $J$ with dimension $k_i$. Each block can be written as $J_{k_i} = \lambda_i I + N_{k_i}$ where $\lambda_i$ is the $i$th eigenvalue, $I$ is the identity matrix, and $N_{k_i}$ is the matrix with ones on the superdiagonal. Notice that $N_{k_i}$ is nilpotent with power $k_i$. We can bound the norm of each block by
\[
    \|e^{t J_{k_i}}\| = \|e^{t(\lambda_i I + N_{k_i})} \| \leq C_{k_i}e^{t \lambda_i}(1+t^{k_i -1})
\]
for constants $C_{k_i}$. Since $\lambda$ is a semi-simple eigenvalue, this implies the norm can be upper-bounded as
\[
    \|e^{tJ}\| \leq e^{t \lambda}\left(2 + \sum_{i:\lambda_i \neq \lambda}C_{k_i} e^{t(\lambda_i - \lambda)}(1+t^{k_i-1})\right) = O(e^{t \lambda})
\]
where the last equality is true since all other eigenvalues are strictly smaller than $\lambda$. Thus, 
\[
    \|e^{tM_0}\| \leq C e^{t\lambda},
\]
where $C$ is a finite constant that depends on $M_0$ through $P$ and the eigenvalues of $M_0$.
\end{proof}

\subsection{Proof of  Proposition~\ref{prop:BP-upper-bound}}

At this point, we have all we need to prove the desired upper bounds on the 2-type epidemic with travel.

\begin{proof}[Proof of Proposition~\ref{prop:BP-upper-bound}]
Recalling Proposition~\ref{prop:e^tM-upper-bound}, we need to bound 
$
\left(e^{tM}\mathbf e\right)_{ij}^{AB}$
where $\mathbf e(0)$ is the 16-dimensional vector $e_{ij}^{AB}=\delta_{i1}\delta^{AH}c_B$.
To this end, we first bound
\[
\mathbf e \leq c_T {\mathbf e}_T + c \mathbf e_H
\]
where ${\mathbf e_T}=\mathbf e_{11}^{HT}+\mathbf e_{12}^{HT}$ and $\mathbf e_H={\mathbf e}_{11}^{HH}+\frac c{c-1}{\mathbf e}_{12}^{HH}$, with the latter being one of the two eigenvectors of $M_0$ corresponding to $\lambda$, and then use that
$$c_T\|e^{tM_0}\mathbf e_T\|_1=\tilde O(n^{-\alpha})\|e^{tM_0}\|=\tilde O(n^{-\alpha}e^{\lambda t}),
$$
where in the last step we used Corollary~\ref{cor:etM0-decay}.
Combining these bounds with the fact that $M=M_0+W$ (where $\|W\|=\tilde O(n^{-\alpha})$) and  
Lemma~\ref{lem:pert},
we get
$$
\left(e^{tM}\mathbf e\right)_{ij}^{AB}
\leq c e^{\lambda t}\left(\mathbf e_H\right)_{ij}^{AB}
+\tilde O(n^{-\alpha}e^{\lambda t})
$$
Using the bound on $\E[Y_i(t)]$ from Proposition~\ref{prop:e^tM-upper-bound}, this implies that

\begin{align*}
\E[Y_i(t)]&\leq \delta_{i1}+
\sum_{A,B,j}\delta_{ij}^{AB}\left(\mathbf e_H\right)_{ji}^{AB}\int_0^t ce^{\lambda s}\, ds+\tilde O(n^{-\alpha}e^{\lambda t})+\frac tn
\\
&=
\left(1+\frac{c(e^{\lambda t}-1)}\lambda\right)\delta_{i1}+\tilde O(n^{-\alpha}e^{\lambda t})
\end{align*}
The bounds from Proposition~\ref{prop:BP-upper-bound} now follow using 
Markov's inequality.
\end{proof}

\subsection{Herd Immunity}\label{sec:Herd-Immunity}
In this section, we prove Proposition \ref{prop:herd-immunity}. To do so, we use the fact that once the epidemic has infected enough susceptible vertices, the epidemic becomes subcritical as does the two community branching process approximation. Then, using the bounds proved in previously in this section we show that the infection dies out in time  $\ln n$.

\begin{proof}[Proof of  Proposition \ref{prop:herd-immunity}]
    Let $\calE_0$ be the event that $\max\{S_1(t_0), S_2(t_0)\} \leq (s_0 - \delta)n$ where  $s_0=1/ R_0$. Under the event $\calE_0$, we may bound the drift of $\tilde X_{ij}^{AB}(t)$ for $t\geq t_0$ 
by 
\begin{equation}\label{tilde-X-drift-upper-bound-hi}
\drift \tilde X_{ij}^{AB}\leq
\left(\tilde M \tilde X\right)_{ij}^{AB},
\end{equation}
 where $\tilde M$ is the $16 \times 16$ matrix defined by
    \begin{equation}\label{tilde-M-def}
    (\tilde M\mathbf v)_{ij}^{AB}=
    -\left(\gamma+\rho_{\bar A}+\rho_{\bar B}+\beta \delta_{ij}^{AB}\right) v_{ij}^{AB}
    +\rho_{\bar A}v_{ij}^{\bar A B}+\rho_{\bar B}v_{ij}^{A\bar B}
    +\tilde c_B\beta_i^A (\mathbf v)      
    \end{equation}
    with   \begin{equation}
    \tilde c_H=c(s_0 - \delta)\quad\text{and}\quad \tilde c_T=c n^{-\alpha}\ln^3n=\tilde O(n^{-\alpha}).
    \end{equation}
This allows us to
replace the bound in Proposition \ref{prop:e^tM-upper-bound} by
    \[
        \E[X_{ij}^{AB}(t)\mid\calE_0]\leq\left(e^{(t-t_0)\tilde M}\E[\tilde X(t_0)\mid\calE_0]\right)_{ij}^{AB} +\frac 1n
    \]
Defining $\tilde M_0$ by
\begin{equation}\label{tilde-M0-def}
(\tilde M_0 \mathbf v)_{ij}^{AB}=
 -\left(\gamma+\rho_H\delta_{AT}+\rho_H\delta_{BT}+\beta \delta_{ij}^{AB}\right)v_{ij}^{AB}
+\rho_H\left(\delta_{AH}v_{ij}^{T B}+\delta_{BH}v_{ij}^{AT}\right)
+\beta c(s_0-\delta) \,\delta_{BH}\sum_{k,C}\delta_{ik}^{AC}v_{ki}^{CA}
\end{equation}
and observing that $\beta c(s_0-\delta)-\beta-\gamma=-\beta c\delta$, we may proceed as in the proof of
Lemma \ref{lem:matrix_eigenvalues} to conclude that  the largest eigenvalue of $\tilde M_0$ is $\tilde \lambda = -\beta c \delta$. Assuming without loss of generality that $\delta$ is small enough to guarantee that all other eigenvalues are strictly smaller, we are again in the setting of
Corollary~\ref{cor:etM0-decay}, allowing us to proceed as before to conclude that
\[
\|e^{(t-t_0)\tilde M}\|=
\|e^{(t-t_0)\tilde M_0}\|+\tilde O(n^{-\alpha}e^{\tilde\lambda(t-t_0)})=
O(e^{-\beta c\delta(t-t_0)})+\tilde O(n^{-\alpha}e^{-\beta c\delta (t-t_0)})=O(e^{-\beta c\delta(t-t_0)}).
\]
Combined with the naive bound  $\|\tilde X(t_0)\|_1=\sum_{i, j, A, B} X_{ij}^{AB}(t_0) \leq  n^2$, this implies that
  \begin{equation}\label{eq:x_tilde_bound}
   \E\left[ \| [ \tilde X(t)\|_1\mid \calE_0\right]=O(n^2e^{-\beta c\delta(t-t_0)})+ O\Big(\frac 1n\Big)
    \end{equation}
    for all $t\in [t_0,\ln^2 n]$.
    Define the stopping time $\tau = \inf \left\{t > t_0: \sum_{i,, j, A, B} \tilde X_{ij}^{AB}(t) = 0\right\}$. Let $C$ be a constant to be determined later. Then, we have 
    \begin{align*}
        \Pr\big(\tau > t_0 + C\ln n \mid \calE_0\big) & = \Pr\bigg(\| X(t_0 + C\ln n)\|_1 \geq 1\;\bigg|\; \calE_0\bigg)
        \leq \E\big[\|\tilde X(t_0 + C\ln n)\|_1\;\big|\;\calE_0\big]
        =O(n^{2-\beta\delta c C})+O\Big(\frac 1n\Big)
    \end{align*}
We conclude that for any $C > \beta c\delta/2$ at time $t^* = t_0 + C\ln n$, conditioned on $\calE_0$, $X(t^*) = 0$ with high probability. Although conditioned on $\calE_0$ $X(t^*) = 0$ w.h.p., this does not immediately imply the same for $I_1(t^*) + I_2(t^*) = 0$ since vertices can be infected but not have any active edges. Thus, our last step is to show that conditioned on $\calE_0$, $I_1(t^*) + I_2(t^*) = 0$. Assume that all $2n$ vertices are infected at time $t^*$. Let $Z_1, \dots, Z_{2n} \sim \Expo(\gamma)$. Then, we have the following:
\begin{align*}
    \Pr(I_1(t^*+s) + I_2(t^*+s) > 0 \mid \calE_0) &\leq \sum_{i=1}^n \Pr(Z_i > s) = ne^{-\gamma s}.
\end{align*}
Thus, letting $s = \tilde C\ln n$, where $\tilde C > 1/\gamma$, and conditioning on $\calE_0$, $I_1(t^*+s) + I_2(t^*+s) =0 $ w.h.p., giving the desired result.
\end{proof}

\subsection{Travel Ban Upper Bounds}
\label{sec;BP-travelban-upper-bd}

For upper bounds on the size of the branching process in the event of a travel ban, we will follow the proofs in Sections \ref{sec:bp_upper_bounds} and \ref{sec:perturbation_bounds}, making adjustments where needed to reflect the effects of the travel ban.

First, we note that the drift of $\tilde X_{ij}^{AB}$ can be upper bounded as follows:
\[
 \drift \tilde X_{ij}^{AB}\leq
 -\left(\gamma+\beta \delta_{ij}^{AB}\right) \tilde X_{ij}^{AB}
+\rho_{ A}\tilde X_{ij}^{\bar A B}+\rho_{B}\tilde X_{ij}^{A\bar B}
+c_B\beta_i^A (\tilde X),   
\]
where the difference between this upper bound and the one found in \eqref{tilde-X-drift-upper-bound} is that the terms arising from edge removal due to travel are removed. This immediately implies that Proposition \ref{prop:e^tM-upper-bound} holds with the matrix $M$ replaced by the matrix $\tilde M$, where $\tilde M$ is defined by
\[
(\tilde M\mathbf v)_{ij}^{AB}=
-\left(\gamma+\beta \delta_{ij}^{AB}\right) v_{ij}^{AB}
+\rho_{A}v_{ij}^{\bar A B}+\rho_{B}v_{ij}^{A\bar B}
+c_B\beta_i^A (\mathbf v).
\]
Now, we will apply Lemma \ref{lem:pert} with $\tilde M = \tilde M_0 + \tilde W$, and $\tilde M_0$ defined by
\[
    (\tilde M_0 \mathbf v)_{ij}^{AB}=
 -\left(\gamma+\beta \delta_{ij}^{AB}\right)v_{ij}^{AB}
+\rho_H\left(\delta_{AH}v_{ij}^{T B}+\delta_{BH}v_{ij}^{AT}\right)
+\beta c \,\delta_{BH}\sum_{k,C}\delta_{ik}^{AC}v_{ki}^{CA}.
\]
Note that, with this new matrix $\tilde M_0$, Lemma \ref{lem:matrix_eigenvalues} still holds. The indices of $\tilde M_0$ can be reordered such that $\tilde M_0$ is lower triangular with eigenvalues $\lambda$ of multiplicity $2$ and all other eigenvalues smaller than $-\gamma$. Furthermore, the restriction to the subspace spanned by $\mathbf e_{11}^{HH}$,   $\mathbf e_{12}^{HH}$,   $\mathbf e_{22}^{HH}$ and  $\mathbf e_{21}^{HH}$ remains the same. This also implies that the eigenvectors corresponding to $\lambda$ for $\tilde M_0$ are the same as the eigenvectors corresponding to $\lambda$ for $M_0$. Furthermore $\|\tilde W \| = \tilde O(n^{-\alpha})$, and by an analogous proof as in the proof of Proposition \ref{prop:BP-upper-bound}, we see that Proposition \ref{prop:BP-upper-bound} still holds using the modified upper bound for the drift of $X_{ij}^{AB}$.

Now, consider the setting where we implement a travel ban at time $\tau_1(\epsilon)$. Notice that for all $t > \tau_1(\epsilon)$, we have the following upper bound on the drift of $\tilde X_{ij}^{AB}$:
\[
 \drift \tilde X_{ij}^{AB}\leq
 -\left(\gamma+\beta \delta_{ij}^{AB}\right) \tilde X_{ij}^{AB}+c_B\beta_i^A (\tilde X).  
\]
This can be represented by a matrix $\tilde M^*$ defined by
\[
(\tilde M^*\mathbf v)_{ij}^{AB}=
-\left(\gamma+\beta \delta_{ij}^{AB}\right) v_{ij}^{AB}+c_B\beta_i^A (\mathbf v).
\]
Notice that $\tilde M$ is an \textit{entry wise} upper bound for the matrix $\tilde M^*$, thus the bounds in Proposition \ref{prop:BP-upper-bound} are upper bounds regardless of whether travel occurs, proving the desired  upper bound on $Y_2(\tau_1(\epsilon))$ and lower bound on $\tau_2(\epsilon)$.
In a similar way,
replacing the matrix $\tilde M$ defined in \eqref{tilde-M-def} by
$$
    (\tilde M\mathbf v)_{ij}^{AB}=
    -\left(\gamma+\beta \delta_{ij}^{AB}\right) v_{ij}^{AB}
    +\rho_{\bar A}v_{ij}^{\bar A B}+\rho_{\bar B}v_{ij}^{A\bar B}
    +\tilde c_B\beta_i^A (\mathbf v),     
$$
we see that Proposition \ref{prop:herd-immunity} continues to hold if we implement a ban at time $\tau_1(\epsilon)$.

\subsection{Social Distancing Upper Bounds}
\label{sec;BP-socialdist-upper-bd}

In this section, we will show that under social distancing, conditioned on a large outbreak, the infection in community 2 will die out in $O(\ln n)$ time and that for all $\delta$, w.h.p.,  $\ln R_2(\infty)$ is upper bounded by $(1-\alpha+\delta)\ln n$. Heuristically, to do this, we condition on the event of a large outbreak and run the epidemic until a time $t_0$, at which point the epidemic has reached herd immunity in community 1 with high probability. We then show that at time $t_0$, the number of infections in community 2 is $\tilde O(n^{-\alpha}e^{\lambda t_0})$ and that implementing social distancing at time $t_0$ will only add a factor $n^\delta$ new infections. We will choose $t_0$ such that the bounds we derive will be an upper bound on the final size of the infection in community 2 if we social distance at time $\tau_1(\epsilon)$. 

Set $t_0 := (1+\delta) \frac{\ln n}{\lambda}$. Let $\calE_0$ be the event of a large outbreak and let $\tilde \calE_0 = \{\tau_1(\epsilon') \leq t_0\}$ where $\epsilon' \geq \epsilon$ is chosen such that $s_0 > 1-\epsilon' > 1- r_\infty$ with $s_0 = 1/R_0$ being the herd immunity threshold for a single community. By Proposition \ref{prop_lower_bounds_comm1}, $\Pr(\tilde \calE_0) > \pi- \delta$. Note that  $\calE_0^c$ implies $\tilde \calE_0^c$, so $\tilde \calE_0 \subseteq \calE_0$. In addition, $\Pr(\calE_{0}) \to \pi$ so
$\Pr(\tilde \calE_0 \mid \calE_0) \to 1$. 

First, we require bounds on $\E[\tilde X(t_0) \mid \tilde \calE_0]$ and $\E[Y_2(t_0) \mid \tilde \calE_0]$, the conditional size of the epidemic at time $t_0$. Since  social distancing only decreases the rate of infections, the bounds we derived when proving Proposition~\ref{prop:BP-upper-bound} still hold, showing  that 
\[
     \E[Y_2(t_0)] = \tilde O(n^{-\alpha}e^{\lambda t_0})\quad\text{and}\quad
        \E[X_{ij}^{AB}(t_0)] \leq c(\mathbf e_H)_{ij}^{AB} e^{\lambda t_0} + \tilde O(n^{-\alpha}e^{\lambda t_0}),
\]
where as before $\mathbf e_H={\mathbf e}_{11}^{HH}+\frac c{c-1}{\mathbf e}_{12}^{HH}$. Since $\E[\tilde\calE_0]\geq \pi-\delta$, this implies  that
\begin{equation}\label{eq:bnd-sd-y-2}
    \E[Y_2(t_0) \mid \tilde \calE_0] \leq \frac{ \E[Y_2(t_0)] }{\E[ \tilde \calE_0]}=
    \tilde O(n^{-\alpha} e^{\lambda t_0})
\end{equation}
and
\begin{equation}\label{eq:bnd-sd-tilde-x}
    \E[\tilde X_{ij}^{AB}(t_0) \mid \tilde\calE_0] \leq  \E[X_{ij}^{AB}(t_0) \mid \tilde\calE_0] \leq \left(\mathbf e_H\right)_{ij}^{AB}  O(e^{\lambda t_0}) + \tilde O(n^{-\alpha}e^{\lambda t_0}).
\end{equation}

Next we note that after social distancing is implemented, 
the rate at which an active edge with indices $(ij,AB)$ transmits an infection becomes
\[
\tilde \beta_{ij}^{AB} = \begin{cases}
    \beta &\text{ for } \; (ij, AB) \in \{(11, HH), (12, HT), (21, TH), (22, TT)\} \\
    \beta'&\text{ for } \; (ij, AB) \in \{(22, HH), (21, HT), (12, TH), (11, TT)\}
\end{cases}
\]
where  $\beta'<\beta$ is such that $\lambda' := c\beta' - \beta' - \gamma < 0$. 

Conditioned on the event $\tilde\calE_0$, we have that
$S_1(t)\leq n(1-\epsilon')= n(s_0-\delta')$ for all  $t\geq t_0$, where $\delta'=s_0-(1-\epsilon')>0$.  Furthermore, since $\epsilon'\geq \epsilon$, we may assume that social distancing is implemented by time $t_0$.  As a consequence, for $t\geq t_0$, the upper bound 
\eqref{tilde-X-drift-upper-bound} for the drift of $\tilde X$ can be improved to
\begin{equation}\label{tilde-X-drift-upper-bound-sd}
     \drift \tilde X_{ij}^{AB}\leq
\left(\tilde M\tilde X\right)_{ij}^{AB}
\end{equation}
where the matrix $\tilde M$ is defined by 
\[
    (\tilde M\mathbf v)_{ij}^{AB}=
-\left(\gamma+\rho_{\bar A}+\rho_{\bar B}+\tilde\beta_{ij}^{AB} \delta_{ij}^{AB}\right) v_{ij}^{AB}
+\rho_{A}v_{ij}^{\bar A B}+\rho_{B}v_{ij}^{A\bar B}
+c_j^B\tilde\beta_i^A (\mathbf v)  
\]
with 
\[
    c_j^H = c((s_0 - \delta')\delta_{j1} + \delta_{j2}) \qquad c_j^T = cn^{-\alpha}\ln^3 n.
\]
and $\tilde \beta_i^A (\mathbf v) = \sum_{k, C} \delta_{ki}^{AC} v_{ki}^{CA} \tilde\beta_{ki}^{CA}$.

As in the case without social distancing, we analyze the matrix $\tilde M$  by decomposing it into a part $\tilde M_0$ involving the evolution without travel, plus a perturbation of order $\tilde O(n^{-\alpha})$, with $\tilde M_0$ defined by
\[
(\tilde M_0 \mathbf v)_{ij}^{AB}=
 -\left(\gamma+\rho_H(\delta_{AT}+\delta_{BT})+\tilde\beta_{ij}^{AB} \delta_{ij}^{AB}\right)v_{ij}^{AB}
+\rho_H\left(\delta_{AH}v_{ij}^{T B}+\delta_{BH}v_{ij}^{AT}\right)
+ c_j^H\delta_{BH}\sum_{k,C}\delta_{ik}^{AC}v_{ki}^{CA}{\tilde\beta_{ki}^{CA}}.
\]
By the same argument as in the proof of Lemma \ref{lem:matrix_eigenvalues}, the largest eigenvalue of $\tilde M_0$ is 
the largest diagonal element of $\tilde M$, and as before, the only contribution to the diagonal coming from the last term is the
contribution where $i=j=k$ and $A=B=C=H$, corresponding to the two diagonal entries
$$
\tilde\lambda:=c\big((s_0 - \delta')-1\big)\tilde\beta_{11}^{HH}-\gamma=-c\delta'\beta<0
\quad\text{and}\quad
\lambda':=(c-1)\tilde\beta_{22}^{HH}-\gamma=(c-1)\beta'-\gamma<0.
$$
Since all other diagonal entries are at most $-\gamma$ and $\lambda'>-\gamma$, this shows that the largest eigenvalue of $\tilde M_0$ is equal to $\nu=\max\{\tilde\lambda,\lambda'\}$, which is negative.  
Note that the matrix $\tilde M_0$ leaves the 
four-dimensional subspace spanned by $\mathbf e_{11}^{HH}, \mathbf e_{12}^{HH}, \mathbf e_{22}^{HH},$ and $\mathbf e_{21}^{HH}$  invariant, with its restriction to this subspace given by
\begin{equation}\label{M0-restricted}
    \begin{bmatrix}
        \tilde \lambda & 0 & 0 & 0 \\
        c\beta & -\gamma & 0 & 0 \\
        0 & 0 & \lambda' & 0 \\
        0 & 0 & c(s_0-\delta')\beta' & - \gamma
    \end{bmatrix}.
\end{equation}
Diagonalizing this matrix, one easily verifies that the eigenvalue
$\nu$ is semi-simple, allowing us to
proceed as in the proof of Proposition \ref{prop:herd-immunity} in Section \ref{sec:Herd-Immunity} to conclude that
 conditioned on $\tilde \calE_0$ the epidemic dies out in $O(\ln n)$ time.

Now, we show that a factor of only $n^{\delta}$ infections are added in community 2 after time $t_0$. For all $t_0 \leq t \leq \ln^2 n$, replacing the bound in Proposition \ref{prop:e^tM-upper-bound} using our matrix  $\tilde M$, we have the upper bounds
\begin{equation}\label{eq:tilde_m_social_distancing}
  \E[X_{ij}^{AB}(t) \mid \tilde \calE_0] \leq \left(e^{(t-t_0) \tilde M}\E[\tilde X(t_0) \mid \tilde \calE_0]\right)_{ij}^{AB} + \frac{1}{n},
\end{equation}
implying that
\begin{equation}\label{eq:Y_social_distancing}
   \E[Y_i(t) \mid \tilde \calE_0]\leq \E[Y_i(t_0) \mid \tilde \calE_0]+
    \sum_{A,B,j}\delta_{ij}^{AB}\int_{t_0}^tds \left(e^{(s-t_0)\tilde M}\E[\tilde X(t_0) \mid \tilde \calE_0]\right)_{ji}^{AB}    +\frac{t-t_0}n\\
\end{equation}
As a consequence of \eqref{eq:Y_social_distancing} and \eqref{eq:bnd-sd-tilde-x}, we have
\begin{equation}
 \label{eq:Y_sd_integral_ub}
\begin{aligned}
   \E[Y_i(t) \mid \tilde \calE_0]&\leq \E[Y_i(t_0) \mid \tilde \calE_0]+
    \sum_{A,B,j}\delta_{ij}^{AB}\int_{t_0}^tds \left(e^{(s-t_0)\tilde M}\E[\tilde X(t_0) \mid \tilde \calE_0]\right)_{ji}^{AB}    +\frac{t-t_0}n\\
    &\leq \E[Y_i(t_0) \mid \tilde \calE_0] + \sum_{A, B, j} \delta_{ij}^{AB} \int_{t_0}^t ds \, \left(e^{(s-t_0)\tilde M}\mathbf e_H\right)_{ji}^{AB} O(e^{\lambda t_0}) + 
    \int_{t_0}^t ds \|e^{(s-t_0)\tilde M}\|
    \tilde O(n^{-\alpha}e^{\lambda t_0})+\frac{t-t_0}n
\end{aligned}
\end{equation}
To continue, we use Lemma~\ref{lem:pert} to approximate $\tilde M$ by $\tilde M_0$, using again that $\|\tilde M-\tilde M_0\|=\tilde O(n^{-\alpha})$.  This gives the bound
\begin{equation}
 \label{eq:Y_sd_integral_ub-2}
\begin{aligned}
   \E[Y_i(t) \mid \tilde \calE_0]
    &\leq \E[Y_i(t_0) \mid \tilde \calE_0] + \sum_{A, B, j} \delta_{ij}^{AB} \int_{t_0}^t ds \, \left(e^{(s-t_0)\tilde M_0}\mathbf e_H\right)_{ji}^{AB} O(e^{\lambda t_0}) 
       + \int_{t_0}^t ds e^{(s-t_0)\nu}\tilde O(n^{-\alpha} e^{\lambda t_0}) 
   + \frac{t-t_0}n\\
   &=\E[Y_i(t_0) \mid \tilde \calE_0] + \sum_{A, B, j} \delta_{ij}^{AB} \int_{0}^{t-t_0} ds \, \left(e^{s\tilde M_0}\mathbf e_H\right)_{ji}^{AB} O(e^{\lambda t_0}) 
       + \tilde O(n^{-\alpha} e^{\lambda t_0}) 
  \end{aligned}
\end{equation}
where in the last step we used that $\nu<0$ to bound the second integral by a constant.

Let $\mathbf v_1$ and $\mathbf v_2$ be the eigenvectors associated with the eigenvalues in the upper left-hand block of the restricted matrix \eqref{M0-restricted}, 
corresponding to eigenvalues $\tilde \lambda$ and $-\gamma$. Since these two eigenvectors are distinct, they span same space as the vectors $\mathbf e_{11}^{TT}$ and $\mathbf e_{12}^{TT}$, showing that $\mathbf e_H$ can be expressed as $\mathbf e_H=a_1 \mathbf v_1 + a_2 \mathbf v_2$ for appropriately chosen constants $a_1, a_2$. Thus
\[
    (e^{s \tilde M_0}\mathbf e_H)_{ji}^{AB} =(e^{s \tilde M_0} (a_1 \mathbf v_1 + a_2 \mathbf v_2 ))_{ji}^{AB} = a_1 e^{\tilde \lambda t}\left(\mathbf v_1\right)_{ji}^{AB} + a_2 e^{-\gamma s}\left(\mathbf v_2\right)_{ji}^{AB}
  =O(e^{s\nu})\delta^{AH}\delta^{BH}\delta_{j1}
\]
where in the last step we used that the $\nu$ is the largest eigenvector of $\tilde M_0$ and thus an upper bound for both $\tilde \lambda$ and $-\gamma$, combined with the fact that both 
$\left(\mathbf v_1\right)_{ji}^{AB}$ and $\left(\mathbf v_2\right)_{ji}^{AB}$ are zero unless $A=B=H$ and $j=1$. Inserting this back into \eqref{eq:Y_sd_integral_ub} and using that $\nu<0$ to bound the integral by a constant yields
\begin{align*}
    \E[Y_i(t) \mid \tilde \calE_0] &\leq \E[Y_i(t_0) \mid \tilde \calE_0] + \delta_{i1} O(e^{\lambda t_0}) + \tilde O(n^{-\alpha}e^{\lambda t_0}), 
\end{align*}
which in particular implies that
\[
    \E[Y_2(t) \mid \tilde \calE_0] = \tilde O(n^{-\alpha}e^{\lambda t_0})=o(n^{\delta -\alpha}e^{\lambda t_0} )
\]
for all $t\in [t_0,T]$.
Using that w.h.p., the pandemic dies out in time $O(\ln n)$ and hence $R_2(\infty)=Y_2(T)$, Markov's inequality then implies that 
\begin{align*}
    \Pr(R_2(\infty)> n^{\delta -\alpha}e^{\lambda t_0}  \mid \tilde \calE_0) =o(1).
\end{align*}
Recalling our definition of $t_0$, this shows that, conditioned on $\tilde \calE_0$ w.h.p. $R_2(\infty) \leq n^{1-\alpha+2\delta}$. 
Finally, using the fact that $\Pr(\tilde \calE_0 \mid \calE_0) \to 1$ we conclude that the result also holds if we condition on $\calE_0$ instead of $\tilde \calE_0$. Since $\delta$ was arbitrary, this proves the desired upper bound on $R_2(\infty$).

\section{Simulations}
In this section, we perform additional robustness checks for our theorems by simulating SIR dynamics on realistic, two community contact networks with various degree distributions. For all simulations done in this section, the disease dynamics are simulated as follows. Given the contact network, an initially infected node is seeded uniformly at random. Nodes travel at rate $\rho_\Travel = n^{-1/2}$ and return home at rate $\rho_\Home = 1$. When edges are active, an infection event happens down each edge at rate $\beta$, and nodes recover according to a $\mathrm{Gamma}(\alpha, \theta)$ distribution parameterized by the shape ($\alpha$) and scale ($\theta$), as is known to be more realistic \cite{wearing2005appropriate, lloyd2001realistic}.  We choose $\alpha = 3.5$ and $\theta = 0.2014$, roughly based on estimates in the existing literature for COVID-19 and a mean recovery time of $5$ days (the timescale here is in weeks) \cite{sender2022unmitigated, kremer2022serial}. In addition, we choose $\beta = 0.1$, keeping all parameters consistent with the simulations in the main text; note that this may change the resulting $R_0$ for the different contact networks depending on the graph structure. Any interventions are implemented once $\epsilon = 1\%$ of the population are infected. As in the main text, we assume that intra-community interventions reduce the transmission rate $\beta$ by a factor of $5$, but any factor which reduces the effective $R_0$ to below $1$ suffices.  

The simulations here together with the ones in the main text demonstrate that our results are robust under misspecification of both the network and the parameters of disease transmission. 

\subsection{Copenhagen Network Study Dataset}

We construct a weekly contact network from the Copenhagen Networks Study (CNS) Bluetooth proximity data \cite{sapiezynski2019interaction, sapiezynski2019data}. The dataset records Bluetooth scans between study participants with associated Received Signal Strength Indicator (RSSI) values and timestamps. We pre-process the data by converting the timestamps to study days and excluding scan records where one participant ID was invalid. An edge in the network between two participants exists if they were detected via Bluetooth on the same study day. No RSSI threshold was applied. 

To construct the weekly network, we aggregate all scan events in the second week (measured relative to study days), creating an undirected edge list where each edge represents at least one proximity detection that day. We remove edges representing pairs of participants that only had one scan detection in that week to reduce noise from chance encounters. We use standard optimization methods (i.e. Nelder-Mead) for maximum likelihood estimation to fit a log-normal distribution to the degree distribution of the resulting contact network. The resulting distribution gives a log-normal distribution with parameters $\mu = 2.453$ and $\sigma = 1.040$ (see Figure \ref{fig:copenhagen-deg-dist}).

We use this log-normal degree distribution to construct a two-community graph, with each community having $n = 10^7$ nodes. For each node, we use the log-normal distribution specified above to sample its degree into its own community and into the other community (i.e. its home degree and travel degree). Since node degrees must be non-negative integers, we take the floor of each sampled value. We then use a configuration model to generate the network by matching degree stubs and removing duplicate edges after the pairing completes. Finally, we implement a stochastic simulation of the travel dynamics and epidemic dynamics. We seed an initially infected node uniformly at random. Our simulations based on the Copenhagen Network Study Dataset are shown in Figure \ref{fig:sim}.

\subsection{Degree Assortativity}
\label{sec:deg-assort}
We construct a contact network with two communities each on $n = 10^7$ vertices that has degree assortativity. We first label vertices with their intrinsic type (1 or 2). Then, for each vertex, we sample log-normal degree distribution with parameters $\mu = 1.4$ and $\sigma = 1.27$ following the degree distribution in \cite{hill2021network}. After the degrees are sampled, we further classify vertices into classes: high $(H)$, medium $(M)$, or low $(L)$ degree. The nodes with top 15\% degree are classified as high degree, the next 35\% are classified as medium degree, and the final 50\% are classified as low degree. Finally, degree stubs are matched in three stages. (1) They are assigned a type preference: with probability $0.8$ they are matched to stubs with their same intrinsic type and with probability $0.2$ they are matched to stubs with the opposite type. (2) They are assigned a degree-class preference: degree-class pairing is performed based on an assortativity matrix $Q$, where $Q_{ab}$ is the probability that a stub from a vertex of degree class $a$ is matched to a stub from a vertex of degree class $b$. The parameters of $Q$ are chosen so that the network demonstrates assortativity, meaning that stubs prefer to match to other stubs within their degree class, and so that the detailed balance equation $\E[N_{ab}] = \E[N_{ba}]$ is satisfied, where $N_{ab}$ is the number of stubs of type $a$ matched to stubs of type $b$.\footnote{In order to achieve detailed balance computationally, the initially chosen $Q$ matrix is perturbed slightly, conditioned on the realized degree distribution, using Metropolis-Hastings.} For our graph, we use
\[
    Q = \begin{pmatrix}
    0.939 &  0.0516 & 0.009 \\
    0.1 & 0.865 & 0.035 \\
    0.05 & 0.1 & 0.85
    \end{pmatrix}.
\]
(3) Remaining stubs are matched uniformly at random. The assortativity coefficient of the resulting graph is around $0.11$, showing mild assortativity consistent with some real-world networks \cite{newman2002assortative}. Finally, we implement a stochastic simulation of the travel dynamics and epidemic dynamics. Our simulations on the degree assortative network are shown in Figure \ref{fig:sim-assortative-v2}.

\subsection{Degree Disassortativity}
We construct a contact network with two communities each on $n = 10^7$ vertices that has mild degree disassortativity in a similar fashion as the degree assortative network constructed in Section \ref{sec:deg-assort}. We first label vertices with their intrinsic type (1 or 2). Then, for each vertex, we sample log-normal degree distribution with parameters $\mu = 1.4$ and $\sigma = 1.27$ following \cite{hill2021network}. After the degrees are sampled, we further classify vertices into classes: high $(H)$, medium $(M)$, or low $(L)$ degree. The nodes with top 5\% degree are classified as high degree, the next 15\% are classified as medium degree, and the final 80\% are classified as low degree. Degree stubs are matched in the same way as in Section \ref{sec:deg-assort}. In this case, the $Q$ matrix is chosen so that stubs are more likely to pair with stubs from vertices of a \textit{different} degree class than their own. For our graph, we use 
\[
    Q = \begin{pmatrix}
        0.056 & 0.08 & 0.864 \\
        0.091 & 0.822 & 0.087 \\
        0.9 & 0.08 & 0.02
    \end{pmatrix}.
\]
The assortativity coefficient of the resulting graph is around $-0.15$, showing mild disassortativity \cite{newman2002assortative}. Finally, we implement a stochastic simulation of the travel dynamics and epidemic dynamics. Our simulations on the degree disassortative network are showing in Figure \ref{fig:sim-disassortative}.

\newpage

\begin{figure}[tbhp]
    \centering
    \includegraphics[width=\linewidth]{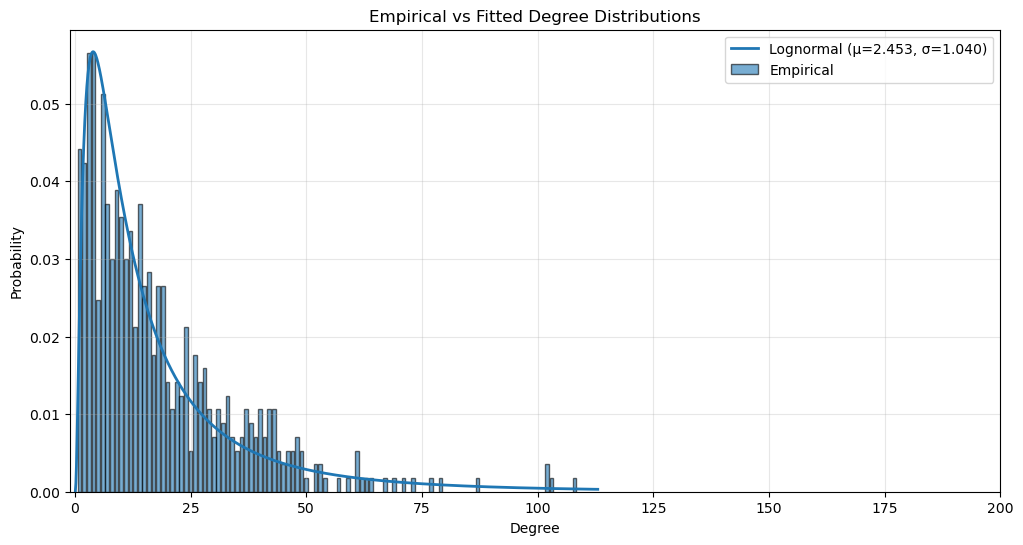}
    \caption{Empirical vs. fitted log-normal distribution $(\mu = 2.453, \sigma = 1.040)$ for Copenhagen Network Study contact network bluetooth dataset constructed from contacts during week 2 of the study. }
    \label{fig:copenhagen-deg-dist}
\end{figure}

\begin{figure}[tbhp]
    \centering
    \begin{subfigure}[b]{0.5\textwidth}
        \centering
        \includegraphics[width=.9\linewidth]{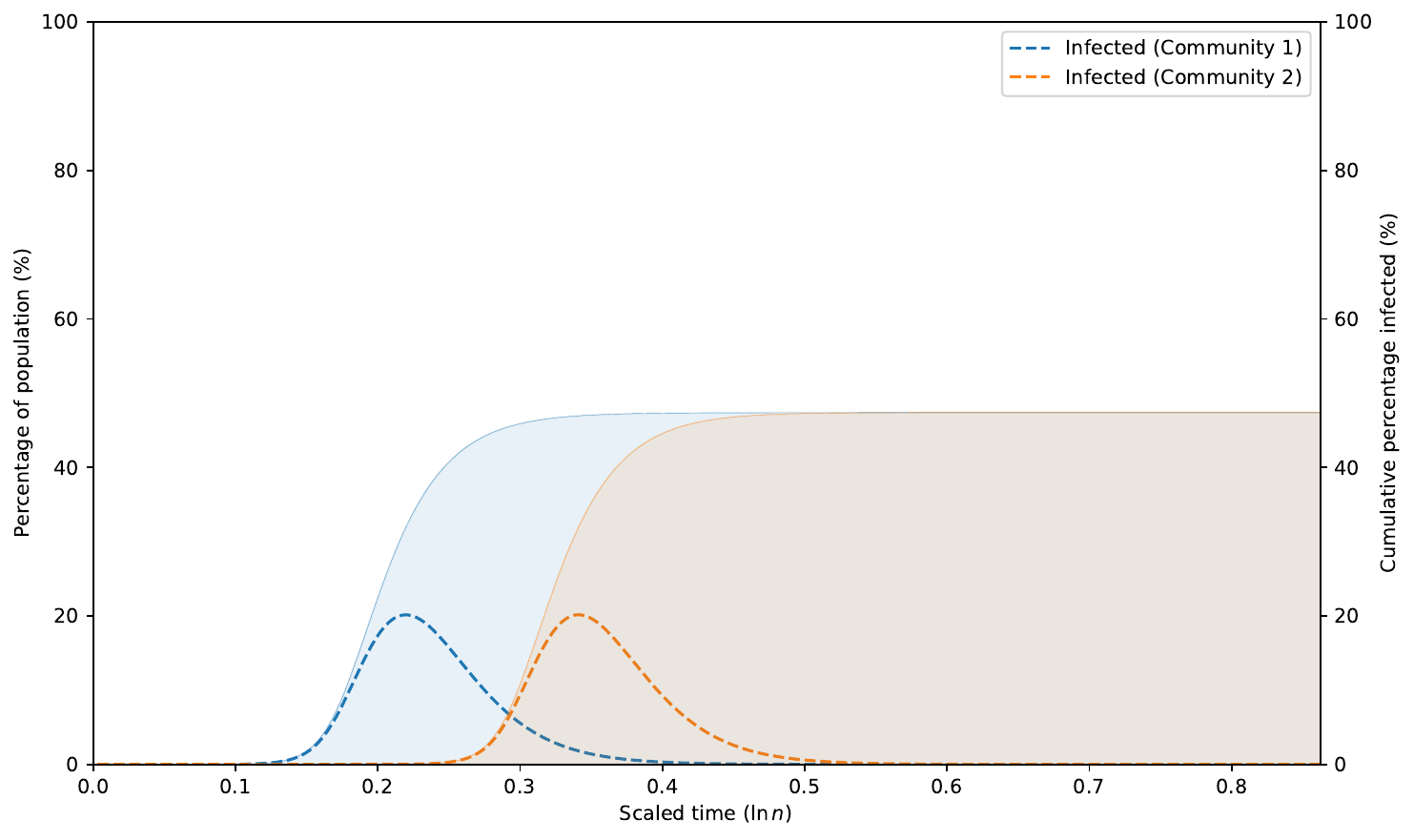}
        \caption{No intervention.}
    \end{subfigure}
    \begin{subfigure}[b]{0.5\textwidth}
        \centering
        \includegraphics[width=.9\linewidth]{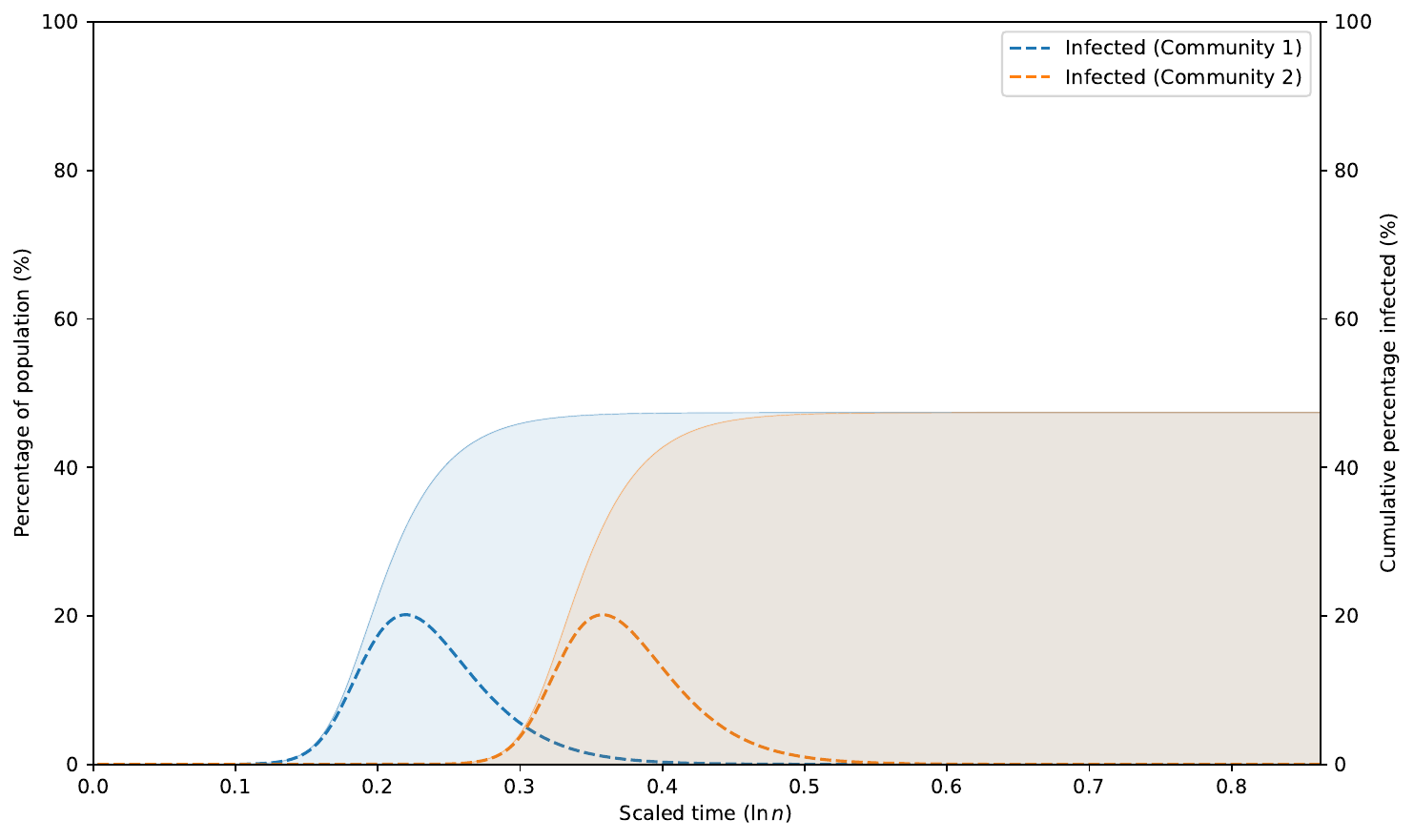}
        \caption{Travel ban implemented at $\tau_1(1\%).$}
    \end{subfigure}
    \begin{subfigure}[b]{0.5\textwidth}
        \centering
        \includegraphics[width=.9\linewidth]{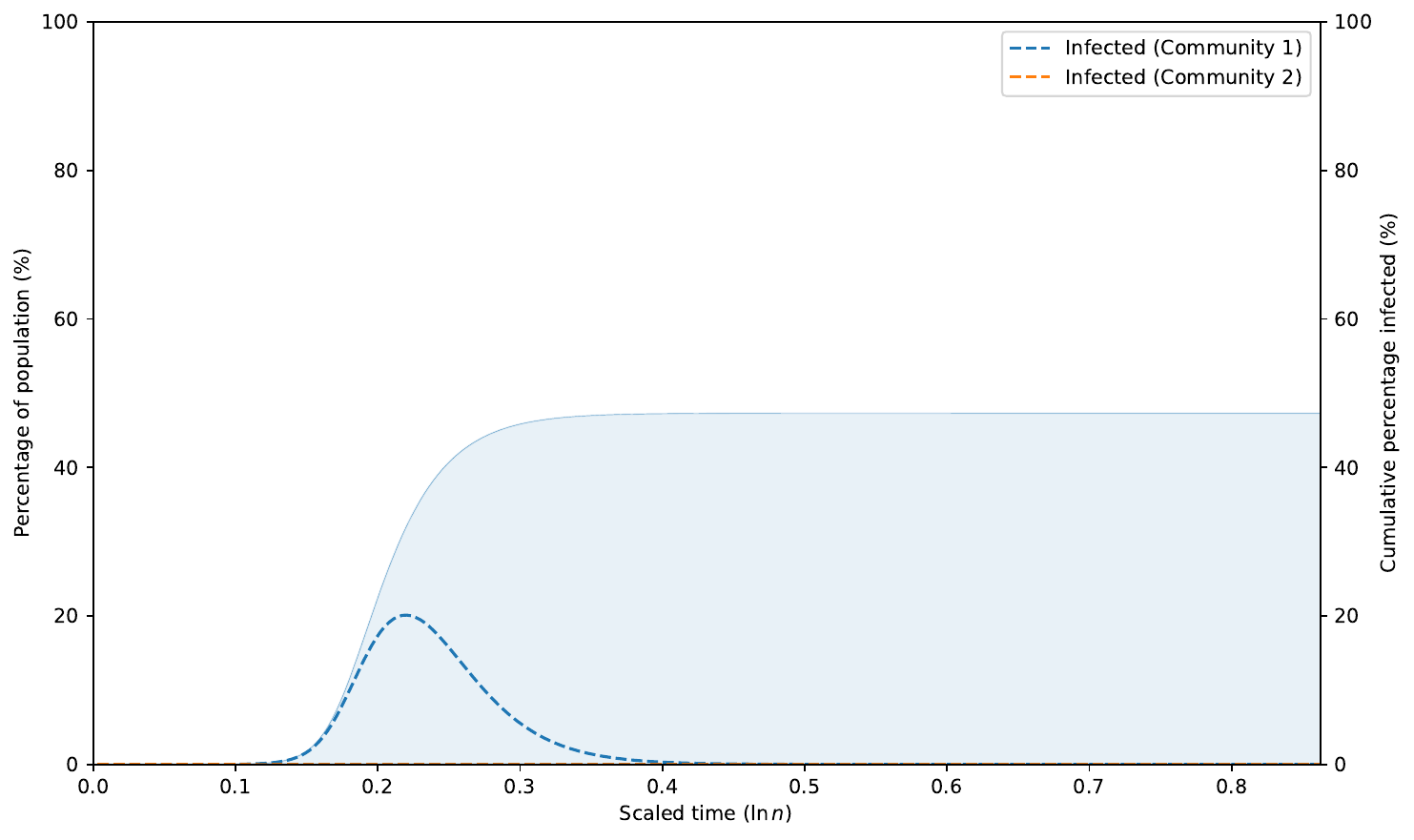}
        \caption{Intra-community intervention in community 2 implemented at $\tau_1(1\%).$}
    \end{subfigure}
    \caption{Simulation on two communities of $n=10^7$ individuals each, with an underlying network generated from a configuration model with lognormal distributed degrees ($\mu = 2.45, \sigma = 1.04$) extrapolated from the Copenhagen Network Study (CNS) Bluetooth network data \cite{sapiezynski2019data}. The dashed infection curves represent the proportion of infected individuals at any point in time (left axis), and the solid curves represent the cumulative percentage of infections (right axis). The infection rates are exponentially distributed with parameter $\beta = 0.1$. The recovery rates are gamma distributed with shape parameter $3.5$ and scale parameter $0.2041$. Other parameters are: $\rho_\Travel = n^{-1/2}$, and $\rho_\Home = 1$.}
    \label{fig:sim}
\end{figure}

\begin{figure}[tbhp]
    \centering
    \begin{subfigure}[b]{0.5\textwidth}
        \centering
        \includegraphics[width=.9\linewidth]{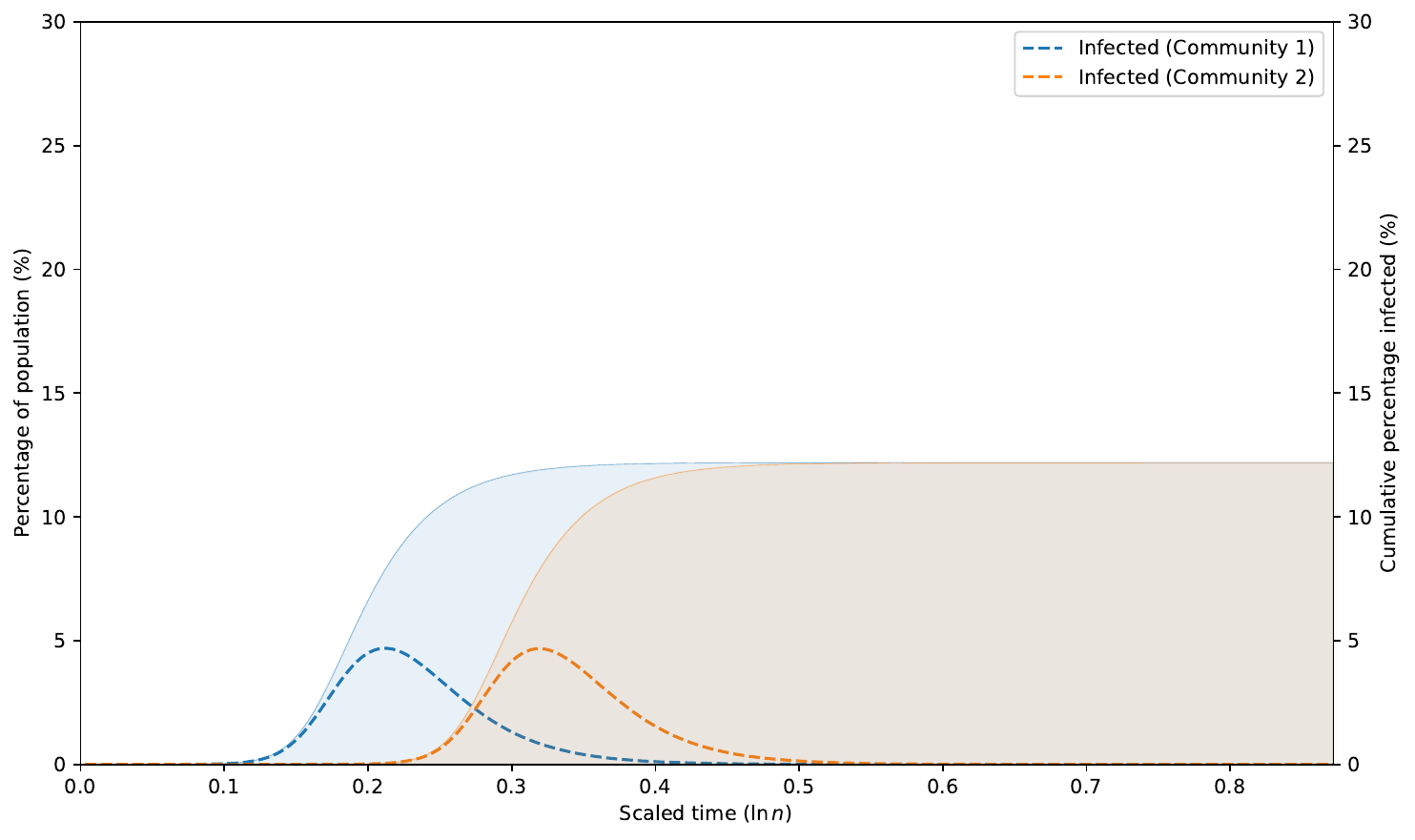}
        \caption{No intervention.}
    \end{subfigure}
    \begin{subfigure}[b]{0.5\textwidth}
        \centering
        \includegraphics[width=.9\linewidth]{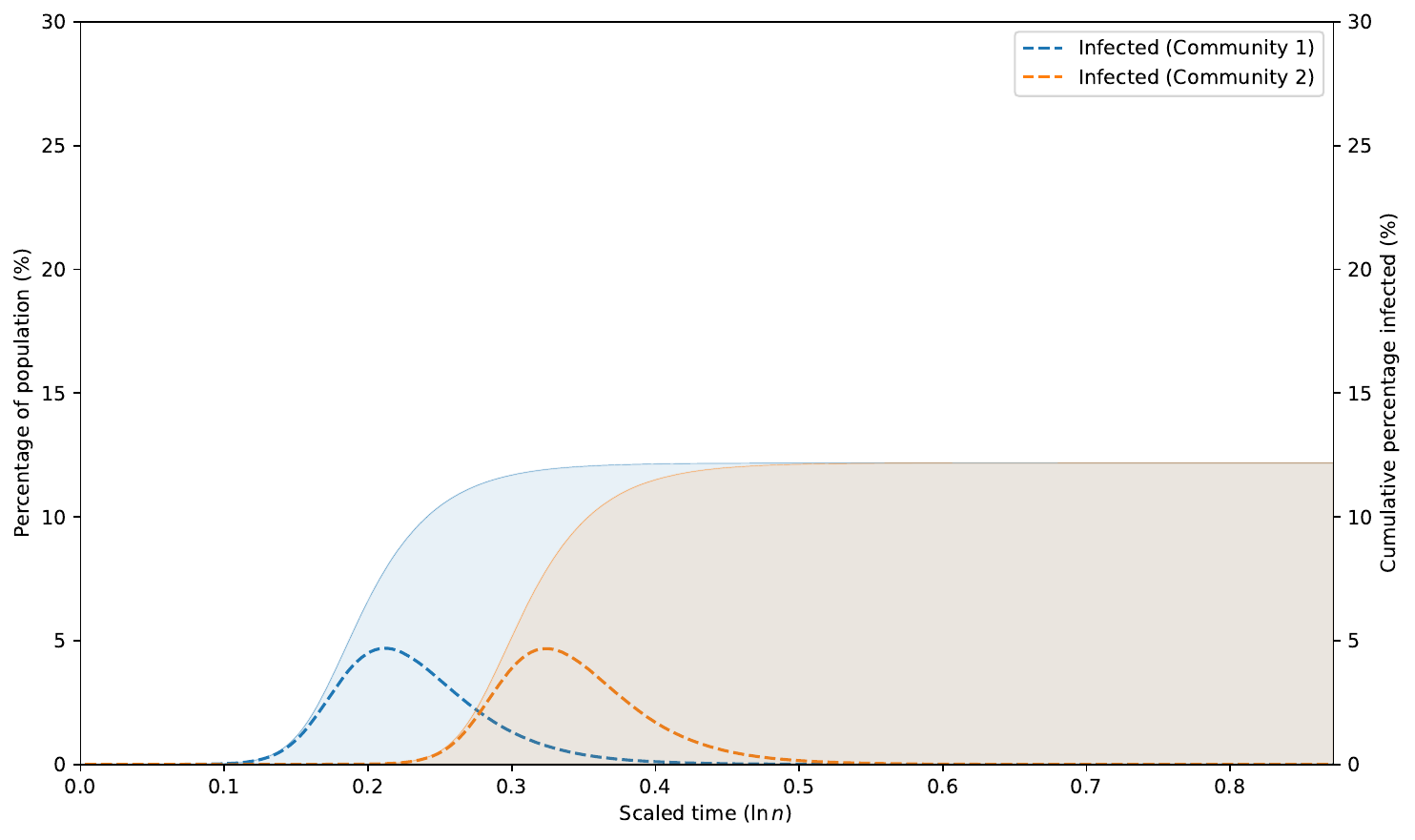}
        \caption{Travel ban implemented at $\tau_1(1\%).$}
    \end{subfigure}
    \begin{subfigure}[b]{0.5\textwidth}
        \centering
        \includegraphics[width=.9\linewidth]{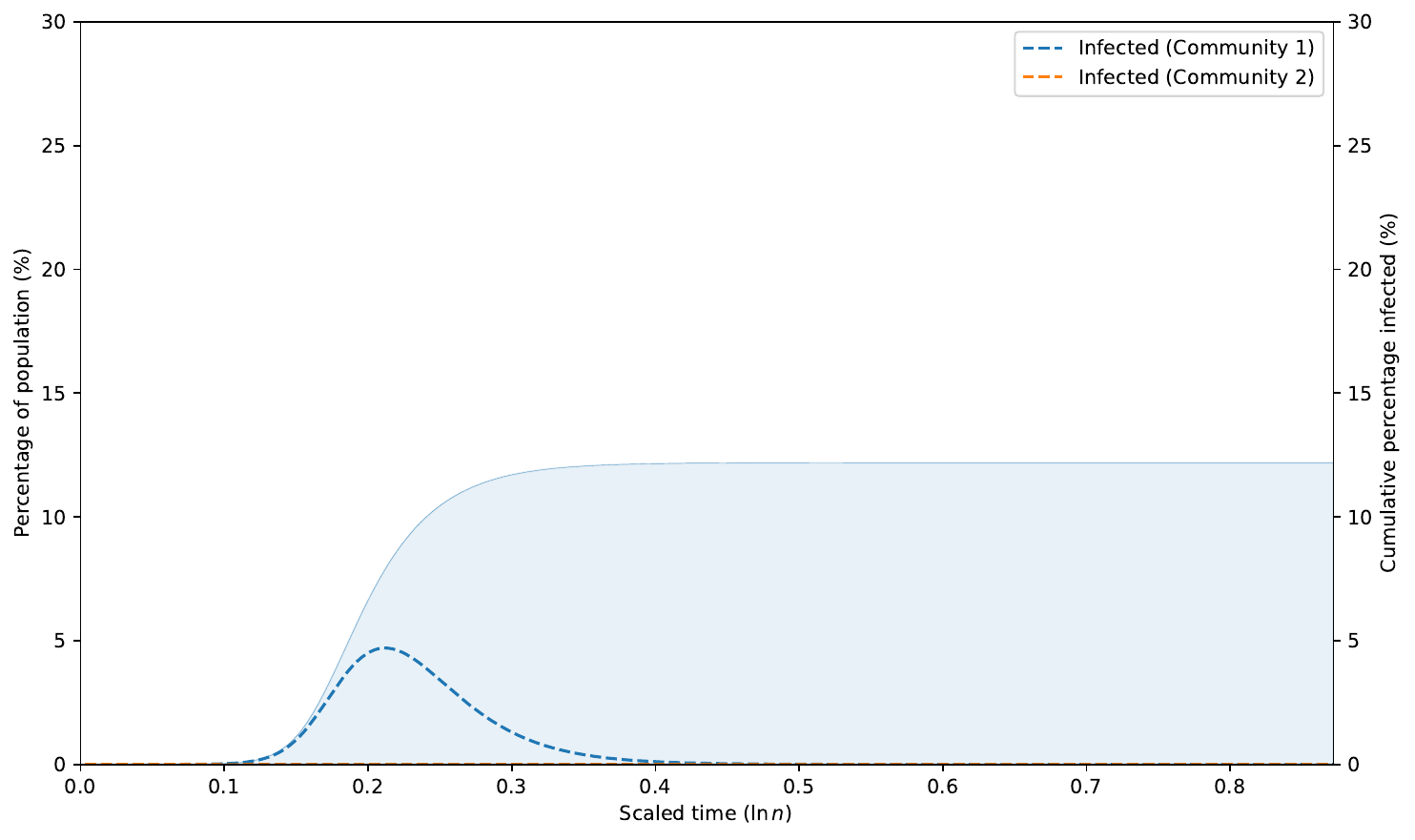}
        \caption{Intra-community intervention in community 2 implemented at $\tau_1(1\%).$}
    \end{subfigure}
    \caption{Simulation on two communities of $n=10^7$ individuals each, with an underlying degree assortative network. The dashed infection curves represent the proportion of infected individuals at any point in time (left axis), and the solid curves represent the cumulative percentage of infections (right axis). The infection rates are exponentially distributed with parameter $\beta = 0.1$. The recovery rates are gamma distributed with shape parameter $3.5$ and scale parameter $0.2041$. Other parameters are: $\rho_\Travel = n^{-1/2}$, and $\rho_\Home = 1$.}
    \label{fig:sim-assortative-v2}
\end{figure}

\newpage

\begin{figure}[tbhp]
    \centering
    \begin{subfigure}[b]{0.5\textwidth}
        \centering
        \includegraphics[width=.9\linewidth]{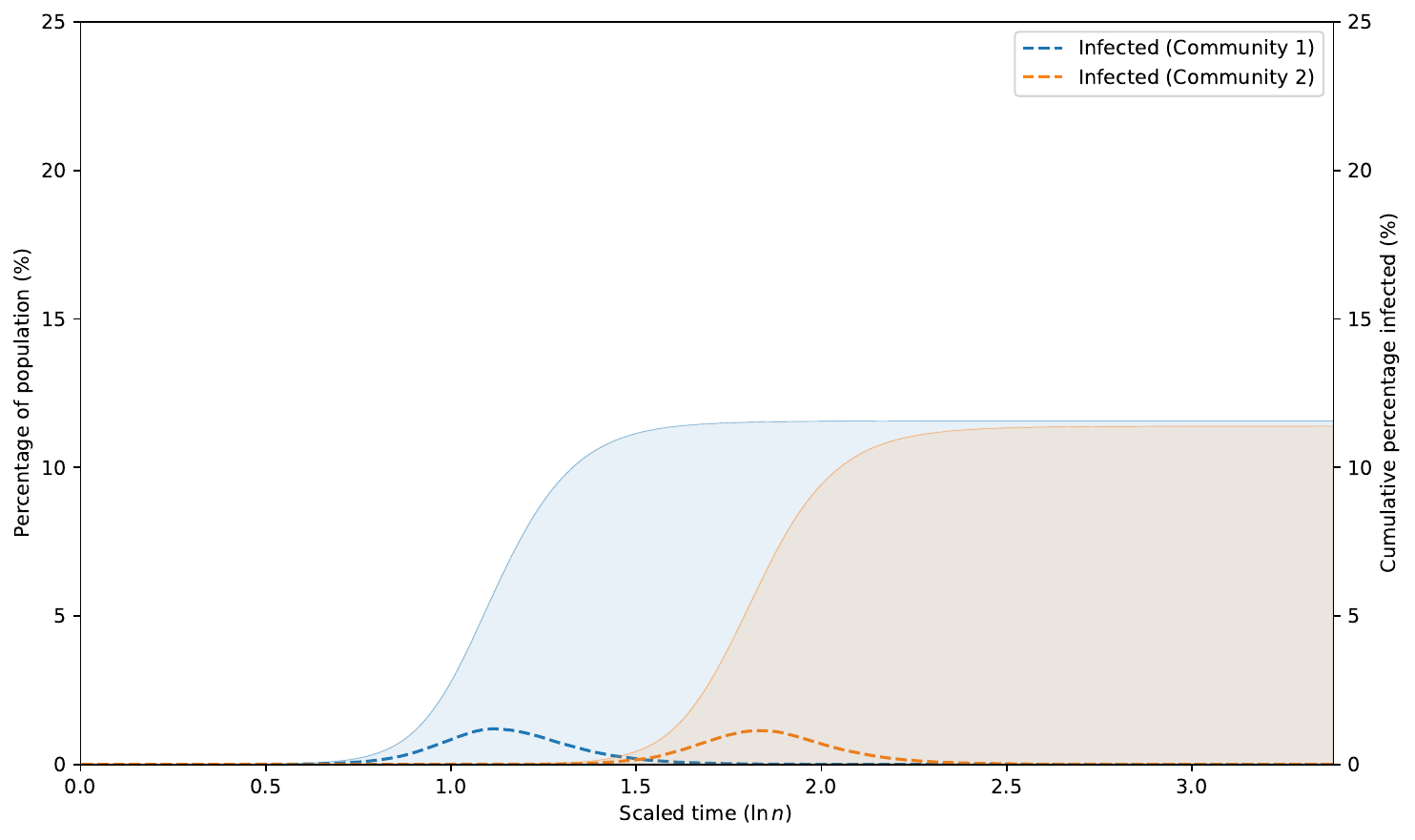}
        \caption{No intervention.}
    \end{subfigure}
    \begin{subfigure}[b]{0.5\textwidth}
        \centering
        \includegraphics[width=.9\linewidth]{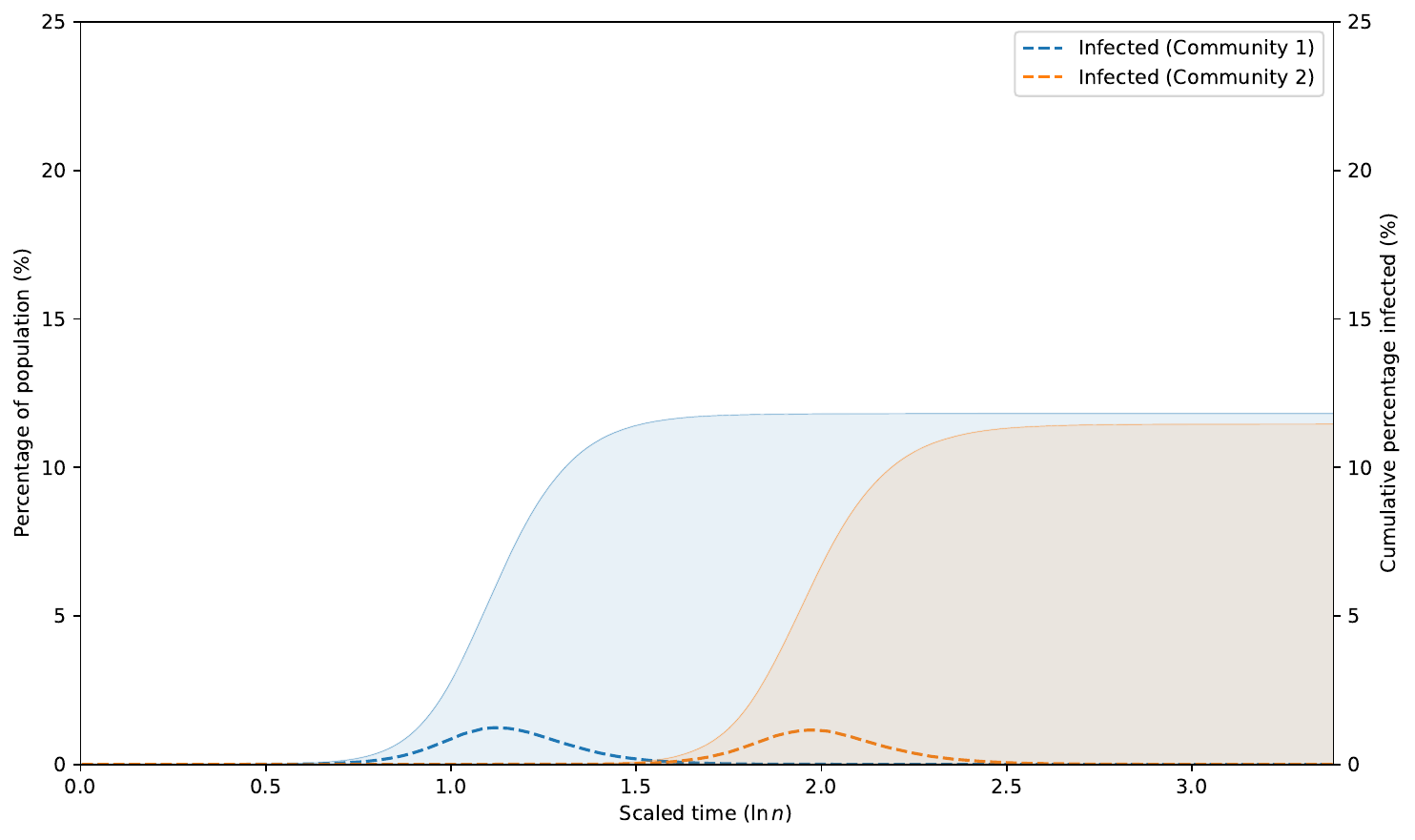}
        \caption{Travel ban implemented at $\tau_1(1\%).$}
    \end{subfigure}
    \begin{subfigure}[b]{0.5\textwidth}
        \centering
        \includegraphics[width=.9\linewidth]{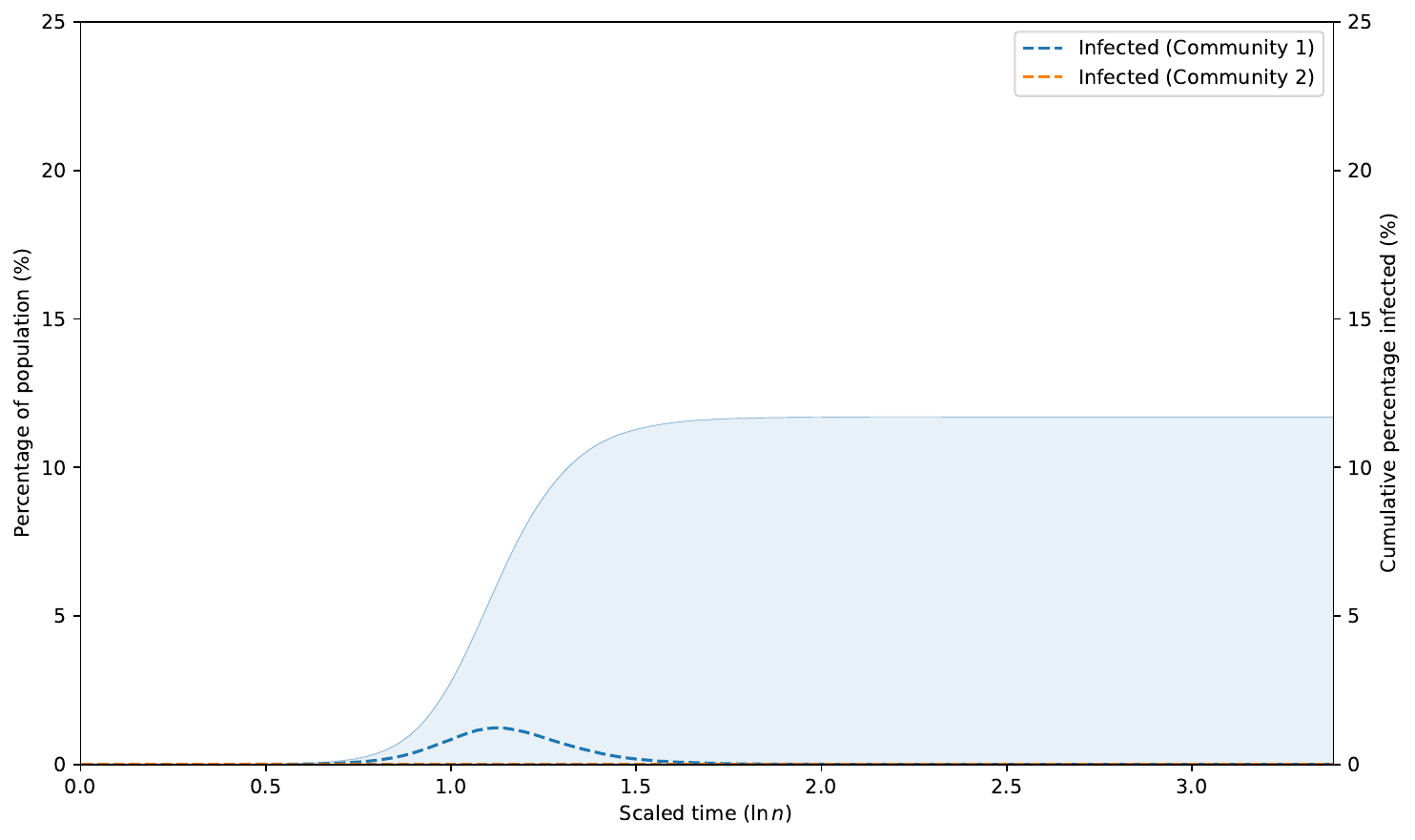}
        \caption{Intra-community intervention in community 2 implemented at $\tau_1(1\%).$}
    \end{subfigure}
    \caption{Simulation on two communities of $n=10^7$ individuals each, with an underlying degree disassortative network. The dashed infection curves represent the proportion of infected individuals at any point in time (left axis), and the solid curves represent the cumulative percentage of infections (right axis). The infection rates are exponentially distributed with parameter $\beta = 0.1$. The recovery rates are gamma distributed with shape parameter $3.5$ and scale parameter $0.2041$. Other parameters are: $\rho_\Travel = n^{-1/2}$, and $\rho_\Home = 1$.}
    \label{fig:sim-disassortative}
\end{figure}

\newpage

\end{document}